\documentclass[lettersize,journal]{article}
\usepackage{arxiv}
\usepackage{fullpage}

\makeatletter
\let\cl@chapter\undefined
\makeatletter

	\usepackage{amsfonts, amsmath, amsthm}
	\usepackage{mathtools}
	\newtheorem{theorem}{Theorem}[section]
	\newtheorem{lemma}[theorem]{Lemma}
	\newtheorem{remark}{Remark}
	
	\usepackage{graphicx}
	\usepackage{subcaption}
	\usepackage{xcolor}

	\usepackage{booktabs}
	\usepackage{multirow}

	\usepackage{algorithm}      
	\usepackage[noend]{algpseudocode}

	\usepackage[hypertexnames=false]{hyperref} 
	\usepackage{cleveref} 

	\definecolor{matlab1}{RGB}{0,    114,  189}
	\definecolor{matlab2}{RGB}{217,   83,   25}
	
	\newcommand{\norm}[1]{\left\lVert#1\right\rVert}
	\DeclareMathOperator*{\argmin}{arg\,min} 

	\title{Variable Projected Augmented Lagrangian Methods for Generalized Lasso Problems}
	
	\author{Stefano Aleotti
		\thanks{University of Insubria, Department of Science and High Technology, Varese, Italy (\email{stefano.aleotti@uninsubria.it}).}
	\and Davide Bianchi
		\thanks{Sun Yat-sen University, School of Mathematics (Zhuhai),  Zhuhai, China (\email{bianchid@mail.sysu.edu.cn}).}
	\and Florian Bo\ss mann
		\thanks{Harbin Institute of Technology, School of Mathematics, Harbin,  China (\email{f.bossmann@hit.edu.cn}).} 
	\and Riley Yizhou Chen
		\thanks{Emory University, Department of Mathematics, Atlanta, USA (\email{yizhou.chen@emory.edu}).}
	\and Matthias Chung
		\thanks{Emory University, Department of Mathematics, Atlanta, USA (\email{matthias.chung@emory.edu}).}}  
	
	\date{\today}  

	\newcommand{\email}[1]{\href{mailto:#1}{\texttt{#1}}}
	
\begin{document}

\maketitle

\begin{abstract}
We introduce \emph{variable projected augmented Lagrangian} (VPAL) methods for solving generalized nonlinear Lasso problems with improved speed and accuracy. By eliminating the nonsmooth variable via soft-thresholding, VPAL transforms the problem into a smooth reduced formulation. We develop a preconditioned variant based on curvature-informed updates and investigate its performance for both linear and nonlinear inverse problems. We prove convergence guarantees for both standard and preconditioned VPAL under mild assumptions and show that variable projection leads to faster convergence and higher solution quality. The method seamlessly extends to nonlinear inverse problems, where it outperforms traditional approaches in applications such as phase retrieval and contrast-enhanced magnetic resonance imaging. 
Across linear and nonlinear imaging applications, numerical experiments demonstrate the accuracy and computational efficiency of the proposed methods, highlighting the potential of variable projection for large-scale inverse problems.
\end{abstract}



\section{Introduction}

Nonlinear optimization problems
\begin{equation} \label{eq:nonlinear_optimization}
    \min_{x,y} \ f(x,y)
\end{equation}
characterized by an objective function $f$ that can be expressed in terms of a two distinct sets of variables $x$ and $y$ are prevalent in many applications. These problems frequently appear in contexts where one set of variables can be optimized separately from another, enabling computationally efficient strategies in various fields, including signal processing, machine learning, and medical imaging~\cite{nocedal2006numerical,beck2009fast,tibshirani1996regression,wright2015coordinate}.

Depending on the properties of $f$, various numerical optimization strategies handling \Cref{eq:nonlinear_optimization} exist. A basic approach is to treat all variables jointly and apply standard optimization algorithms~\cite{nocedal2006numerical}. Another widely used strategy is \emph{alternating optimization}, which iteratively optimizes one set of variables while keeping the other fixed. A prominent instance of this is the \emph{Alternating Direction Method of Multipliers} (ADMM), which alternates between subproblems while introducing dual updates to enforce consistency \cite{boyd2011distributed}. Another approach is \emph{coordinate descent}, which sequentially updates individual variables, often yielding efficient steps when variables are weakly coupled~\cite{wright2015coordinate}. However, both alternating optimization and coordinate descent may suffer from slow convergence, particularly when the variables are strongly coupled. Notably, \emph{variable projection} offers an alternative by eliminating one set of variables analytically, thereby reducing the optimization to a lower-dimensional and often better-conditioned problem.

Optimization methods leveraging \emph{variable projection} have often been shown to provide substantial computational advantages for solving \Cref{eq:nonlinear_optimization}, yet their application remains relatively underexploited in a broader context. The seminal paper by Golub and Pereyra \cite{golub1973differentiation} introduced the variable projection approach for nonlinear least squares problems of the form
$$\min_{x,y} \quad \|C(y)x - b\|^2_2,$$
where the variables $x$ and $y$ can be separated into linear and nonlinear components, $C\colon \mathbb{R}^\ell \to \mathbb{R}^{m\times n}$, with  a linear mapping $C(y)\colon \mathbb{R}^n\to \mathbb{R}^{m}$ and $b \in \mathbb{R}^m$. This method allows for the analytical elimination of one set of variables, transforming the problem into a lower-dimensional space, which may significantly reduce computational complexity as elaborated in \cite{oleary_variable_2013,golub2003separable}.

In this work, we specifically focus on a generalized nonlinear Lasso optimization scenario where the objective function $f$ does not exhibit direct separability into two distinct sets of variables. This generalized nonlinear Lasso problem is defined by the optimization problem
\begin{equation}\label{eq:genLasso}
    \min_x \quad f(x) = \tfrac{1}{2\sigma^2} \left\| A(x) - b \right\|_2^2 + \mu \left\| Dx \right\|_1,
\end{equation}
where $\mu>0 $ is a regularization parameter and $A\colon \mathbb{R}^{n} \to \mathbb{R}^m$. We assume that observations are given by $b = A(x) + \epsilon$, where $\epsilon$ is some additive noise, e.g., $\epsilon \sim \mathcal{N}(0,\sigma^2 I)$. Here, $D \in \mathbb{R}^{\ell \times n}$ is a linear operator; for the standard Lasso approach we have $D = I$. Various other choices for $D$ exist, most prominently discretizations of the total variation or Laplace operator \cite{chung2023variable,tibshirani2011solution}, or structure-preserving differential operators~\cite{gilboa2009nonlocal,bianchi2025data,aleotti2026iterated}.

The generalized Lasso is inherently connected to compressed sensing, which aims to recover sparse signals from far fewer measurements than those required by traditional Nyquist–Shannon sampling criteria \cite{donoho2006compressed,candes2008introduction}. This connection stems from the generalized Lasso's inherent enforcement of sparsity via regularization. This closely aligns it with the foundational principles of compressed sensing and facilitates robust signal recovery from undersampled data. Thus, generalized Lasso methods are widely applicable in diverse domains that prioritize sparsity, such as medical imaging, remote sensing, machine learning, statistical data analysis, dictionary learning, and sparse representations in signal processing \cite{beck2009fast,candes2006robust,tibshirani2011solution,mairal2010online}.

\paragraph{Contributions} This work addresses a significant class of generalized nonlinear Lasso optimization problems. A primary contribution lies in extending the variable projection framework to nonlinear settings and introducing an advanced version of the variable projected augmented Lagrangian (\texttt{vpal}) method. Originally proposed in its basic form in 2023~\cite{chung2023variable}, the \texttt{vpal} framework provided a foundation for incorporating variable projection into linear generalized Lasso problems. However, the initial approach was limited by slow convergence due to gradient-based updates. In this work, we present a preconditioning strategy for linear models that significantly accelerates convergence by approximating the Hessian of the reduced problem. This enhancement is particularly effective, as demonstrated in numerical experiments on image deblurring and inpainting.

Furthermore, building on the theoretical convergence results for \texttt{vpal} for nonlinear generalized Lasso problems established in~\cite{solomon2025fast}, we introduce for the first time a nonlinear \texttt{vpal} algorithm. We validate its effectiveness and convergence behavior through real-world applications in phase retrieval from short-time Fourier transform measurements and in low-dose-to-high-dose medical image reconstruction, showcasing the method's practical relevance and robustness across complex nonlinear inverse problems.

\paragraph{Structure} The remainder of the paper is organized as follows. In \Cref{sec:background}, we provide background on variable projection and review related optimization strategies for generalized Lasso problems. \Cref{sec:nonlinear_vpal} introduces the variable projected augmented Lagrangian (\texttt{vpal}) method for nonlinear optimization problems, including the preconditioned variant. \Cref{sec:linear_numEx,sec:nonlinear_numEx} present numerical experiments, demonstrating the performance of the proposed methods on various imaging problems, including deblurring (\Cref{ssec:deblurring}), inpainting (\Cref{ssec:inpainting}), phase retrieval (\Cref{sec:ptycho}), and  the low-dose-to-high-dose medical image application (\Cref{sec:LIPCAR}). Finally, \Cref{sec:conclusion} summarizes the findings and outlines directions for future research.

\section{Background on Variable Projected Augmented Lagrangian (\texttt{vpal})}\label{sec:background}
We begin the derivation of the \texttt{vpal} algorithm by separating the smooth residual term from the nonsmooth $\ell_1$ regularization term. Using a splitting approach, the surrogate variable $y \in \mathbb{R}^\ell$ is introduced, and \Cref{eq:genLasso} can be written as a constrained optimization problem: 
\begin{equation} \label{eq:lag}   
    \min_{x,y} \quad f(x,y) = \tfrac{1}{2\sigma^2} \left\| A(x) - b \right\|_2^2 + \mu \left\|y  \right\|_1, \qquad \text{s.t.} \quad  Dx -y = 0.
\end{equation}
Letting $c \in \mathbb{R}^\ell$ be the vector of Lagrange multipliers, we obtain the corresponding Lagrange function of \Cref{eq:lag} as 
\begin{equation}    
    L(x, y, c) =  \tfrac{1}{2\sigma^2} \left\| A(x) - b \right\|_2^2 + c^\top (Dx -y) + \mu \left\|y  \right\|_1.
\end{equation}
The augmented Lagrangian framework \cite{nocedal2006numerical} then penalizes data misfit through an additive quadratic term $\left\|Dx -y\right\|_2^2$, leading to 
\begin{equation}    
    \tilde L_{\text{aug}}(x,y, c) =  \tfrac{1}{2\sigma^2} \left\| A(x) - b \right\|_2^2 + c^\top (Dx -y) + \mu \left\|y  \right\|_1 + \tfrac{\lambda^2}{2} \left\|Dx -y\right\|_2^2
\end{equation}
with penalty parameter $\lambda>0$. By combining the second and last additive terms and setting $z = c/\lambda^2$ we get
\begin{equation}\label{eq:aug_lagr}    
    \tilde L_{\text{aug}}(x,y,z) =  \tfrac{1}{2\sigma^2} \left\| A(x) - b \right\|_2^2 + \mu \left\|y  \right\|_1 + \tfrac{\lambda^2}{2} \left\|Dx -y + z\right\|_2^2 - \tfrac{\lambda^2}{2}\left\|z\right\|_2^2.
\end{equation}
In the ADMM splitting framework for minimizing the augmented Lagrangian $L_{\text{aug}}$, the optimization proceeds by alternating updates of the primal variables $x$ and $y$, interleaved with updates of the scaled dual variable $z$ whose update at iteration $k$ is $z_{k+1} = Dx_{k+1} - y_{k+1} + z_k$, see ~\cite{nocedal2006numerical}. 

A major advantage of the splitting approach is that at iteration $k$ the update of $y$, assuming $x$ and $z$ to be fixed, reduces to a soft-thresholding (or shrinkage) step, which can efficiently be computed by
\begin{equation} \label{eq:softTH}
    y_{k+1} =  \text{sign}(Dx_{k+1} + z_k) \odot \text{ReLU}(|Dx_{k+1} + z_k| - \mu/\lambda^2 e),
\end{equation}
where $e$ denotes the vector of ones, \text{ReLU} is the rectified linear unit, $|\,\cdot\,|$ denotes the element-wise absolute value, and $\odot$ is the Hadamard product. Note that determining the optimal $y$ given $x$ and $z$ in \Cref{eq:softTH} can be seen as a functional relationship
\begin{equation} \label{eq:softTH2}
    y_z(x) =  \text{sign}(Dx + z) \odot \text{ReLU}(|Dx + z| - \mu/\lambda^2 e),
\end{equation} where $y_z\colon\mathbb{R}^n\rightarrow\mathbb{R}^\ell$ is the continuous shrinkage map.
Hence, rewriting the objective function in \Cref{eq:aug_lagr} results in
\begin{equation}
    \tilde L_{\text{aug}}(x,z) =  \tfrac{1}{2\sigma^2} \left\| A(x) - b \right\|_2^2 + \mu \left\|y_z(x)  \right\|_1 + \tfrac{\lambda^2}{2} \left\|Dx -y_z(x) + z\right\|_2^2 - \tfrac{\lambda^2}{2}\left\|z\right\|_2^2.
\end{equation}
Unlike alternating optimization schemes, the \texttt{vpal} method updates $x$ in a descent step implicitly using $y_z(x)$ according to \Cref{eq:softTH2}. Here, a further property of the variable projection approach becomes advantageous: since $y$ is optimal with respect to $x$, any (sub)gradient contribution of $y$ to $\tilde L_{\text{aug}}$ vanishes. Hence, our iteration is 
\begin{align}
    x_{k+1} &= x_k + \alpha_k s(\tilde L_{\text{aug}}, x_k,z_k) \label{eq:xupdate} \\
    z_{k+1} &= Dx_{k+1} - y_{z_k}(x_{k+1}) + z_{k} 
    \label{eq:update_x_z}
\end{align}
where $s$ is an appropriate descent direction for $\tilde L_{\text{aug}}(x)$, $\alpha_k$ is the corresponding step size, and $y$ is given by \Cref{eq:softTH}.

The efficiency of the optimization method depends crucially on the update of
$x$ in \Cref{eq:xupdate}. The \texttt{vpal} algorithm employs a gradient
descent approach; however, other update strategies may be more efficient, such
as nonlinear conjugate gradient methods or preconditioning techniques
\cite{chung2023variable,solomon2025fast}. The selection of the step size
$\alpha_k$ is another important aspect. In the present implementation, this
step size is obtained by treating $y$ as fixed and computing an optimal
linearized step. In our numerical experiments, this strategy leads to good results. However, alternative step-size selection strategies could also be
developed through computationally efficient objective evaluations, at least in
the case of a linear forward operator $A$.

\section{\texttt{vpal} for nonlinear problems}\label{sec:nonlinear_vpal}
This section presents the update strategy of \texttt{vpal} for generalized nonlinear Lasso problems and introduces its preconditioned variant.

Note that $z$ and $\lambda$ in \Cref{eq:aug_lagr} are considered constant; thus, the problem can be reformulated as
\begin{equation}\label{eq:f_joint}
    \underset{x,y}{\min} \quad f_{\mathrm{joint}}(x,y) = \tfrac{1}{2\sigma^2}\|A(x)-b\|_2^2+\tfrac{\lambda^2}{2}\|Dx-y+z\|_2^2+\mu\|y\|_1.
\end{equation}
The projected function corresponding to $f_{\mathrm{joint}}$ is the map $f_{\mathrm{proj}}\colon\mathbb{R}^n\rightarrow\mathbb{R}$ defined by
\begin{equation}\label{eq:f_proj}
  f_{\mathrm{proj}}(x) = \tfrac{1}{2\sigma^2}\|A(x)-b\|_2^2+\tfrac{\lambda^2}{2}\|Dx-y_z(x)+z\|_2^2+\mu\|y_z(x)\|_1.
\end{equation}
For the update of $x$ shown in \Cref{eq:xupdate}, two parameters need to be specified: the step length $\alpha$ and the descent direction $s$. First, the steepest descent direction $s = -g$ is given by
\begin{equation}\label{eq:descent_direction}
    g = \nabla_x f_{\mathrm{proj}}(x).
\end{equation}
Let $J_A(x)$ and $J_{y_z}(x)$ denote the Jacobians of $A(x)$ and $y_z(x)$, respectively. Then
\begin{align*}
	\begin{split}
    \nabla_x f_{\mathrm{proj}}(x)
    =\tfrac{1}{\sigma^2}J_A(x)^\top (A(x)-b) + \lambda^2(D-J_{y_z}(x))^\top(Dx-y_z(x)+z) + \mu \nabla_x \left(\|y_z(x)\|_1\,\right).
    \end{split}
\end{align*}

The term $\mu \|y_z(x)\|_1$ is differentiable away from the coordinate axes. Its gradient is given by
$$
 \mu \nabla_x \|y_z(x)\|_1 = \mu J_{y_z}(x)^\top {\rm sign}(y_z(x))\odot e\,.
$$
Therefore,
\begin{align*}
  \nabla_x f_{\mathrm{proj}}(x) =&  \tfrac{1}{\sigma^2} J_A(x)^\top(A(x)-b) + \lambda^2 D^\top(Dx + z - y_z(x)) \\
  &+ J_{y_z}(x)^\top\!\left(\mu\, \operatorname{sign}(y_z(x))\odot e
  - \lambda^2(Dx - y_z(x) + z)\right)\,.
\end{align*}
This expression can be further simplified by realizing that the $y_z(x)$ represents the optimal $y$ for \Cref{eq:f_joint}, given $x$. To obtain the optimal $y$ in \Cref{eq:f_joint}, we compute $\nabla_y f_{\mathrm{joint}}(x,y)$. Therefore, in particular for $y = y_z$ we have
\begin{equation*}
    \mu\operatorname{ sign}(y_z(x))\odot e-\lambda^2(Dx-y_z(x)+z) = 0\,.
\end{equation*}
In this way, the last term of $\nabla_x f_{\mathrm{proj}}(x)$ vanishes, and we obtain
\begin{equation}
    g = \tfrac{1}{\sigma^2} J_A(x)^\top(A(x)-b) + \lambda^2D^\top(Dx + z- y_z(x)) = \nabla_x f_{\mathrm{joint}}(x,y)\,,
\end{equation}
where $y_z(x)$ is the optimal $y$ for $f_{\mathrm{joint}}(x,y)$. A detailed analysis of the nondifferentiability of the shrinkage operator can be found in \cite{solomon2025fast}. 

There are various ways to find the step size $\alpha$, and we focus on the exact and linearized optimal step size selection. In the first case we have
\begin{align}\label{eq:optimal_step_size}
\hat{\alpha} = & \argmin_{\alpha>0} f_{\mathrm{proj}}(x+\alpha s)\nonumber\\
= &\argmin_{\alpha>0} 
\tfrac{1}{2\sigma^2}\|A(x+\alpha s)-b\|_2^2+\tfrac{\lambda^2}{2}\|D\cdot(x+\alpha s)-y_z(x+\alpha s)+z\|_2^2 + \mu \|y_z(x+\alpha s)\|_1\,,
\end{align}
where the dot indicates a matrix-vector multiplication. There exist multiple iterative algorithms to solve this single-variable optimization problem, for instance, trust-region or quasi-Newton methods. If the shrinkage term $y_z(x)$ is treated as the linear term $y$, we obtain the linearized optimal step size selection
\begin{equation}\label{eq:alpha_lin}
    \hat{\alpha} =  \argmin_{\alpha>0}  \tfrac{1}{2\sigma^2}\|A(x+\alpha s)-b\|_2^2+\tfrac{\lambda^2}{2}\|D\cdot(x+\alpha s)-y+z\|_2^2.
\end{equation}
Then, $\hat{\alpha}$ is the solution of the following equation:
\begin{equation}
    \tfrac{1}{\sigma^2}s^\top J_A(x+\alpha s)^\top(A(x+\alpha s)-b) + \lambda^2s^\top D^\top(Dx+\alpha Ds-y+z)=0\,. 
\end{equation}
Note that if $A$ is linear, the linearized optimal step size has a closed form: 
$$
 \hat{\alpha} = \frac{s^\top(A^\top b-A^\top Ax-\sigma^2\lambda^2 D^\top D x + \sigma^2\lambda^2 D^\top(y-z))}{s^\top(A^\top A+\sigma^2\lambda^2 D^\top D)s}.
$$

Building on the standard \texttt{vpal} framework, we extend the method to \emph{nonlinear} forward models and introduce a \emph{curvature-aware search
direction} that naturally incorporates higher-order information from the projected problem.  Instead of performing a conventional gradient step along
$s$, the update follows an adjusted direction
\[
    s = -P^{-1}g,
\]
where the matrix $P$ acts as a nonlinear preconditioner capturing the local geometry of the reduced objective $f_{\mathrm{proj}}(x)$.  This modification effectively transforms \texttt{vpal} into a higher-order method: it leverages a Gauss–Newton–type approximation of the reduced Hessian to accelerate convergence without requiring the full Hessian of the nonlinear operator $A$ \cite{nocedal2006numerical}.  The resulting preconditioned formulation preserves the simplicity of the original \texttt{vpal} iterations while achieving improved stability and convergence speed as we will show in the following numerical experiments.  A detailed derivation of the preconditioner and its analytical motivation is provided in \Cref{sec:precond}, and the complete algorithm is summarized in \Cref{alg:vpal}.

\begin{algorithm}
    \caption{Nonlinear Variable Projected Augmented Lagrangian (\texttt{vpal})}
    \label{alg:vpal}\small
    \begin{algorithmic}[1]
        \Require $A$, $b$, $D$, $\sigma$, $\mu$, $\lambda$ 
        \State Initialize $x_0 = y_0 = z_0 = 0$, $k = 0$
        \While{not converged} 
           \State Set $j = 0$, $x^0 =x_k$, $y^0 = y_k$, $g^0 = \nabla_xf_{\mathrm{joint}}(x_k,y_k)$
           \While{not converged}
           \State Compute preconditioner $P$
            \State Solve $Ps^j = -g^j$ 
            \State Compute step size $\alpha^j$
            \State Update $x^{j+1} = x^j + \alpha^j s^j$
            \State Update $y^{j+1} = y_z(x^{j+1})$
            \State Compute $g^{j+1} = \nabla_xf_{\mathrm{joint}}(x^{j+1},y^{j+1})$
            \State Set $j = j+1$
           \EndWhile
           \State Set $x_{k+1} = x^j$ and $y_{k+1} = y^j$
           \State Update  $z_{k+1} = Dx_{k+1} - y_{k+1} + z_{k}$
            \State $k \gets k+1$
        \EndWhile
        \Ensure $x_k$, minimizer of $f$.
    \end{algorithmic}
\end{algorithm}

We next establish that the fundamental equivalence between the joint and projected formulations persists in the \emph{nonlinear} {\tt vpal} setting (with $P = I$),
providing a solid theoretical foundation for the proposed method. The proofs follow a structure similar to those of the arguments presented in~\cite{chung2023variable}.

\begin{remark}\label{rem:coercivity}
In the following, we assume that $f_{\mathrm{joint}}$ is strictly convex and coercive, which are standard and natural conditions ensuring the existence of minimizers.
When both $A$ and $D$ are linear, a sufficient condition for these properties is the null space condition $\ker(A)\cap \ker(D)=\{0\}$; see, for example,~\cite[p.~197]{engl1996regularization}.
For nonlinear $A$, a more involved set of assumptions is derived in~\cite{seidman1989well}.
In \Cref{sec:linear_numEx,sec:nonlinear_numEx}, these hypotheses will be discussed case by case for each  experiment.
\end{remark}

\begin{lemma}\label{lemma1}
    Assume that $A\in C^0(\mathbb{R}^{n}, \mathbb{R}^m)$. Furthermore, assume that $f_{\mathrm{joint}}(x,y)$ is strictly convex and coercive for $\mu=0$. Then for arbitrary but fixed $z$ and $\mu,\lambda>0$, $f_{\mathrm{joint}}(x,y)$ is strictly convex and has a unique minimizer $(\hat{x},\hat{y})$.
\end{lemma}
\begin{proof}
    The result follows immediately from the fundamental properties of strictly convex and coercive functions \cite{beck2014introduction}. 
\end{proof}
\begin{lemma}\label{lemma:3Statements}
    The following statements hold: 
    \begin{enumerate}
        \item  For any $x$, there exists a $y$ such that $f_{\mathrm{proj}}(x)=f_{\mathrm{joint}}(x,y)$.
        \item  For arbitrary $x$ and $y$, the inequality $f_{\mathrm{proj}}(x)\leq f_{\mathrm{joint}}(x,y)$ holds.
        \item  Let $(\hat{x},\hat{y})$ be the (unique) global minimizer of $f_{\mathrm{joint}}$. Then, $\hat{x}$ is the (unique) global minimizer of $f_{\mathrm{proj}}$.
  \end{enumerate}
\end{lemma}
\smallskip
\begin{proof}
    $\,$ \\[-2ex]
     \begin{enumerate}
     \item This follows directly from the definition in \Cref{eq:f_joint,eq:softTH2,eq:f_proj}: if $y = y_z(x)$, then $f_{\mathrm{proj}}(x)=f_{\mathrm{joint}}(x,y)$.
     \item By definition, 
     \begin{align*}
     y_z(x) &=\argmin_y   f_{\mathrm{joint}}(x,y) =  \argmin_y\tfrac{1}{2\sigma^2} \left\| A(x) - b \right\|_2^2 + \mu \left\|y  \right\|_1 + \tfrac{\lambda^2}{2} \left\|Dx -y + z\right\|_2^2.
    \end{align*}
     Therefore,
     \begin{align*}
         f_{\mathrm{proj}}(x) &=   \tfrac{1}{2\sigma^2} \left\| A(x) - b \right\|_2^2 + \mu \left\|y_z(x)  \right\|_1 + \tfrac{\lambda^2}{2} \left\|Dx -y_z(x) + z\right\|_2^2 
         \leq  f_{\mathrm{joint}}(x,y).
     \end{align*}
     \item By \Cref{eq:f_proj}, 
     $$
     f_{\mathrm{proj}}(x) = f_{\mathrm{joint}}(x,y_z(x)) \geq f_{\mathrm{joint}}(\hat{x},\hat{y})\,, \quad \forall x\in \mathbb{R}^n.
     $$
    Therefore, $f_{\mathrm{proj}}(\hat{x})\geq f_{\mathrm{joint}}(\hat{x},\hat{y})$. By \Cref{lemma:3Statements} item $2$,
    $$
    f_{\mathrm{proj}}(\hat{x}) \leq  f_{\mathrm{joint}}(\hat{x},\hat{y}).
    $$
    Therefore, $f_{\mathrm{proj}}(\hat{x}) = f_{\mathrm{joint}}(\hat{x},\hat{y})$.
    It follows from \Cref{eq:f_proj} that
    $$
    f_{\mathrm{proj}}(x) \geq f_{\mathrm{joint}}(\hat{x},\hat{y}) = f_{\mathrm{proj}}(\hat{x}) \,, \quad \forall x\in \mathbb{R}^n.
    $$
    $\hat{x}$ is the global minimizer of $f_{\mathrm{proj}}(x)$.
    Moreover, since $y_z(\hat{x})$ minimizes $f_{\mathrm{joint}}(\hat{x},\cdot)$ with respect to the second component and $y_z(\hat{x})$ is its unique minimizer, we have $y_z(\hat{x}) = \hat{y}$. To prove the uniqueness, assume that $(\hat{x},\hat{y})$ is the unique global minimizer of $f_{\mathrm{joint}}(x,y)$, but there exists $x^*\neq \hat{x}$ such that $f_{\mathrm{proj}}(\hat{x}) = f_{\mathrm{proj}}(x^*)\leq f_{\mathrm{proj}}(x), \forall x\in\mathbb{R}^n$. Therefore,
    $$
    f_{\mathrm{joint}}(\hat{x},\hat{y}) = f_{\mathrm{joint}}(\hat{x},y_z(\hat{x})) = f_{\mathrm{proj}}(\hat{x}) = f_{\mathrm{proj}}(x^*) = f_{\mathrm{joint}}(x^*,y_z(x^*)).
    $$
    In this case, $(x^*,y_z(x^*))\neq (\hat{x},y_z(\hat{x}))$ is another global minimizer of $f_{\mathrm{joint}}(x,y)$, which contradicts the assumption. Therefore, $\hat{x}$ is the unique global minimizer of $f_{\mathrm{proj}}(x)$.
      \end{enumerate}
\end{proof}

\begin{lemma}
Assume $A\in C^0(\mathbb{R}^{n}, \mathbb{R}^m)$ and $f_{\mathrm{joint}}(x,y)$ is strictly convex and coercive for $\mu=0$. Furthermore, assume that $z$ is fixed and $\mu,\lambda>0$; then all local minimizers of $ f_{\mathrm{proj}}(x)$ are global.
\end{lemma}
\begin{proof}
    Let $(\hat{x},\hat{y}) =(\hat{x},y_z(\hat{x})) $ denote the unique global minimizer of $f_{\mathrm{joint}}(x,y)$ according to \Cref{lemma1}. Let $\bar{x}$ be a local minimizer of $f_{\mathrm{proj}}(x)$. If $\bar{x} = \hat{x}$. Then, $ f_{\mathrm{proj}}(\bar{x}) =  f_{\mathrm{proj}}(\hat{x}) = f_{\mathrm{joint}}(\hat{x},y_z(\hat{x}))$. If $\bar{x} \neq \hat{x}$. We know that $f_{\mathrm{joint}}(x,y)$ is strictly convex in $x$ and $y$. Fix $\bar{y}\in \mathbb{R}^n$, where $\bar{y} = y_z(\bar{x})$. The strict convexity in $x$ of $f_{\mathrm{joint}}(x,y)$ implies that for any $\bar{x}\in \mathbb{R}^n$ with $\bar{x}\neq\hat{x}$ and $\epsilon >0$, $\exists x^*$ with $x^*\neq \bar{x}$ and $\|x^*-\bar{x}\|_2<\epsilon $ such that 
    $$
    f_{\mathrm{joint}}(x^*,\bar{y}) < f_{\mathrm{joint}}(\bar{x},\bar{y}) = f_{\mathrm{proj}}(\bar{x})\,.
    $$
    By Lemma \ref{lemma:3Statements} item $1$, we have $f_{\mathrm{proj}}(x^*) = f_{\mathrm{joint}}(x^*,y_z(x^*))$, and by Lemma \ref{lemma:3Statements} item $2$, we have 
    $f_{\mathrm{joint}}(x^*,y_z(x^*)) \leq f_{\mathrm{joint}}(x^*,\bar{y})$. Therefore
    $$
    f_{\mathrm{proj}}(x^*) = f_{\mathrm{joint}}(x^*,y_z(x^*))\leq f_{\mathrm{joint}}(x^*,\bar{y}) < f_{\mathrm{joint}}(\bar{x},\bar{y}) = f_{\mathrm{proj}}(\bar{x})\,.
    $$
    Thus, there does not exists a neighborhood $\mathcal{D}(\bar{x})$ in which $f_{\mathrm{proj}}(x)\geq f_{\mathrm{proj}}(\bar{x}), \forall x\in \mathcal{D}(\bar{x})$. Thus, $\bar{x}\neq \hat{x}$ is not a local minimizer. This contradiction shows that all local minimizers are global.
    \end{proof}

\begin{theorem}
Assume $A\in C^1(\mathbb{R}^{n}, \mathbb{R}^m)$.
Suppose that the joint objective $f_{\mathrm{joint}}(x, y)$ is strictly convex and coercive for $\mu = 0$, and that $\mu, \lambda > 0$ are chosen sufficiently large. Further, assume that the matrix $P$ is symmetric positive definite, and that there exists a constant $\beta > 0$ such that the angle condition
$
    -\frac{g^\top s}{\|g\| \cdot \|s\|} \geq \beta
$
holds uniformly across all iterations. Let $\alpha$ be the exact step size in \Cref{eq:optimal_step_size}. Then the iterates produced by \Cref{alg:vpal} converge to the unique minimizer $\hat{x}$ of the generalized Lasso problem \Cref{eq:genLasso}.
\end{theorem}

\begin{proof}
    The convergence of the outer loop of \Cref{alg:vpal} has been established for augmented Lagrangian methods \cite{boyd2011distributed}. Thus, it remains to be shown that the variable projection in the inner loop solves 
    \begin{equation}\label{eq:pf}
    (x_{k+1},y_{k+1}) = \arg\min_{x,y} \ \tilde L_{\mathrm{aug}}(x,y,z_k)\,.
    \end{equation}
    Note that the inner loop utilizes a line search method. Recall that in \Cref{sec:nonlinear_vpal} we showed that $g= \nabla f_{\mathrm{proj}}(x)$. With $P$ being SPD, it is well known that $s = -P^{-1}g$ is a descent direction for $f_{\mathrm{proj}}(x)$. 
    The exact step size guarantees sufficient decrease at every iteration. Thus, the three conditions ensure that every accumulation point of $\{x_k\}$ generated by the inner loop is a stationary point of $f_{\mathrm{proj}}(x)$ \cite{geiger1999numerische}. Since $f_{\mathrm{proj}}(x)$ is strictly convex for $\mu=0$, the inner loop converges to the unique global minimizer of $f_{\mathrm{proj}}(x)$. Then, we update the optimal $y$ using the soft-thresholding formula $y_z(x)$, ensuring that \Cref{eq:pf} is satisfied. Thus, \Cref{alg:vpal} converges to the unique minimizer of \Cref{eq:genLasso}.
\end{proof}

We note that the inner problem (lines 4–11 in \Cref{alg:vpal}) can be solved inexactly; in fact, our numerical experiments indicate that a single iteration of the inner loop is often sufficient.

\subsection{Preconditioned \texttt{vpal}}\label{sec:precond}
In the previous section, we established the convergence theory for update directions of the form $ s = -P^{-1}g$, where $P$ is a symmetric positive definite matrix. This generalized search direction encompasses both standard gradient descent and curvature-informed updates. Motivated by this flexibility, we now introduce a specific choice of $P$ that leverages second-order information about the reduced problem, leading to an accelerated variant of the \texttt{vpal} algorithm.

We refer to this new method as the \emph{preconditioned \texttt{vpal}} (\texttt{pvpal}), as it replaces the steepest descent step in $x$ with a preconditioned update. The underlying idea is inspired by classical optimization theory: convergence speed is often enhanced when the preconditioner approximates the Hessian of the objective function, ideally reproducing the behavior of Newton’s method.

However, designing such a preconditioner is nontrivial in this setting. The objective $f_{\mathrm{proj}}(x)$ includes the nonsmooth term $\mu\|y_z(x)\|_1$, complicating direct computation of second-order derivatives. Nonetheless, by exploiting the optimality conditions satisfied by $y_z(x)$ and utilizing a smooth approximation of the shrinkage operator, we derive an efficient and analytically motivated preconditioner that captures the local curvature of the reduced problem. The resulting method retains the simplicity of the original \texttt{vpal} scheme while significantly improving its convergence behavior, as we demonstrate in the numerical experiments that follow.

	The gradient of $f_{\mathrm{proj}}(x)$ can be expressed as
	\begin{equation}\label{eq:gradient}
		\nabla_x f_{\mathrm{proj}}(x) = \tfrac{1}{\sigma^2} J_A(x)^\top\bigl(A(x)-b\bigr) + \lambda^2 D^\top\bigl(Dx-y_z(x)+z\bigr).
	\end{equation}
	
	At points where the involved mappings are differentiable, the Hessian of $f_{\mathrm{proj}}$ is given by
	\begin{align}
		\nabla_x^2 f_{\mathrm{proj}}(x) ={}& \tfrac{1}{\sigma^2}J_A(x)^\top J_A(x) + \tfrac{1}{\sigma^2} \sum_{i=1}^{m} \bigl(A_i(x)-b_i\bigr)\nabla^2 A_i(x)
		+ \lambda^2 D^\top D - \lambda^2 D^\top J_{y_z}(x).
		\label{eq:true_Hessian}
	\end{align}
	The second term contains second-order derivatives of the nonlinear forward operator $A$. Following the classical Gauss--Newton approach, we neglect this term and obtain the approximation
	\begin{equation}\label{eq:GN_Hessian}
		H_{\mathrm{GN}}(x) = \tfrac{1}{\sigma^2}J_A(x)^\top J_A(x) + \lambda^2 D^\top D - \lambda^2 D^\top J_{y_z}(x).
	\end{equation}
	
	Recall that $y_z(x)$ is obtained by applying the soft-thresholding operator component-wise to the vector $Dx+z$, see \Cref{eq:softTH2}. That is,
	\[
	y_z(x) = \mathcal{T}_{\frac{\mu}{\lambda^2}}(Dx+z),
	\]
	where $\mathcal{T}_{\frac{\mu}{\lambda^2}}$ denotes the soft-thresholding operator with threshold parameter $\frac{\mu}{\lambda^2}$. At points where the soft-thresholding operator is differentiable,
	\[
	J_{y_z}(x) = J_{\mathcal{T}_{\frac{\mu}{\lambda^2}}}(Dx+z)D.
	\]
	Consequently, the Gauss--Newton approximation can be written as
	\begin{equation}\label{eq:GN_Hessian_threshold}
		H_{\mathrm{GN}}(x) = \tfrac{1}{\sigma^2}J_A(x)^\top J_A(x) + \lambda^2 D^\top \left( I- J_{\mathcal{T}_{\frac{\mu}{\lambda^2}}}(Dx+z) \right)D.
	\end{equation}
	
	To exploit this structure, we consider a continuously differentiable approximation of the soft-thresholding operator. Consider first the simpler case $q(x)=|x|$ with $x\in\mathbb{R}$. The Huber approximation of $q$ is defined by
	\begin{equation*}
		h(x) =
		\begin{cases}
			\tfrac{1}{2\delta}x^2, & |x|\le \delta, \\[4pt]
			|x|-\tfrac{1}{2}\delta, & \text{otherwise},
		\end{cases}
	\end{equation*}
	which smooths the non-differentiability of $q$ in a neighborhood of zero.
	
	Since $\mathcal{T}_{\frac{\mu}{\lambda^2}}(\cdot)$ acts component-wise, we apply the same idea to smooth its non-differentiable points at $\pm\tfrac{\mu}{\lambda^2}$. Let
	\[
	\zeta=\tfrac{\mu}{\lambda^2}
	\]
	denote the threshold parameter, and let $\epsilon>0$. We define the smoothed operator $S_\epsilon:\mathbb{R}^{\ell}\to\mathbb{R}^{\ell}$ component-wise by
	\begin{equation*}
		[S_{\epsilon}(x)]_i =
		\begin{cases}
			x_i-\zeta-\tfrac{1}{2}\epsilon,
			& x_i>\zeta+\epsilon, \\[6pt]
			\tfrac{1}{2\epsilon}(x_i-\zeta)^2,
			& \zeta < x_i \le \zeta+\epsilon, \\[6pt]
			0,
			& -\zeta \le x_i \le \zeta, \\[6pt]
			-\tfrac{1}{2\epsilon}(x_i+\zeta)^2,
			& -\zeta-\epsilon \le x_i < -\zeta, \\[6pt]
			x_i+\zeta+\tfrac{1}{2}\epsilon,
			& x_i<-\zeta-\epsilon.
		\end{cases}
	\end{equation*}
	
	The operator $S_\epsilon$ provides a continuously differentiable approximation of the soft-thresholding operator. Its Jacobian $J_\epsilon\in\mathbb{R}^{\ell\times\ell}$ is diagonal, with entries
	\begin{equation*}
		[J_{\epsilon}(x)]_{i,i} =
		\begin{cases}
			1,
			& x_i>\zeta+\epsilon, \\[6pt]
			\dfrac{x_i-\zeta}{\epsilon},
			& \zeta < x_i \le \zeta+\epsilon, \\[8pt]
			0,
			& -\zeta \le x_i \le \zeta, \\[6pt]
			-\dfrac{x_i+\zeta}{\epsilon},
			& -\zeta-\epsilon \le x_i < -\zeta, \\[8pt]
			1,
			& x_i<-\zeta-\epsilon.
		\end{cases}
	\end{equation*}
	Thus,
	\[
	0\le [J_{\epsilon}(x)]_{i,i}\le 1, \qquad i=1,\ldots,\ell.
	\]
	
	Replacing the Jacobian of the soft-thresholding operator by $J_\epsilon$ yields the Gauss--Newton-type preconditioner
	\begin{equation}\label{eq:approx_Hessian}
		\widetilde H_{\mathrm{GN}}(x) = \tfrac{1}{\sigma^2}J_A(x)^\top J_A(x) + \lambda^2 D^\top \bigl(I-J_{\epsilon}(Dx+z)\bigr)D.
	\end{equation}
	
	From a theoretical standpoint, the operator $\widetilde H_{\mathrm{GN}}(x)$ must be invertible. Since
	\[
	0\le [J_{\epsilon}(x)]_{i,i}\le 1,
	\]
	the matrix $I-J_{\epsilon}(x)$ is positive semidefinite, and therefore both terms in \Cref{eq:approx_Hessian} are positive semidefinite. In particular, for any $v\in\mathbb{R}^n$,
	\begin{align*}
		v^\top \widetilde H_{\mathrm{GN}}(x)v &= \tfrac{1}{\sigma^2} \|J_A(x)v\|_2^2 + \lambda^2 \| \bigl(I-J_\epsilon(Dx+z)\bigr)^{1/2}Dv \|_2^2.
	\end{align*}
	Hence, $\widetilde H_{\mathrm{GN}}(x)$ is symmetric positive definite, and therefore invertible, provided that
	\begin{equation}\label{eq:kernel_condition_GN}
		\ker\bigl(J_A(x)\bigr) \cap \ker\!\left( \bigl(I-J_\epsilon(Dx+z)\bigr)^{1/2}D \right) = \{0\}.
	\end{equation}
	For a nonlinear forward operator, we assume that this condition holds along the iterates generated by the algorithm. In the linear case, $J_A(x)=A$, so the first term reduces to the corresponding linear expression.
	
	Combining the above results, the update rule for the variable $x$ in the {\tt pvpal} method is given by
	\begin{equation}\label{eq:pvpal_update}
		x_{k+1} = x_k - \alpha_k\, \widetilde H_{\mathrm{GN}}(x_k)^{-1} \nabla_x f_{\mathrm{proj}}(x_k),
	\end{equation}
	where $\alpha_k$ denotes the step size, which can be estimated using the same strategies employed in the plain {\tt vpal} method and discussed in \Cref{sec:nonlinear_vpal} (\Cref{eq:optimal_step_size,eq:alpha_lin}).

\section{Linear Numerical Experiments}\label{sec:linear_numEx}
Next, we present a selection of classical linear numerical experiments addressing various imaging problems. Specifically, we compare the performance of the {\tt vpal} method and its preconditioned counterpart, {\tt pvpal}, across two different settings: deblurring and inpainting. An additional experiment on computed tomography (CT) is given in Appendix A. In all experiments, the regularization operator $D\in\mathbb{R}^{l\times n}$ is chosen as the first-order forward finite-difference operator. For a two-dimensional image of size $n_1\times n_2$, its output dimension is
\[
l=(n_1-1)n_2+n_1(n_2-1)
  =2 n_1n_2-n_1-n_2.
\]
The operator is implemented using non-periodic forward differences between adjacent grid points within the computational domain; differences requiring grid points outside the domain are not included. Hence, no ghost values are introduced at the boundary. The two algorithms are evaluated using two distinct strategies for selecting the step size $ \alpha$: a \textit{linearized} approach and an \textit{optimal} one. Accordingly, we refer to the linearized and optimal variants of each algorithm throughout the discussion. We fix the regularization parameter $\mu$ and the penalty parameter $\lambda$ within each experiment, with their values reported in the corresponding sections. These parameters were selected empirically, without using the true solution. A systematic study of parameter selection is beyond the scope of this work. In practice, however, both parameters can be updated adaptively. For instance, \cite{chung2023variable} discusses strategies for updating $\mu$, while \cite{boyd2011distributed} presents a residual-balancing scheme for adapting $\lambda$ in ADMM.

Regarding the assumptions discussed in~\Cref{rem:coercivity}, let us first highlight that  $\ker(D) = \{t\, e \mid t\in \mathbb{R}\}$, where $e$ denotes the constant vector of ones. Therefore, the null space condition is equivalent to requiring $e \notin \ker(A)$. This condition is trivially satisfied in the deblurring and inpainting experiments (\Cref{ssec:deblurring,ssec:inpainting}).

\subsection{Experiment 1: Deblurring Peppers}\label{ssec:deblurring}
In this first experiment, we consider an image-deblurring problem in which the ground-truth image $x_{\mathrm{true}}$ is a grayscale image depicting a group of peppers, with $n=m=256$. The image is degraded using a motion-blur point-spread function (PSF) and corrupted with 1\% additive white Gaussian noise. This means that the final observed image is given by
	\begin{equation*}
		b = Ax_{\mathrm{true}}
		+ \eta \frac{\|Ax_{\mathrm{true}}\|_2}
		{\|\varepsilon\|_2}\,\varepsilon,
		\qquad \eta=0.01,
	\end{equation*}
	where $\varepsilon$ is a realization of white Gaussian
	noise. The experiment was generated using the IR Tools MATLAB package \cite{Gazzola}.
Since the PSF has a support of approximately 10 pixels, the resulting blurred image is cropped to mitigate the effects of boundary conditions, yielding a final observed image of size $237 \times 237$. \Cref{fig:Peppers setting} illustrates the original image, the PSF used for blurring, and the corresponding blurred and noisy observation.

To evaluate the performance of the {\tt vpal} method and its preconditioned counterpart {\tt pvpal}, we set the regularization parameter to $\mu = 10^{-2}$. The augmented Lagrangian penalty parameter is set to $\lambda = 0.5$ for both methods, and both algorithms are run for 200 iterations.

\Cref{fig:Peppers RRE and time} presents the relative reconstruction error (RRE) achieved by each method at every iteration, as well as its progression over time (in seconds), for both step-size strategies. Specifically, the first row illustrates the RRE behavior with respect to the number of iterations and computation time using the optimal step-size strategy, while the second row focuses on the linearized approach. In both cases, the use of a preconditioner clearly enhances the convergence speed, requiring fewer iterations and less time than standard \texttt{vpal} to achieve a given accuracy. To further highlight this acceleration, which is our primary interest, \Cref{tab:Peppers comparison} summarizes the reconstruction quality and the computation time needed to achieve it, again for both step-size strategies. The first two result columns report the results of each method after 200 iterations. The third result column indicates the number of {\tt pvpal} iterations required to match the runtime of 200 {\tt vpal} iterations, while the fourth result column shows the number of {\tt pvpal} iterations needed to reach an RRE comparable to that of 200 {\tt vpal} iterations.  Additionally, \Cref{fig:Peppers reconstructions} shows the reconstructions obtained with both methods at equal computation times. 
More specifically, the first result column reports the reconstructions obtained after $0.2$ seconds of running time, 
the second result column shows the results obtained after $0.5$ seconds, 
and the last one corresponds to approximately $1$ second. 
We can observe that the {\tt pvpal} method instantly removes the blur in the image and, after a few iterations, 
reduces the ringing artifacts near the boundaries caused by the choice of boundary conditions. 
The standard {\tt vpal} method achieves comparable results but requires more iterations and a longer computation time.

Lastly, since the step-size selection plays a crucial role in the performance and efficiency of the methods, 
\Cref{fig:Peppers step size} reports the values chosen at each iteration under both the linearized and optimal strategies. 
For the {\tt vpal} method, both strategies produce highly oscillatory step-size sequences. 
In contrast, the preconditioned version exhibits a stable and consistent behavior, as for the {\tt pvpal} method both strategies always select $\alpha = 1$ at each iteration (data not show).

\begin{figure}[htbp]
	\centering
	\begin{subfigure}[b]{0.32\textwidth}
		\includegraphics[width=\textwidth]{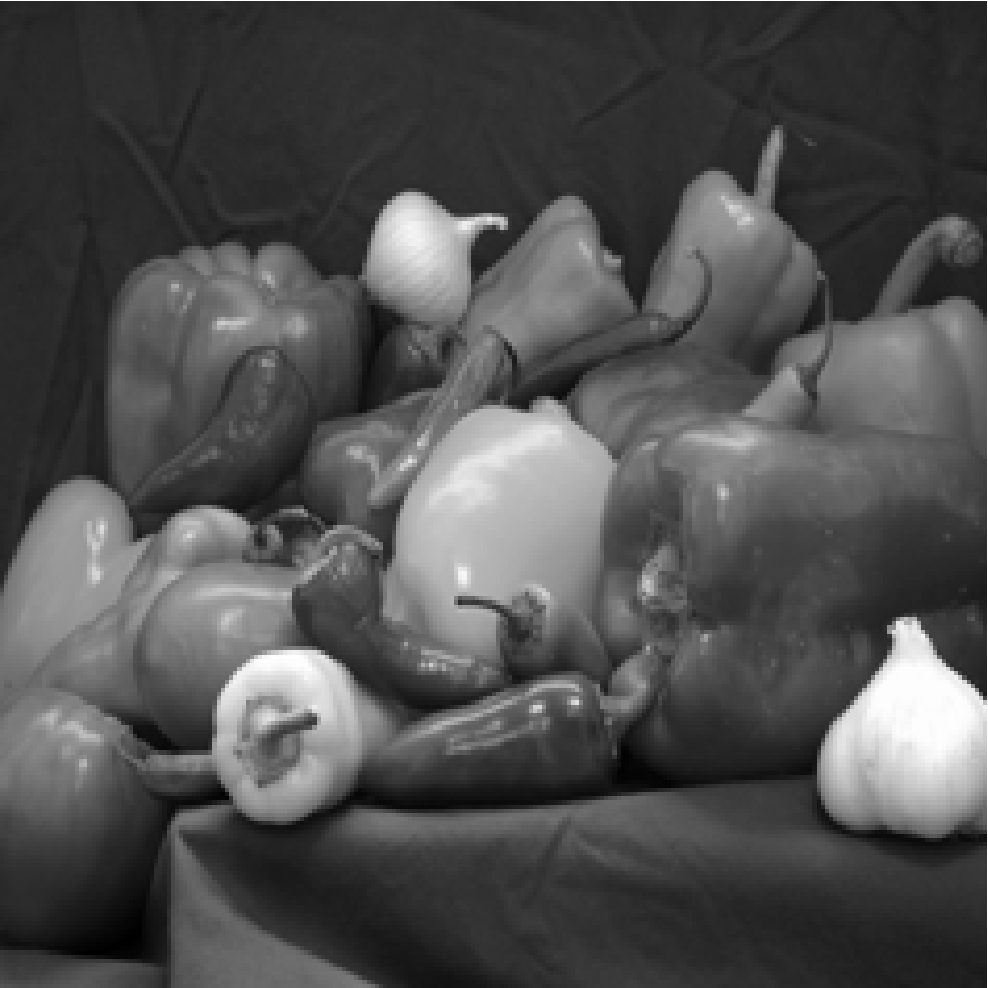}
		\caption{True image}
	\end{subfigure}
	\hfill
	\begin{subfigure}[b]{0.32\textwidth}
		\includegraphics[width=\textwidth]{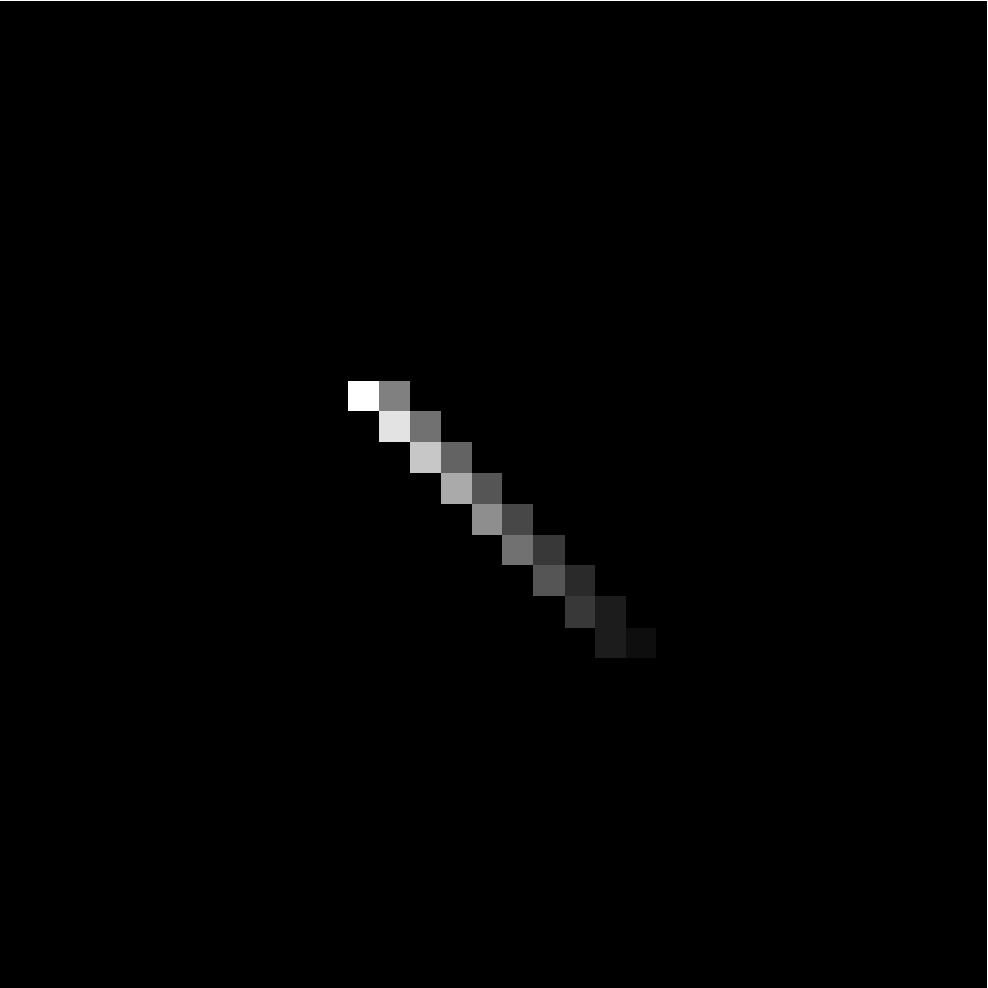}
		\caption{Motion PSF}
	\end{subfigure}
	\hfill
	\begin{subfigure}[b]{0.32\textwidth}
		\includegraphics[width=\textwidth]{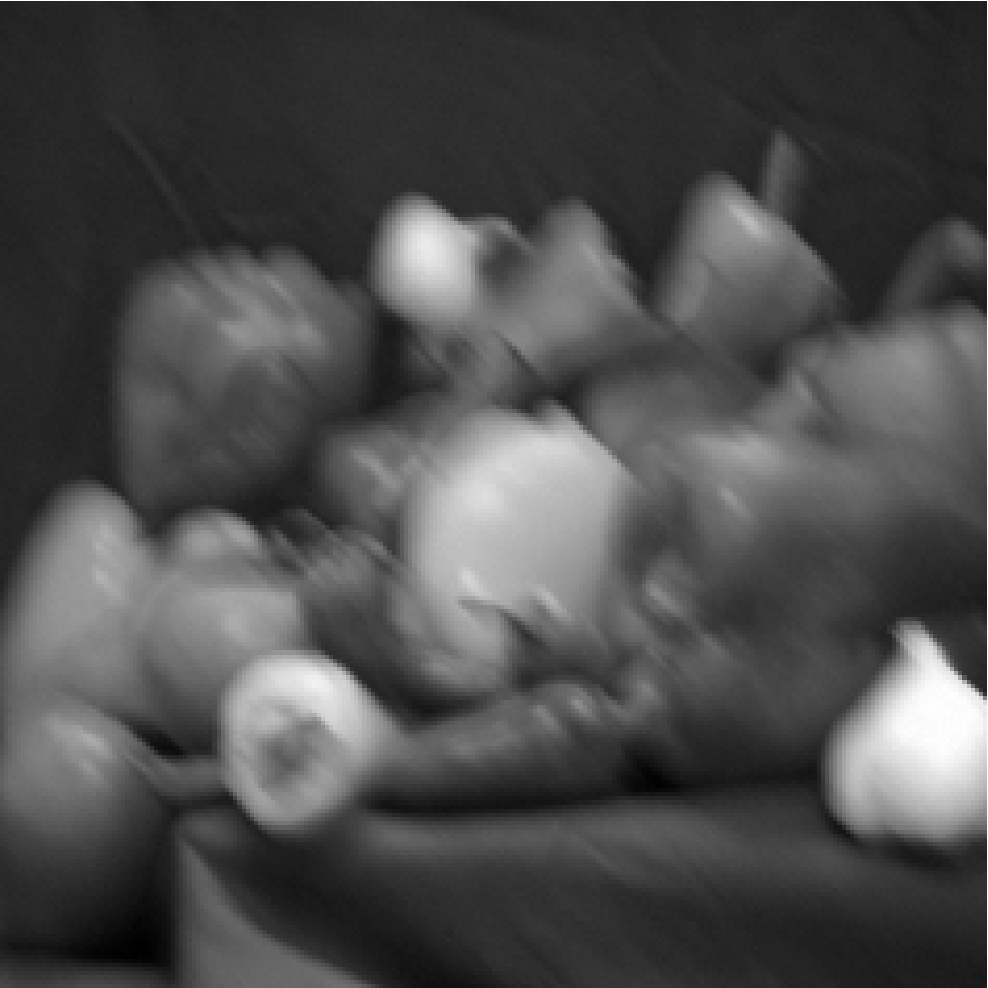}
		\caption{Observed image}
	\end{subfigure}
	\caption{Experiment 1: (a) Original image, (b) PSF used in the deblurring process, and (c) resulting blurred and noisy observation.}
	\label{fig:Peppers setting}
\end{figure}

\begin{figure}[htbp]
	\centering
	\begin{subfigure}[b]{0.42\textwidth}
		\includegraphics[width=\textwidth]{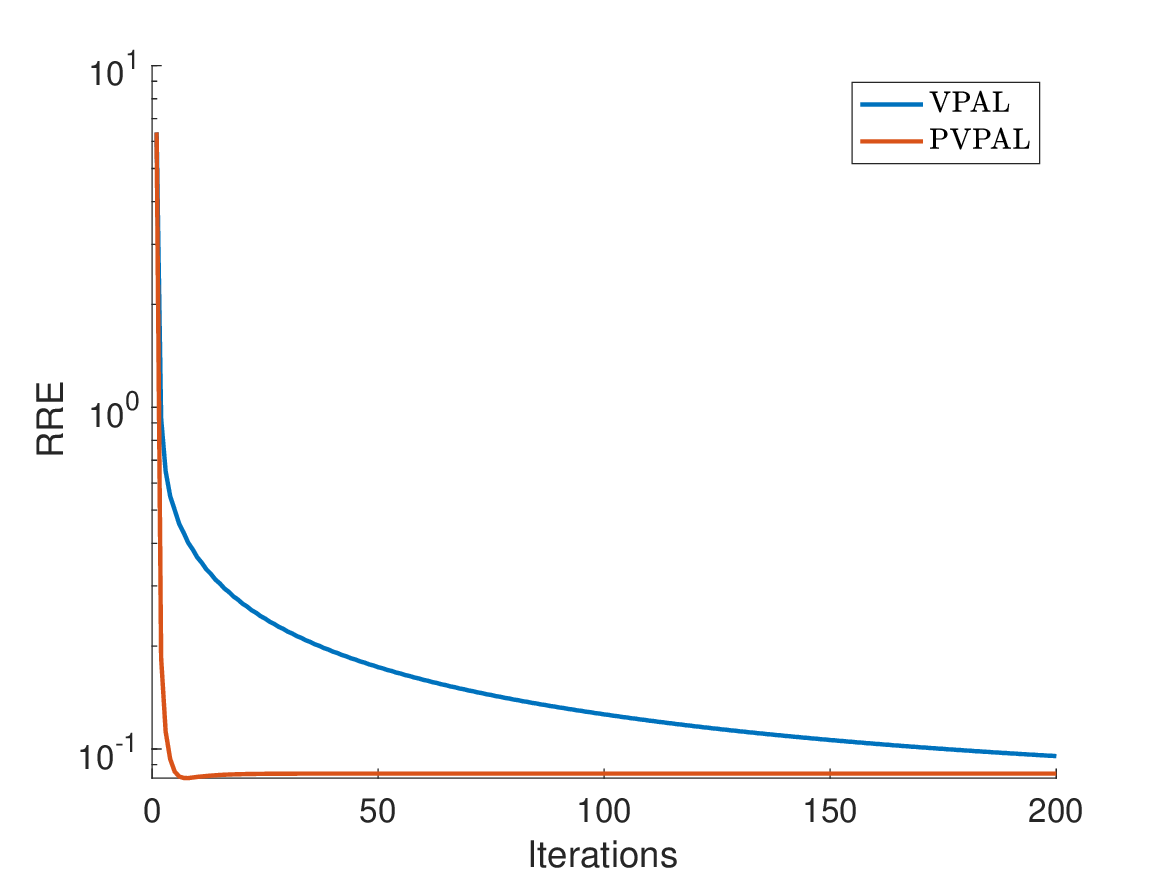}
		\caption{opt: RRE vs Iterations}
		
	\end{subfigure}
	\begin{subfigure}[b]{0.42\textwidth}
		\includegraphics[width=\textwidth]{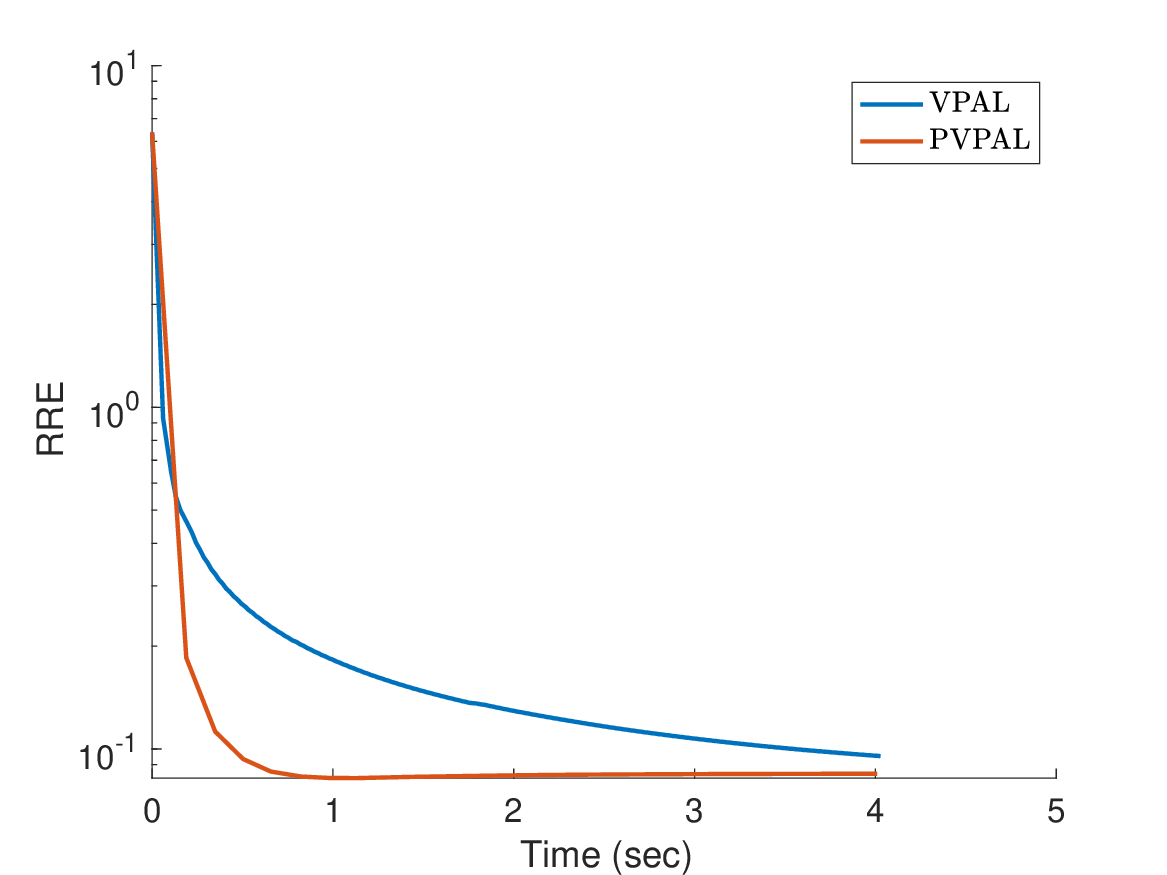}
		\caption{opt: RRE vs Time}
		
	\end{subfigure}
	\\
	\begin{subfigure}[b]{0.42\textwidth}
		\includegraphics[width=\textwidth]{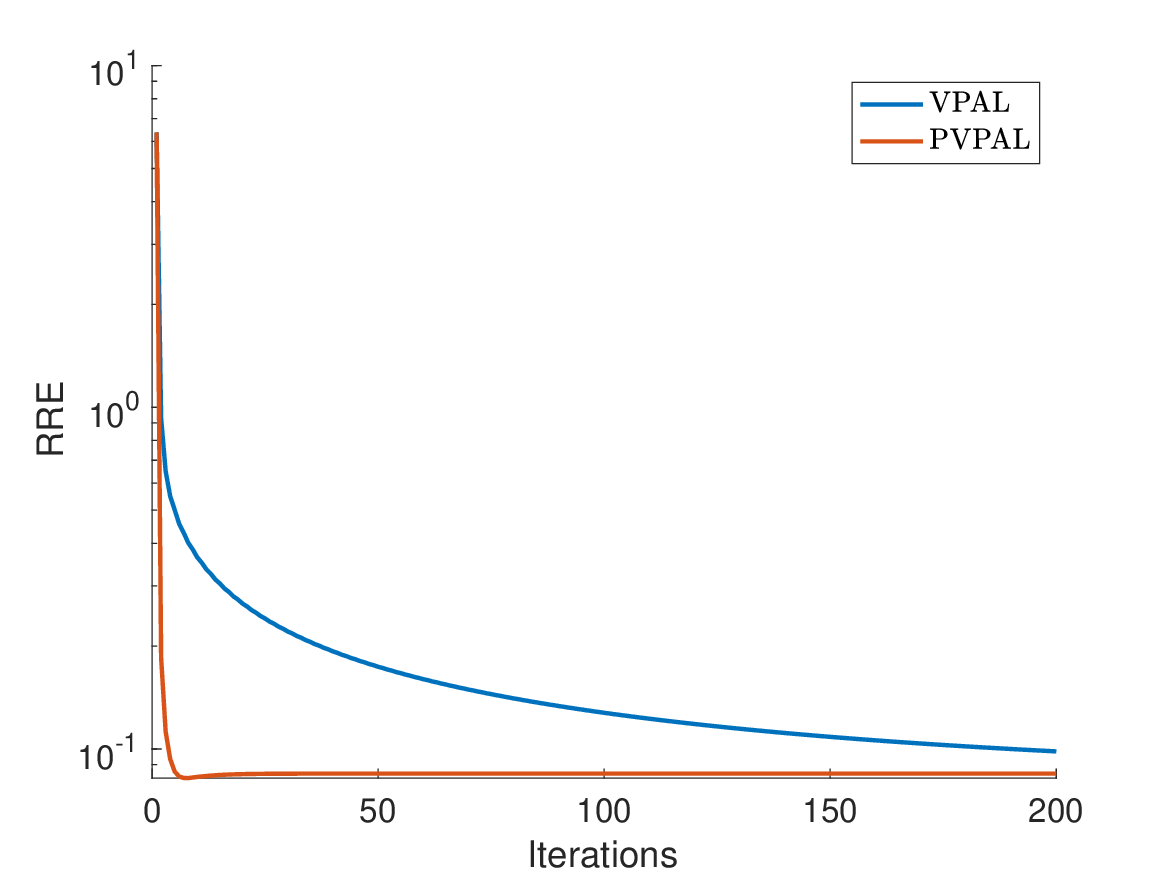}
		\caption{lin: RRE vs Iterations}
	\end{subfigure}
	\begin{subfigure}[b]{0.42\textwidth}
		\includegraphics[width=\textwidth]{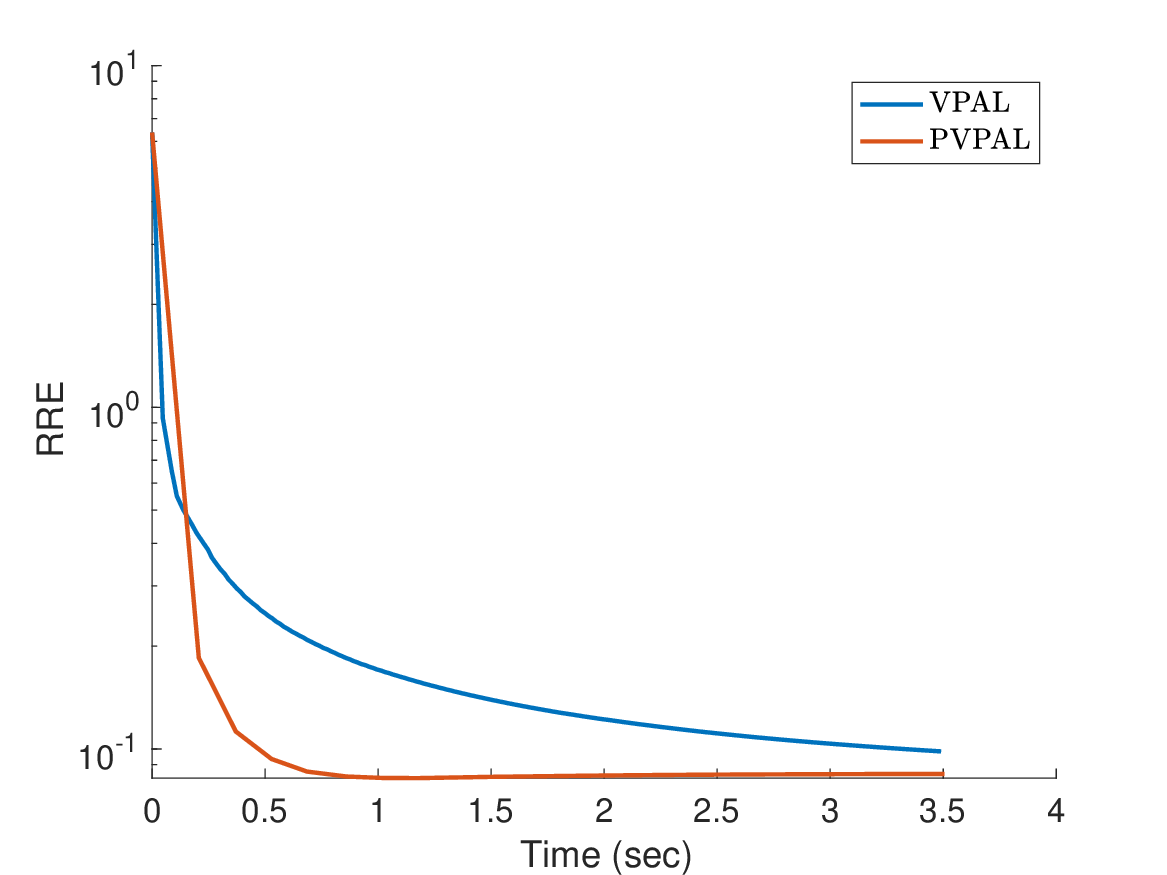}
		\caption{lin: RRE vs Time}
		
	\end{subfigure}

	\caption{Experiment 1: Comparison between {\tt vpal} and {\tt pvpal} methods. The first row shows the behavior of the RRE with respect to the number of iterations and the corresponding computation time (in seconds) using the optimal step-size strategy. In contrast, the second row focuses on the linearized strategy.}
	\label{fig:Peppers RRE and time}
\end{figure}

\begin{figure}[htbp]
	\centering
	\begin{subfigure}[b]{0.3\textwidth}
		\includegraphics[width=\textwidth]{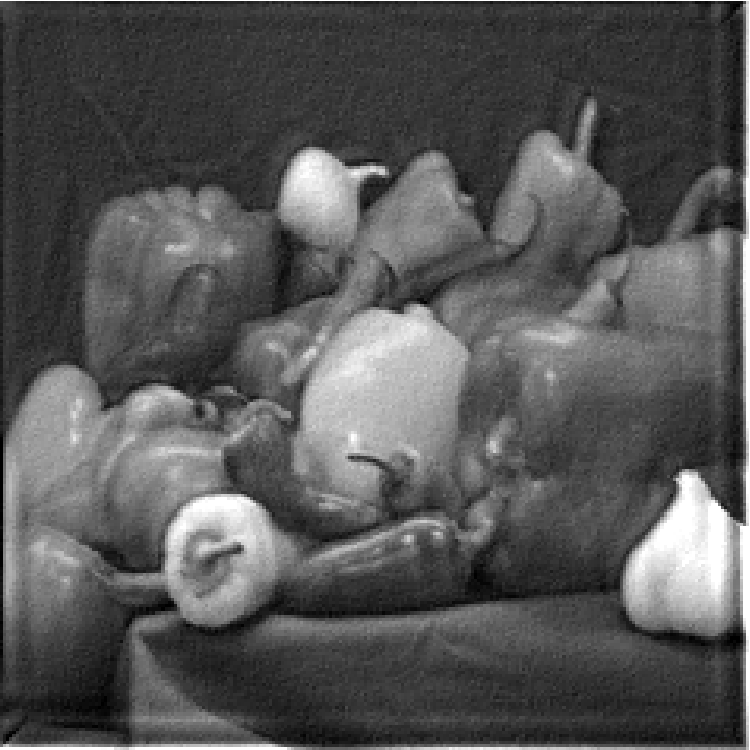}
		\caption{{\tt vpal}: Iter: 13, Time: 0.2250 seconds, RRE: 0.3132, PSNR: 18.62 dB.}
	\end{subfigure}
	\hfill
	\begin{subfigure}[b]{0.3\textwidth}
		\includegraphics[width=\textwidth]{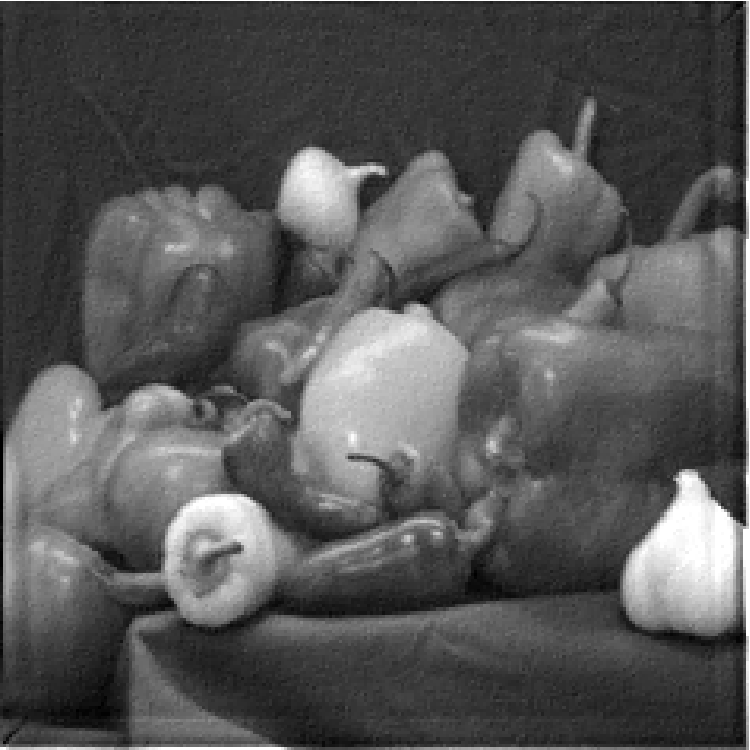}
		\caption{{\tt vpal}: Iter: 29, Time: 0.5301 seconds, RRE: 0.2207, PSNR: 21.43 dB.}
	\end{subfigure}
	\hfill
	\begin{subfigure}[b]{0.3\textwidth}
		\includegraphics[width=\textwidth]{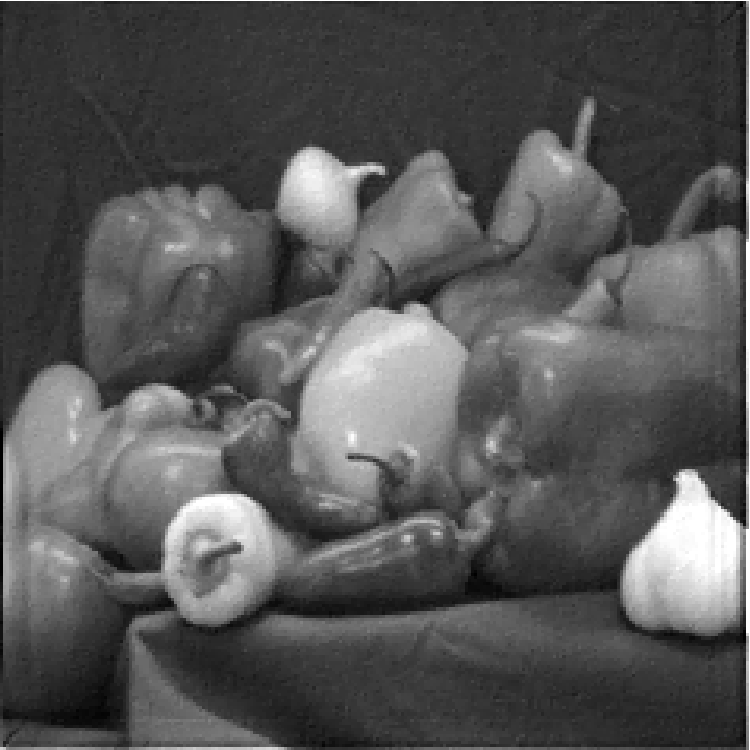}
		\caption{{\tt vpal}: Iter: 58, Time: 1.007 seconds, RRE: 0.1614, PSNR: 24.15 dB.}
	\end{subfigure}
	
	\vspace{0.5cm} %
	
	\begin{subfigure}[b]{0.3\textwidth}
		\includegraphics[width=\textwidth]{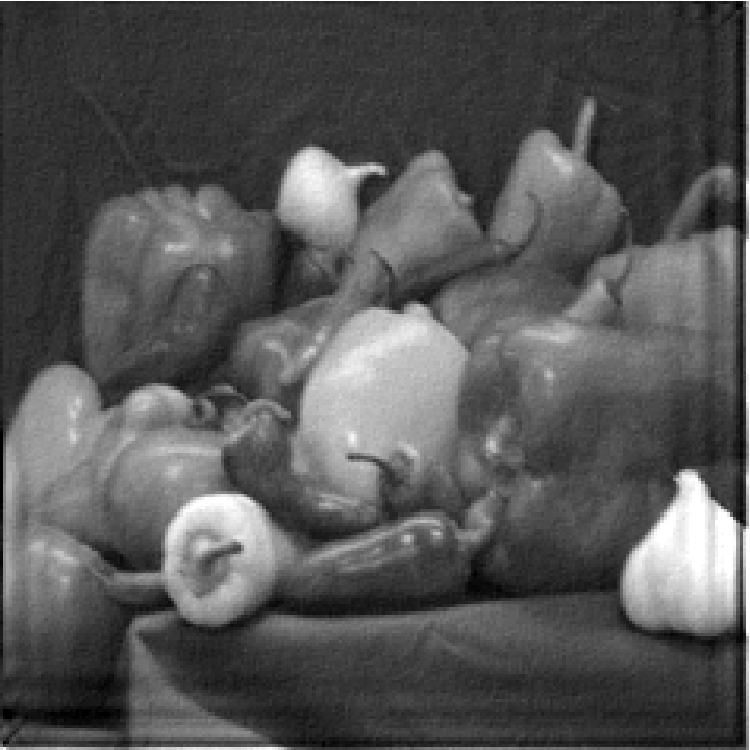}
		\caption{{\tt pvpal}: Iter: 1, Time: 0.2172 seconds, RRE: 0.1848, PSNR: 22.97 dB.}
	\end{subfigure}
	\hfill
	\begin{subfigure}[b]{0.3\textwidth}
		\includegraphics[width=\textwidth]{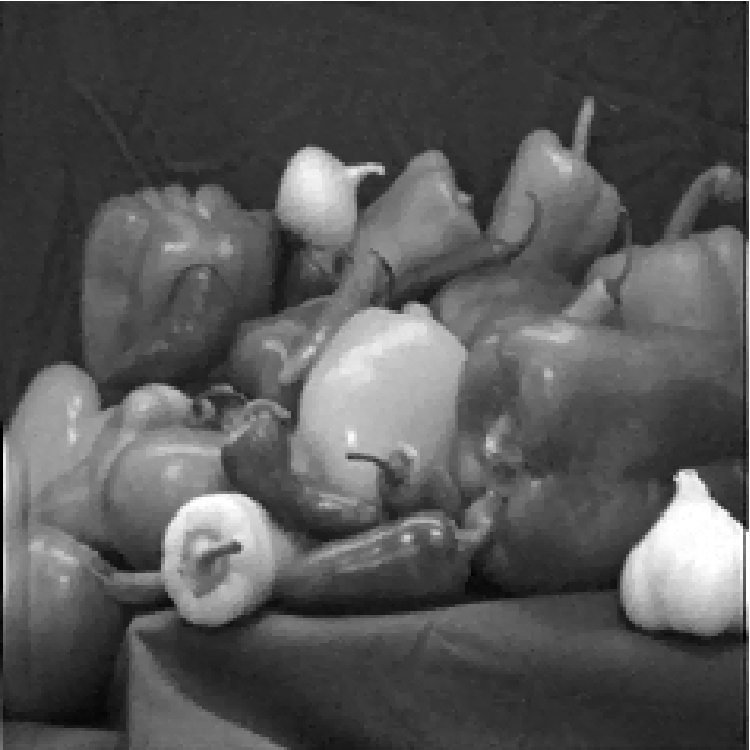}
		\caption{{\tt pvpal}: Iter: 3, Time: 0.5287 seconds, RRE: 0.0934, PSNR: 28.90 dB.}
	\end{subfigure}
	\hfill
	\begin{subfigure}[b]{0.3\textwidth}
		\includegraphics[width=\textwidth]{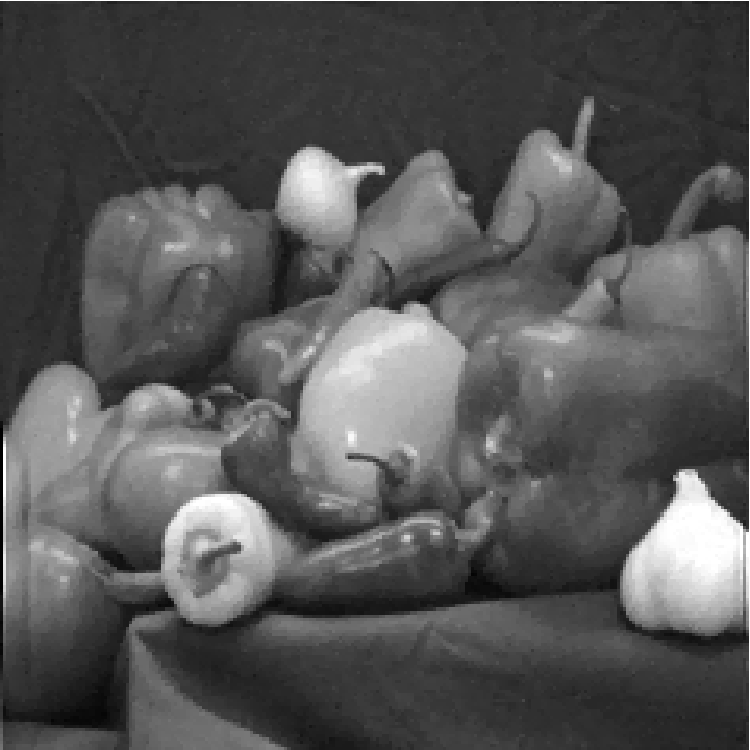}
		\caption{{\tt pvpal}: Iter: 6, Time: 0.9994 seconds, RRE: 0.0823, PSNR: 30.00 dB.}
	\end{subfigure}
	
	\caption{Comparison of reconstructions produced by the {\tt vpal} and {\tt pvpal} methods at equal computation times.}
	\label{fig:Peppers reconstructions}
\end{figure}

\begin{table}[tbhp]
	\centering \small
	\begin{tabular}{c|c|c|c|c|c|}
		\cline{2-6}
		& & {\tt vpal} (200 iter) & {\tt pvpal} (200 iter) & {\tt pvpal} (26 iter) & {\tt pvpal} (3 iter) \\ 
		\hline 
		\multirow{2}{*}{\rotatebox[origin=c]{90}{\textbf{opt}}} 
		& \textbf{RRE} & \textcolor{matlab1}{0.0951} & 0.0847 & 0.0846 & \textcolor{matlab1}{0.0934} \\ 
		\cline{2-6} & \textbf{Time (s)}
		& \textcolor{matlab2}{4.0466} & 29.768 & \textcolor{matlab2}{4.0097} & 0.5035 \\
		\hline
		& & {\tt vpal} (200 iter) & {\tt pvpal} (200 iter) & {\tt pvpal} (22 iter) & {\tt pvpal} (3 iter) \\ \hline
		\multirow{2}{*}{\rotatebox[origin=c]{90}{\textbf{lin}}} 
		& \textbf{RRE} & \textcolor{matlab1}{0.0982} & 0.0847 & 0.0845 & \textcolor{matlab1}{0.0934} \\ 
		\cline{2-6} & \textbf{Time (s)}
		& \textcolor{matlab2}{3.5068} & 30.314 & \textcolor{matlab2}{3.5049} & 0.5287 \\
		\hline
	\end{tabular}
	\caption{Experiment 1: Comparison between the standard {\tt vpal} method and its preconditioned version, {\tt pvpal}. The top half of the table reports results obtained using the optimal strategy to compute the step-size at each iteration, while the bottom half refers to the linearized approach. Values highlighted in \textcolor{matlab1}{blue} indicate the final RRE achieved by {\tt vpal} after 200 iterations. We also report the number of {\tt pvpal} iterations needed to reach the same RRE and the corresponding computation time. Values in \textcolor{matlab2}{red} show the time required by {\tt vpal} to complete 200 iterations, and the performance of {\tt pvpal} within approximately the same amount of time.}
	\label{tab:Peppers comparison}
\end{table}

\begin{figure}[htbp]
	\centering
	\begin{subfigure}[b]{1\textwidth}
		\includegraphics[width=\textwidth, trim=0 290 0 0, clip]{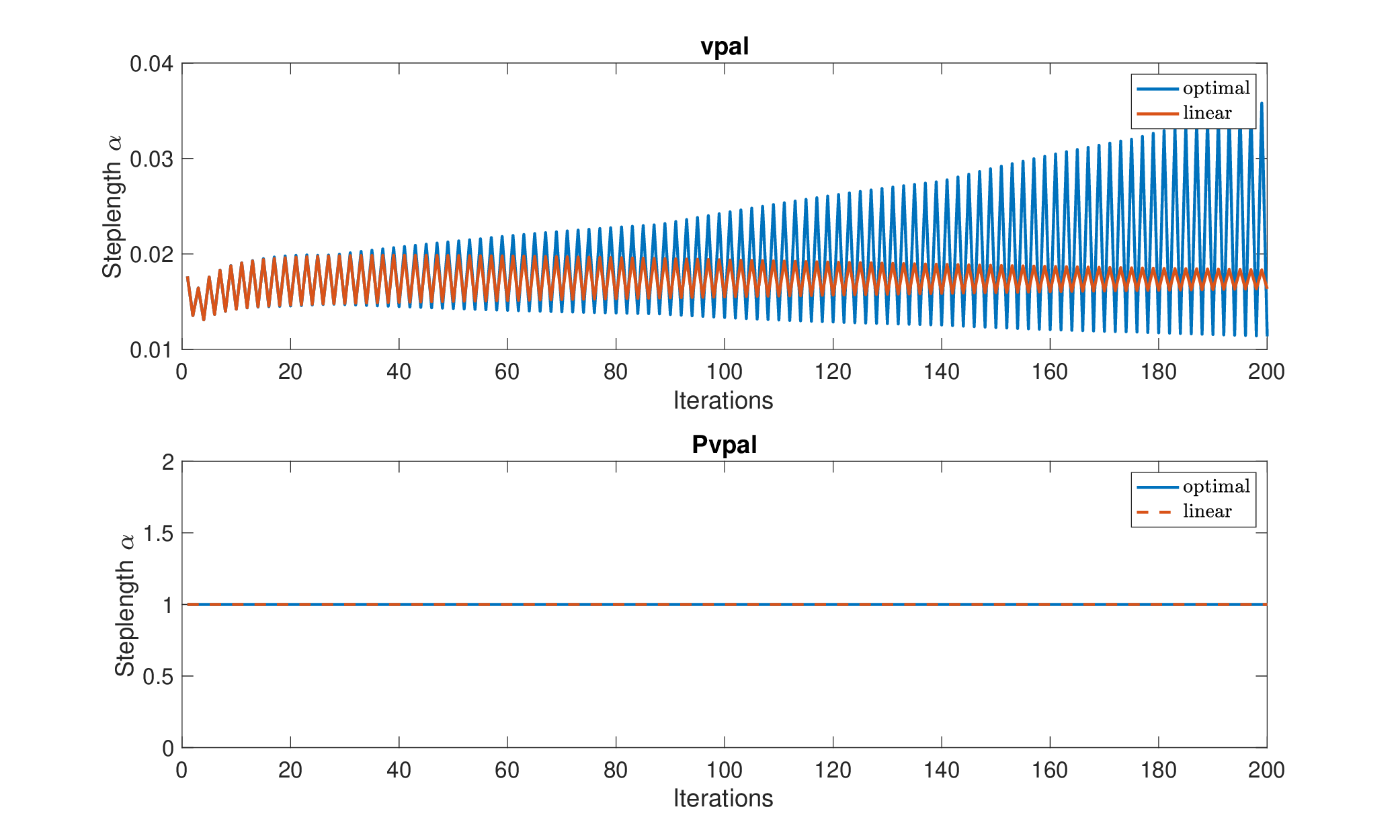}    
	\end{subfigure}
	\caption{Experiment 1: Step-size behavior of {\tt vpal} for the linearized and optimal strategies. For {\tt pvpal} the step-size is constant at $\alpha = 1$.}
	\label{fig:Peppers step size}
\end{figure}

\subsection{Experiment 2: Inpainting}\label{ssec:inpainting}
In this second example, we consider the RGB image of the ``\textit{Meisje met de parel}"  with $n =240$ and $m =205$ and address an inpainting problem using a mask that removes 85\% of the pixels from the image. Since the image has three different color channels, we apply the {\tt vpal} and {\tt pvpal} methods to each channel separately and report the numerical results obtained. \Cref{fig:inpainting setting} displays the true and the observed images.

To compare the performance of the {\tt vpal} method and its preconditioned variant {\tt pvpal}, 
we set the regularization parameter to $\mu = 10^{-2}$, 
and the augmented Lagrangian penalty parameter to $\lambda = 0.1$ for {\tt vpal} and $\lambda = 0.5$ for {\tt pvpal}. 
We use different values of $\lambda$ for {\tt vpal} and {\tt pvpal} because the preconditioner itself depends on this parameter, and we observed that a slightly higher value than the one used for {\tt vpal} yields better results. 
Each algorithm is run for $400$ iterations.

\Cref{fig:inpainting RRE and time} reports the average RRE over the three color channel obtained by each method  at every iteration, as well as its average progression over time (in seconds) both for the optimal and linearized step-size selection strategy. In this case as well, the reduction in both computation time and number of iterations is remarkable. To provide a more comprehensive analysis of the performance of the two algorithms, \Cref{tab:inpainting comparison} summarizes the average reconstruction quality over the three color channels and the average computation time required to achieve it, taking into account both step-size strategies. The first two results columns present the results obtained by the two methods after the full 400 iterations. The third result column reports the average number of iterations required by {\tt pvpal} to match the runtime of 400 iterations of {\tt vpal}, while the fourth result column shows the average number of {\tt pvpal} iterations needed to reach an RRE comparable to that of 400 {\tt vpal} iterations. Notably, only 3 iterations of the {\tt pvpal} method are sufficient to achieve the same average RRE over the three color channels, reducing the computation time from 4 seconds (for 400 iterations of {\tt vpal}) to just 0.1 seconds with its preconditioned counterpart. 

Lastly, \Cref{fig:inpainting step size} illustrates the step-length values selected at each iteration under both the linearized and optimal strategies for each color channel. 
Unlike the deblurring examples, the step-size selected by {\tt vpal} oscillate during the first $50$ iterations, before gradually stabilizing around a value close to $2$. 
In contrast, for the {\tt pvpal} algorithm the linearized approach shows the same behavior as in the previous case, with a constant sequence of ones.
The optimal strategy, however, produces mildly oscillatory step-size of approximately $1.3$.

\begin{figure}[htbp]
	\centering
	\begin{subfigure}[b]{0.40\textwidth}
		\includegraphics[width=\textwidth]{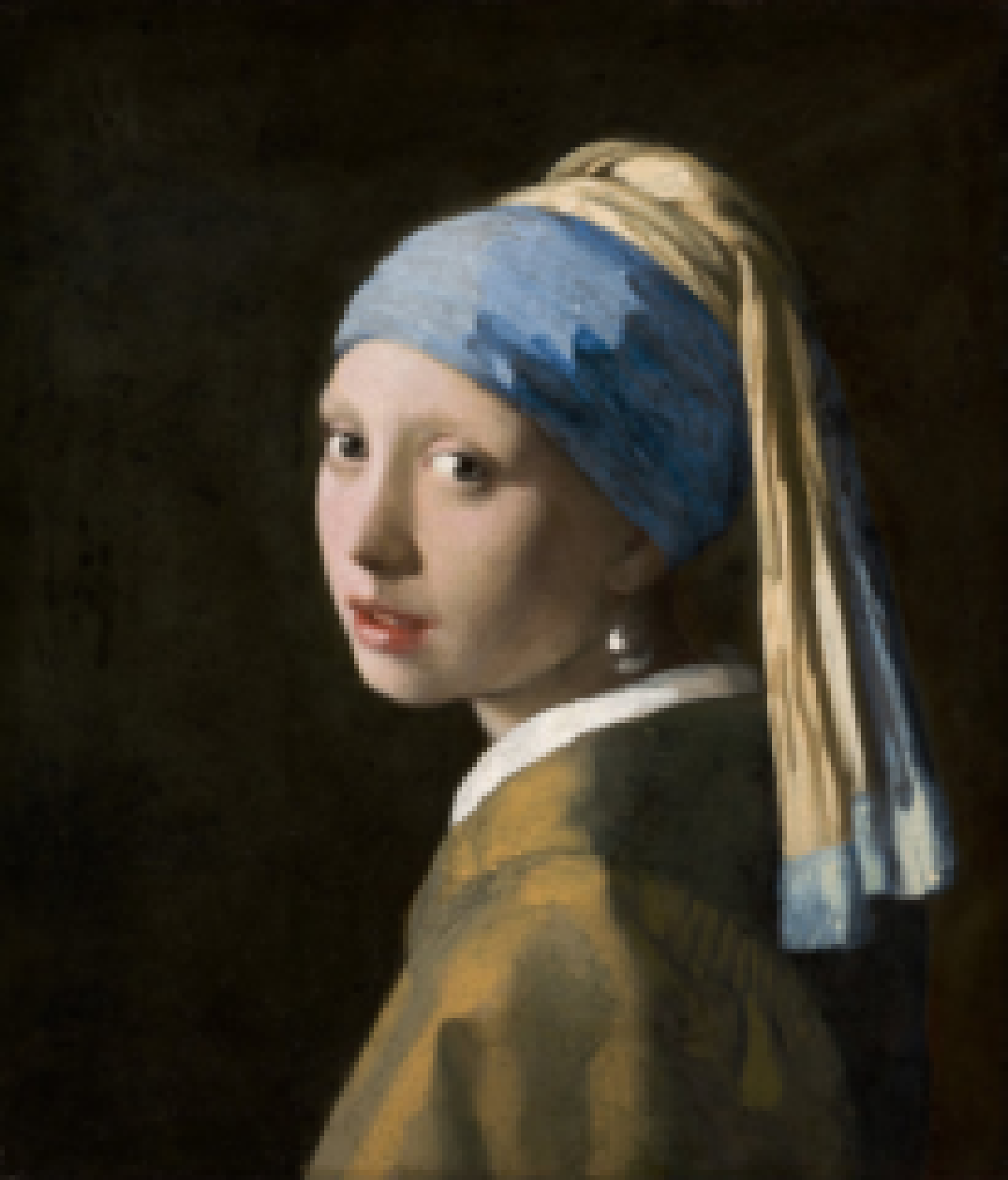}
		\caption{True image}
	\end{subfigure}
	\begin{subfigure}[b]{0.40\textwidth}
		\includegraphics[width=\textwidth]{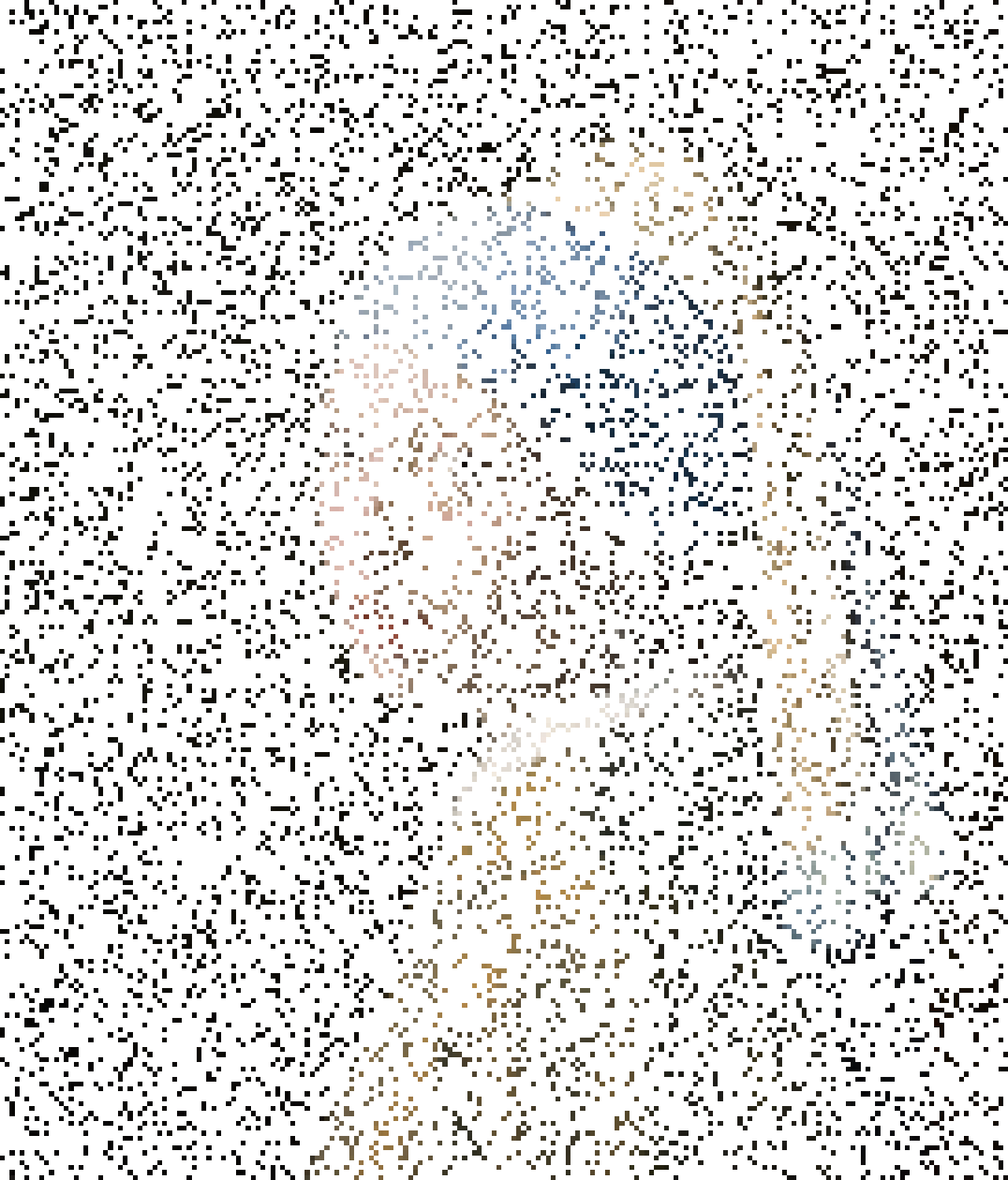}
		\caption{Observed image}
	\end{subfigure}
	\caption{Experiment 2: (a) Original image, (b) Observed image}
	\label{fig:inpainting setting}
\end{figure}

\begin{figure}[htbp]
	\centering
	\begin{subfigure}[b]{0.48\textwidth}
		\includegraphics[width=\textwidth]{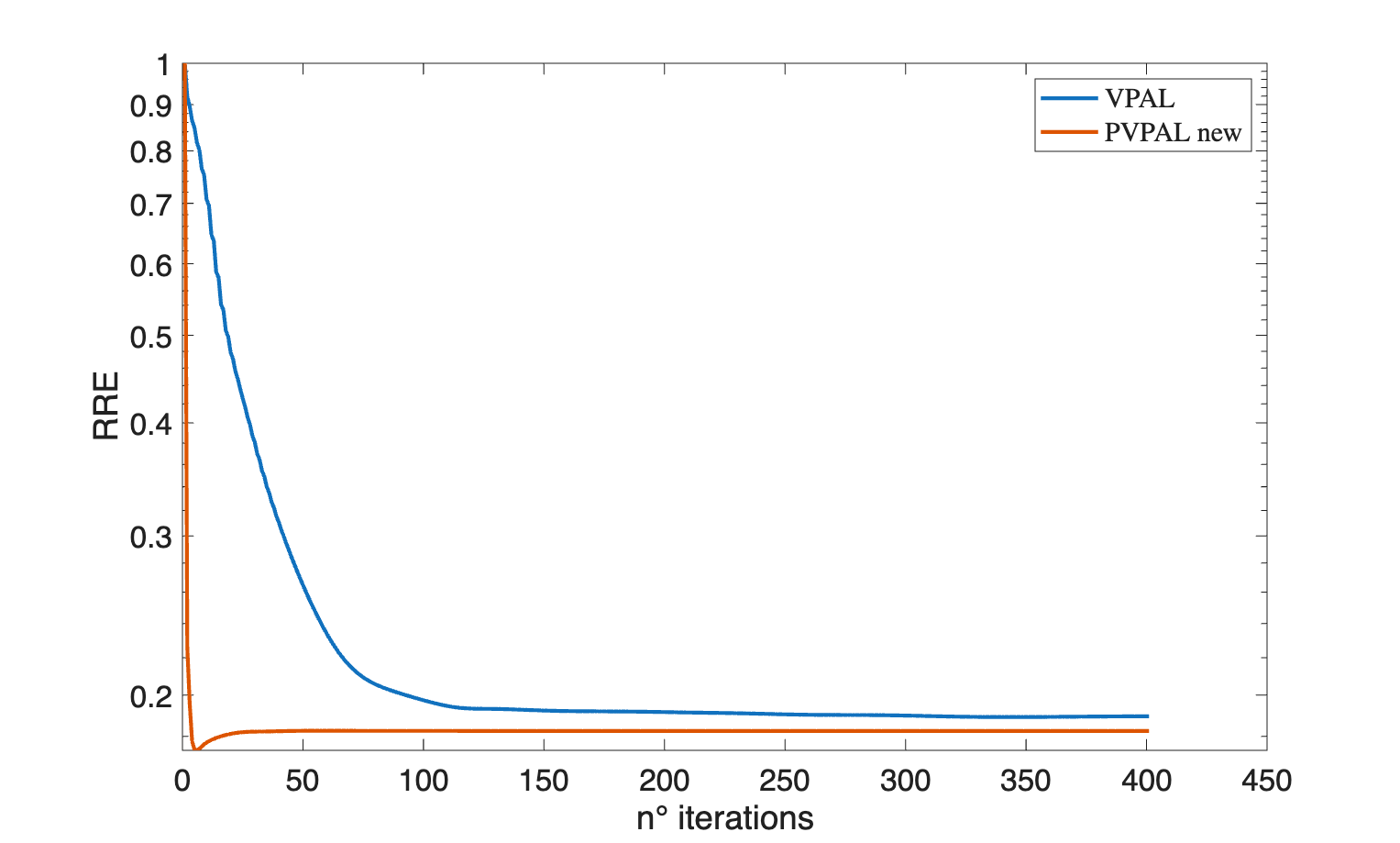}
		\caption{lin: RRE vs Iterations}
	\end{subfigure}
	\begin{subfigure}[b]{0.48\textwidth}
		\includegraphics[width=\textwidth]{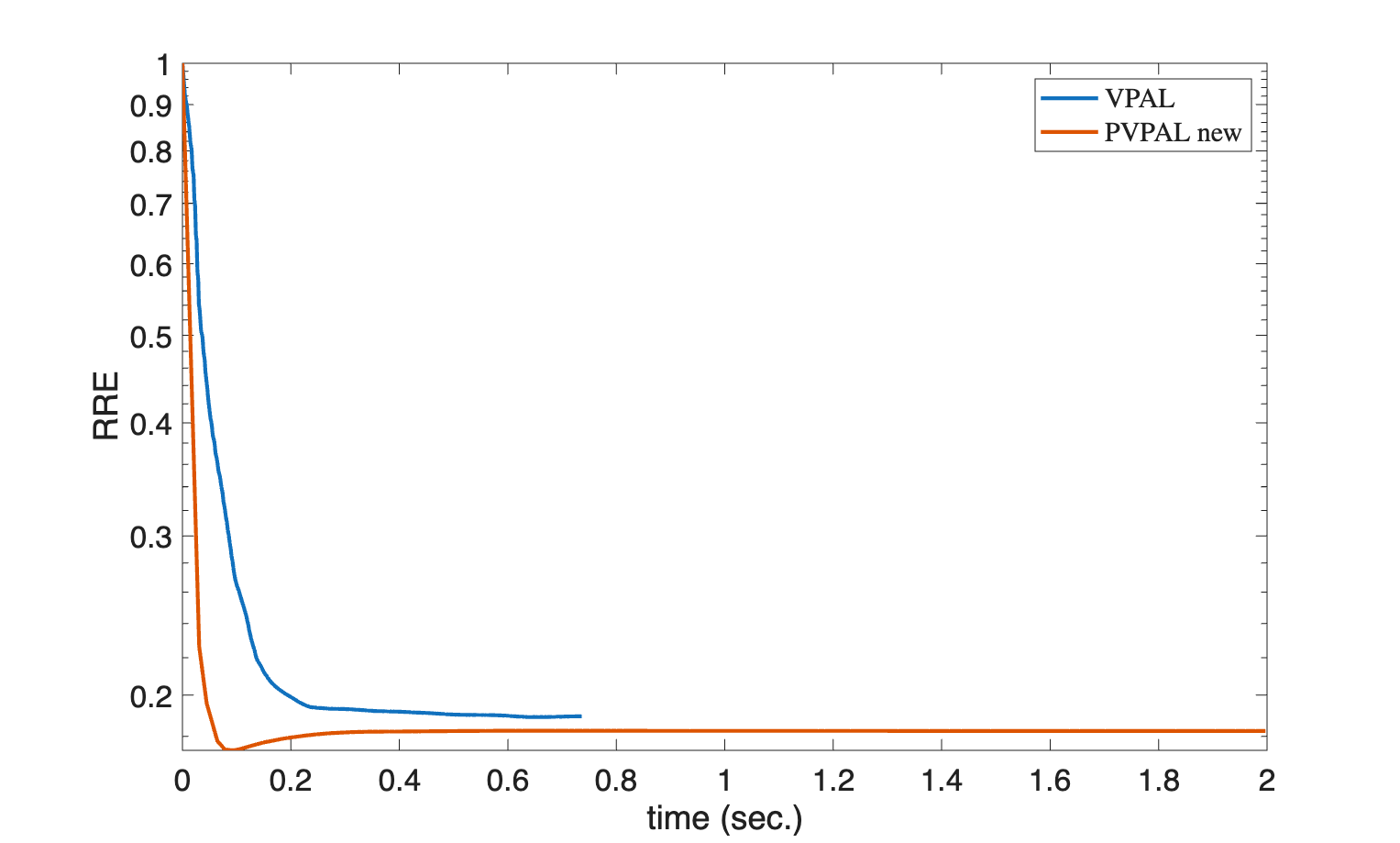}
		\caption{lin: RRE vs Time}
		
	\end{subfigure}
	\\
	\begin{subfigure}[b]{0.48\textwidth}
		\includegraphics[width=\textwidth]{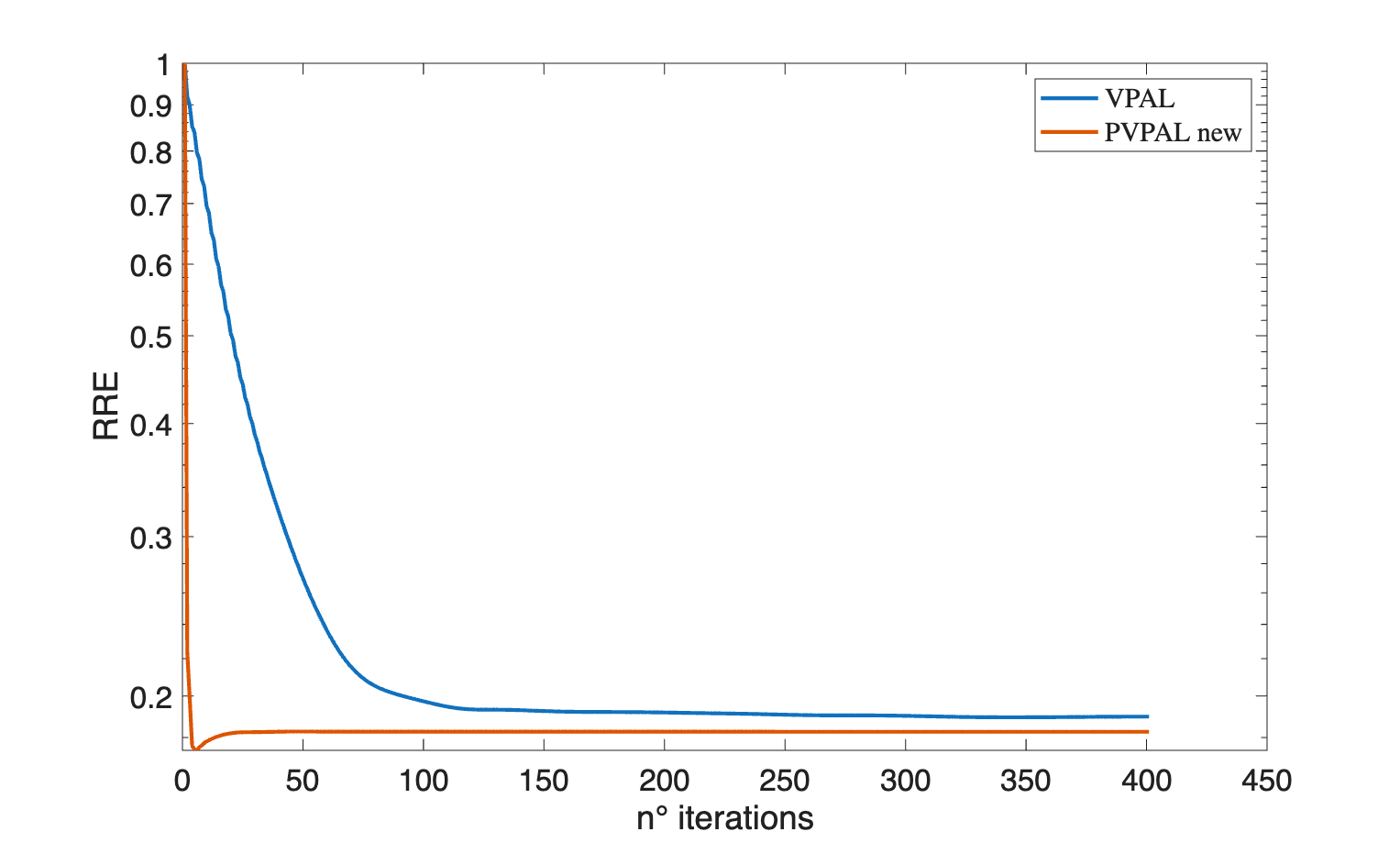}
		\caption{opt: RRE vs Iterations}
	\end{subfigure}
	\begin{subfigure}[b]{0.48\textwidth}
		\includegraphics[width=\textwidth]{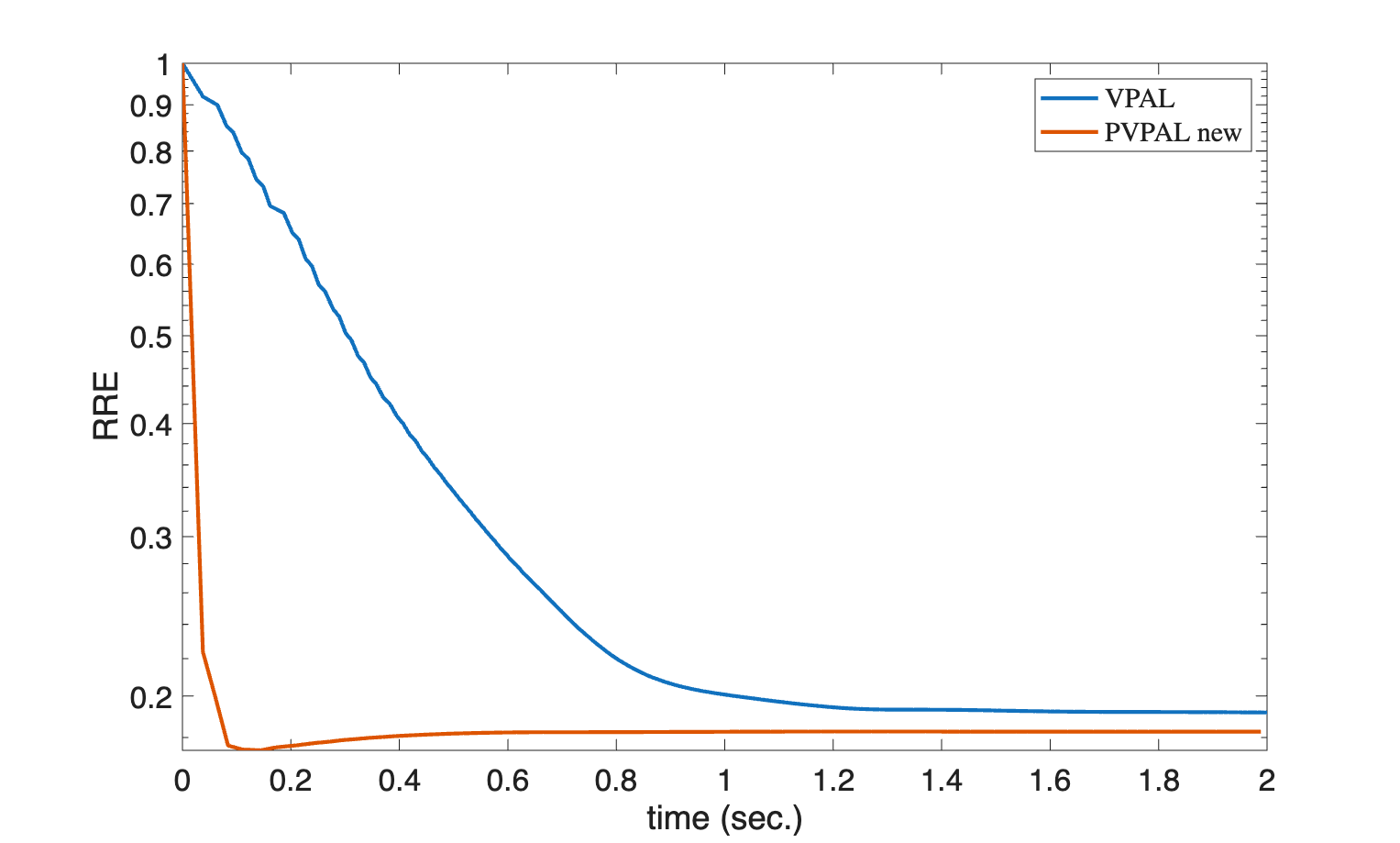}
		\caption{opt: RRE vs Time}
		
	\end{subfigure}
	\caption{Experiment 2: Comparison between the {\tt vpal} and {\tt pvpal} methods.
		The first row reports the behavior of the average RRE over the three color channels as a function of the number of iterations and the corresponding computational time (in seconds) when using the optimal step-length strategy.
		In contrast, the second row presents the same analysis for the linearized step-length strategy.}
	\label{fig:inpainting RRE and time}
\end{figure}

\begin{table}[tbhp]
	\centering \small
	\begin{tabular}{c|c|c|c|c|c|}
		\cline{2-6}
		& & {\tt vpal} (400 iter) & {\tt pvpal} (400 iter) & {\tt pvpal} (149 iter) & {\tt pvpal} (3 iter) \\ 
		\hline 
		\multirow{2}{*}{\rotatebox[origin=c]{90}{\textbf{opt}}} 
		& \textbf{RRE} & \textcolor{matlab1}{0.1898} & 0.1826 & 0.1826 & \textcolor{matlab1}{0.1763} \\ 
		\cline{2-6} & \textbf{Time (s)}
		& \textcolor{matlab2}{3.9684} & 10.018 & \textcolor{matlab2}{3.6691} & 0.1179 \\
		\hline
		& & {\tt vpal} (400 iter) & {\tt pvpal} (400 iter) & {\tt pvpal} (34 iter) & {\tt pvpal} (3 iter) \\ \hline
		\multirow{2}{*}{\rotatebox[origin=c]{90}{\textbf{lin}}} 
		& \textbf{RRE} & \textcolor{matlab1}{0.1894} & 0.1825 & 0.1822 & \textcolor{matlab1}{0.1777} \\ 
		\cline{2-6} & \textbf{Time (s)}
		& \textcolor{matlab2}{0.5813} & 5.8440 & \textcolor{matlab2}{0.5386} & 0.0720 \\
		\hline
	\end{tabular}
	\caption{Experiment 2: Comparison between the standard {\tt vpal} method and its preconditioned version, {\tt pvpal}. The top half of the table shows the average results over the three color channels using the optimal strategy to compute the step-size at each iteration, while the bottom half refers to the linearized approach. Values highlighted in \textcolor{matlab1}{blue} represent the final average RRE achieved by {\tt vpal} after 400 iterations. For {\tt pvpal}, we report the average number of iterations required to reach the same RRE and the corresponding computation time. Values highlighted in \textcolor{matlab2}{red} indicate the average time taken by {\tt vpal} to perform 400 iterations, along with the results obtained by {\tt pvpal} within approximately the same time frame.}
	\label{tab:inpainting comparison}
\end{table}

\begin{figure}[htbp]
	\centering
	\begin{center}
		\includegraphics[width=0.8\textwidth]{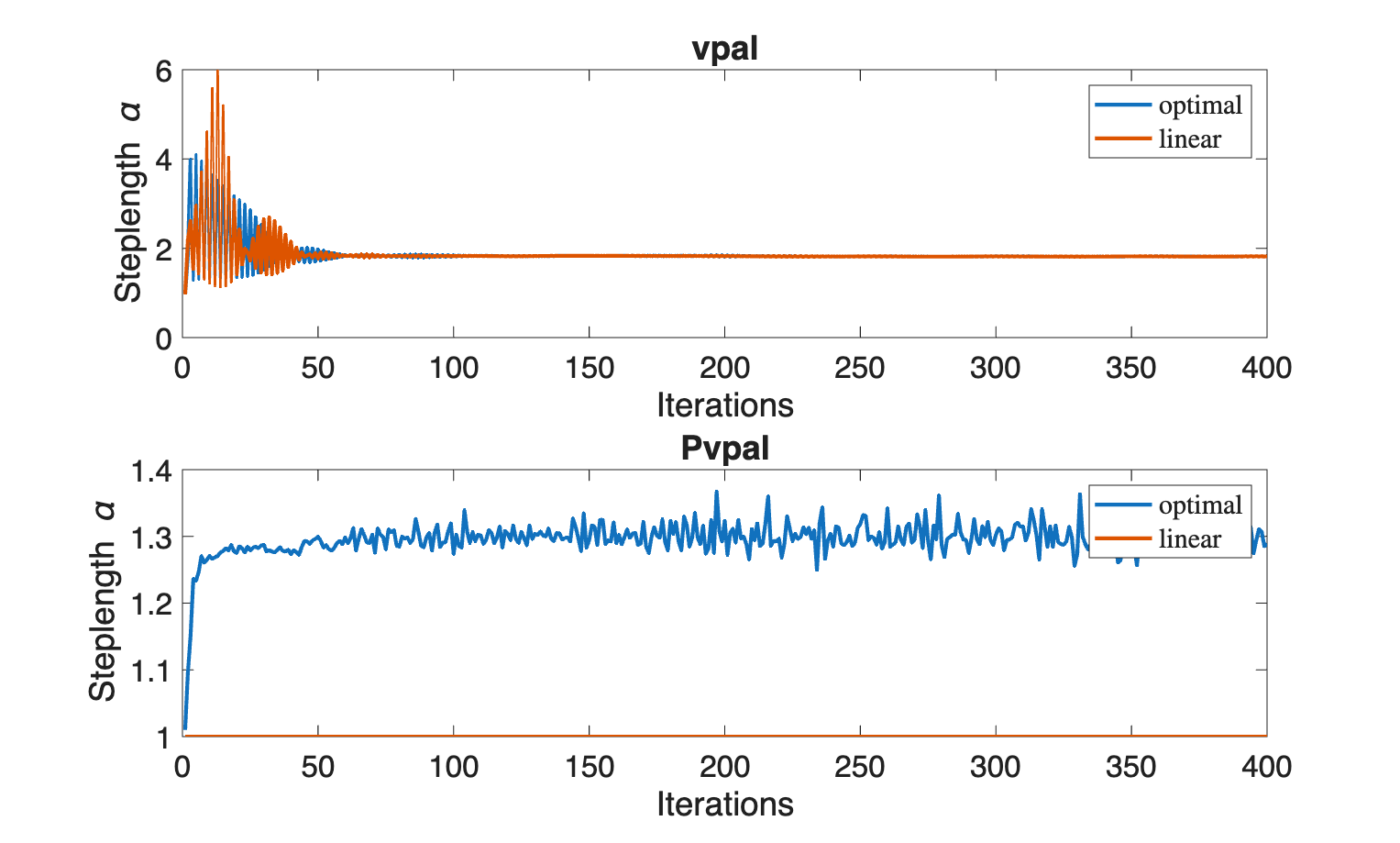}
	\end{center}
	\caption{Experiment 2: Average step-size behavior of {\tt vpal} and {\tt pvpal} for the linearized and optimal strategies.}
	\label{fig:inpainting step size}
\end{figure}

\section{Nonlinear Numerical Experiments}\label{sec:nonlinear_numEx}
We now demonstrate the performance of the proposed \texttt{vpal} and \texttt{pvpal} algorithms on two challenging real-world nonlinear inverse problems. In \Cref{sec:ptycho}, we study a phase retrieval problem arising in ptychographic imaging. In \Cref{sec:LIPCAR}, we consider the Learned Inverse Problem for Contrast Agent Reduction (LIP-CAR), a medical imaging task involving deep neural networks as forward operators.  These experiments showcase the robustness, flexibility, and efficiency of our methods in highly nonlinear and noisy settings.

Regarding the assumptions of~\Cref{rem:coercivity}: for the LIP-CAR experiment (\Cref{sec:LIPCAR}), several practical strategies exist to enforce convexity in neural networks; see, for example,~\cite{amos2017input,bianchi2023uniformly,goujon2024learning}.
Nevertheless, convex architectures often lead to suboptimal performance, and such constraints are not applied here.
Despite the lack of strict convexity or coercivity guarantees for $f_{\mathrm{joint}}$, and the use of very general regularizing terms, the \texttt{vpal} method exhibits stable and consistent behavior in all experiments.

\subsection{Experiment 4: Phase Retrieval in Ptychographic Imaging}\label{sec:ptycho}

Here, we discuss the problem of recovering a signal from short-time Fourier transform samples for which the phase information has been lost. Such phase retrieval problems commonly occur in applications such as ptychographic imaging \cite{iwen2016fast,sissouno2020direct}. For a discrete signal $x\in\mathbb{C}^N$, the (periodic) short-time Fourier transform (STFT) $F_w\colon\mathbb{C}^N\rightarrow\mathbb{C}^{N\times N}$ is defined as
\begin{equation}\label{eq:stft}
	F_w(x)=\left(\sum\limits_{n=0}^{N-1}x_nw_{(n-j)\mod N}\operatorname{e}^{2\pi \operatorname{i}\frac{kn}{N}}\right)_{j,k=0}^{N-1},
\end{equation}
where $j,k$ are the row and column indices of the matrix coefficients ranging from $0$ to $N-1$ here. The vector $w\in\mathbb{C}^N$ is a discrete window that is often real-valued and compactly supported, with $w_n=0$ for $n\geq K\ll N$. The forward operator for the phase retrieval problem now reads as
\begin{align}\label{eq:ptycho}
	\widetilde{A}\colon\mathbb{C}^N\rightarrow\mathbb{R}^{N\times N}, && \widetilde{A}(x)=|F_w(x)|^2,
\end{align}
where $|\cdot|^2$ is applied element-wise. The operator $\widetilde{A}$ is nonlinear but can be linearized by lifting into a higher-dimensional space \cite{iwen2016fast,sissouno2020direct}. We compare phase retrieval using \texttt{vpal}/\texttt{pvpal} against the linearized approach presented in \cite{sissouno2020direct}. Since $|\cdot|^2$ is not complex differentiable, we split $x$ into its amplitude and phase components and define the operator $T\colon\mathbb{R}^N\times\mathbb{R}^N\rightarrow\mathbb{C}^N$ as
\begin{equation}
	T(x^\textnormal{amp},x^\textnormal{phase})=\left(x^\textnormal{amp}_k\operatorname{e}^{\operatorname{i}x_k^\textnormal{phase}}\right)_{k=0}^{N-1}.
\end{equation}
Then we can define a real-variable differentiable forward operator $A\colon\mathbb{R}^N\times\mathbb{R}^N\rightarrow\mathbb{R}^{N\times N}$ as
\begin{equation}
	A(x^\textnormal{amp},x^\textnormal{phase})=\widetilde{A}(T(x^\textnormal{amp},x^\textnormal{phase})).
\end{equation}
A common strategy to reduce measurement time in applications is to only sample every $s$-th shift of the short-time Fourier transform, where $1\leq s< K$. To model this, we define the subsampling operator $P_s\colon\mathbb{R}^{N\times N}\rightarrow\mathbb{R}^{\lceil\frac{N}{s}\rceil\times N}$ with $P_s(x)=\left(x_{js,k}\right)_{j,k=0}^{\lceil\frac{N}{s}\rceil-1, N-1}$, i.e., sampling only every $s$-th row. We then define the family of forward operators $A_s\colon\mathbb{R}^N\times\mathbb{R}^N\rightarrow\mathbb{R}^{\lceil\frac{N}{s}\rceil\times N}$
\begin{equation}
	A_s(x^\textnormal{amp},x^\textnormal{phase})=P_s(A(x^\textnormal{amp},x^\textnormal{phase})),
\end{equation}
where $A_1=A$.

In preliminary experiments, \texttt{vpal}, used with a finite-difference
	regularization operator $D$ and a linear step-size rule, often converged to the
	true solution. However, the gradient exhibited oscillatory behavior when the
	iterates approached a minimum. Counteracting this behavior required very small
	step sizes, which led to very slow convergence. Hence, we implemented the
	modified step-size rule introduced in the next subsection.

\subsubsection{Step size choice}

For a given point $x^\textnormal{amp},x^\textnormal{phase}\in\mathbb{R}^N$ and gradient $g^\textnormal{amp},g^\textnormal{phase}\in\mathbb{R}^N$, we seek the optimal step size $\widetilde{\alpha}>0$ that minimizes
\begin{equation}\label{eq:optstepsize}
	f_{\mu}(\alpha)=\tfrac{1}{2} \left\| A_s(x^\textnormal{amp}-\alpha g^\textnormal{amp},x^\textnormal{phase}-\alpha g^\textnormal{phase}) - b \right\|_2^2
	+ \mu \left\| D\begin{pmatrix}x^\textnormal{amp}-\alpha g^\textnormal{amp}\\x^\textnormal{phase}-\alpha g^\textnormal{phase}\end{pmatrix} \right\|_1.
\end{equation}
If we temporarily ignore the phase-gradient contribution to the data-fidelity term and use the linearity of the Fourier transform, we can show that
\begin{equation}\label{eq:polyApproximation}
	A_s(x^\textnormal{amp}-\alpha g^\textnormal{amp},x^\textnormal{phase})
	=
	P_s\left(\left|F_w(T(x^\textnormal{amp},x^\textnormal{phase}))-\alpha F_w(T( g^\textnormal{amp},x^\textnormal{phase}))\right|^2\right),
\end{equation}
which is a polynomial of degree $2$ in $\alpha$. Hence, the data fidelity term in \Cref{eq:optstepsize} can be approximated by a polynomial of degree $4$ if the impact of the phase update is small compared to the amplitude update (which aligns with our observations). Furthermore, the regularizer in \Cref{eq:optstepsize} behaves linearly whenever we are close to the true solution. This inspires the following strategy:

Sample $f_{\mu}(\alpha)$ for different values of $\alpha$. Based on these samples, fit a polynomial of degree 4 to these samples. Calculate the smallest real non-negative critical point $\widetilde{\alpha}$ of this polynomial. As we expect $f_{\mu}(\alpha)$ to decrease, this should be a minimum. If no such critical point exists or if it is a maximum, we repeat this process with slight variations: First, we sample $f_{\mu}(\alpha)$ for small values of $\alpha=0,0.001,0.005,0.01,0.1$ as we expect the optimal step size to be small in general. If this does not succeed, we add another sampling point $\alpha=1$ to search for a larger step size. Finally, we enforce convexity of the approximation polynomial by adding the (nonlinear) constraints $a_4>0$ and $3a_3^{2}-8a_4a_2<0$ for the coefficients of the polynomial. (These constraints ensure that the second derivative, which is a quadratic polynomial, is positive.) If all three approaches fail, we fall back to a default step size of $\widetilde{\alpha}=0.001$. Altogether, we obtain the following algorithm:

\begin{algorithm}[H]
	\caption{Phase retrieval step-size choice}
	\label{alg:phase-step}
	\begin{algorithmic}[1]\small
		\Require
		\Statex $A_s(x^\textnormal{amp},x^\textnormal{phase})$ computed by \texttt{vpal} or \texttt{pvpal},
		\Statex Parameter $\mu$ from \texttt{vpal}/\texttt{pvpal} and $\widetilde{\mu}$ from previous call of this function,
		\Statex Function $f_{\widetilde{\mu}}$ 
		\Statex
		\If{first function call} $\widetilde{\mu}=0$ \Comment{Include/Exclude regularizer}\EndIf
		\If{$\widetilde{\mu}=0$ and \texttt{vpal}/\texttt{pvpal} approximation error increased} $\widetilde{\mu}=\mu$ \EndIf
		\Statex
		\State Calculate $f_{\widetilde{\mu}}(0)$ from $A_s(x^\textnormal{amp},x^\textnormal{phase})$ \Comment{search small step size}
		\State Sample $f_{\widetilde{\mu}}(\alpha)$ for $\alpha=0.001,0.005,0.01,0.1$
		\State Approximate $f_{\widetilde{\mu}}(\alpha)\approx\textnormal{pol}(\alpha)$ by a degree 4 polynomial 
		\State Find smallest real non-negative root $\widetilde{\alpha}$ of $\textnormal{pol}'(\alpha)$ 
		\If{$\widetilde{\alpha}$ exists and $\textnormal{pol}''(\widetilde{\alpha})>0$} \Return $\widetilde{\alpha}$ \EndIf
		\Statex
		\State Add additional sampling point $f_{\widetilde{\mu}}(1)$ \Comment{search large step size}
		\State Approximate $f_{\widetilde{\mu}}(\alpha)\approx\textnormal{pol}(\alpha)$ by a degree 4 polynomial 
		\State Find smallest real non-negative root $\widetilde{\alpha}$ of $\textnormal{pol}'(\alpha)$ 
		\If{$\widetilde{\alpha}$ exists and $\textnormal{pol}''(\widetilde{\alpha})>0$} \Return $\widetilde{\alpha}$ \EndIf
		\Statex 
		\State Approximate $f_{\widetilde{\mu}}(\alpha)\approx\textnormal{pol}(\alpha)=a_0+a_1\alpha+a_2\alpha^2+a_3\alpha^3+a_4\alpha^4$ \Comment{Force convexity}
		\Statex Force convexity by requiring $a_4>0$ and $3{a_3^{2}}-8a_4a_2<0$
		\State Find smallest real non-negative root $\widetilde{\alpha}$ of $\textnormal{pol}'(\alpha)$ 
		\If{$\widetilde{\alpha}$ exists and $\textnormal{pol}''(\widetilde{\alpha})>0$} \Return $\widetilde{\alpha}$ \EndIf
		\Statex
		\State \Return $\widetilde{\alpha}=0.001$.\Comment{default step size}
	\end{algorithmic}
\end{algorithm}

In lines $2-3$, the auxiliary variable $\widetilde\mu$ is set to $0$
	during the initial iterations and later updated to $\widetilde\mu=\mu$. This ensures that the regularizer in \Cref{eq:optstepsize} is ignored for the first iterations where it has a more nonlinear behavior and only taken into account once it behaves more linear.

This strategy is designed to minimize the number of computational expensive function calls of the forward-operator $A_s$. Additional sampling points can be added to stabilize the method if required. Alternatively, replacing the data fidelity term in \Cref{eq:optstepsize} with the polynomial approximation \Cref{eq:polyApproximation} reduces the number of additional forward-operator calls to just one. However, the obtained step size is less accurate and increases the number of iterations such that it did not lead to any advantage in our experiments.

\subsubsection{Experimental setup}

We run \texttt{vpal} and \texttt{pvpal} with $\mu=0.05$, $\lambda=0.5$, and set $D$ to the finite difference operator. To keep the runtime small, we use the preconditioner only during the first $20$ iterations of \texttt{pvpal} during which it has the most significant effects. The achieved results after $2000$ iterations are compared with the linearized approach. The signal size is set to $N=100$, i.e., $x\in\mathbb{C}^{100}$ and $x^\textnormal{amp},x^\textnormal{phase}\in\mathbb{R}^{100}$. The test-signal set consist of $1000$ piecewise-constant vectors with $10$ jumps at randomly chosen locations. Each interval has a random complex value drawn from a Gaussian distribution. The short-time Fourier transform window $w\in\mathbb{C}^{100}$ has a support length of $K=10$. We use an exponential window $w^\textnormal{exp}$ and a Gaussian window $w^\textnormal{gauss}$ defined as
\begin{align}
	w_k^\textnormal{exp}=
	\begin{cases}
		\operatorname{e}^{-\frac{5}{9}k}, & 0 \leq k<10 \\
		0, & \textnormal{otherwise}
	\end{cases}, &&
	w_k^\textnormal{gauss}=
	\begin{cases}
		\operatorname{e}^{-\left(\frac{5}{9}k-\frac{5}{2}\right)^2}, & 0 \leq k<10 \\
		0, & \textnormal{otherwise}
	\end{cases}.
\end{align}
The exponential window yields a stable forward operator for the linearized approach with STFT shift $s=1$ \cite{iwen2016fast}. The Gaussian window is closer to the application but ill-posed \cite{melnyk2021stable}.

We note that $x$ can only be recovered up to a global phase shift from \eqref{eq:ptycho}. Moreover, the phase is $2\pi$-periodic. Hence, we measure the phase error of the reconstructed signal by computing the mean squared error of the phase differences where each value is mapped to its $2\pi$-periodic equivalent with smallest absolute value:
\begin{align}\label{eq:phaseMSE}
	\textnormal{PhaseMSE}(x^\textnormal{approx}) =&
	\frac{1}{N-1}\sum\limits_{k=1}^{N-1}\textnormal{PerMin}\left(
	\Delta_{k}^{\textnormal{phase}} x^{\textnormal{org}}
	-
	\Delta_{k}^{\textnormal{phase}} x^{\textnormal{approx}}\right)\\
	\Delta_{k}^{\textnormal{phase}} x =& x_k^{\textnormal{phase}}-x_{k-1}^{\textnormal{phase}} \\
	\textnormal{PerMin}(v) =& \min_{k\in\mathbb{Z}}|v-2k\pi|.
\end{align}

\subsubsection{\texttt{vpal} and \texttt{pvpal} evaluation}

First, we note that both \texttt{vpal} and \texttt{pvpal} fail to converge for a small number of test samples. Without preconditioner, \texttt{vpal} fails in $0.5\%$ (exponential window) and $15.6\%$ (Gaussian window) of the samples. Using the proposed preconditioner strategy, \texttt{pvpal} fails in $0.5\%$ (exponential window) and only $1.4\%$ (Gaussian window) of the samples, i.e., for the ill-posed Gaussian window \texttt{pvpal} is much more stable. In all cases, restarting the method with a different random starting guess can resolve the issue. The failed approaches are not taken into account for the following error statistics.

In \Cref{fig:PR_MSE} we show the Mean Squared Error (MSE) using \texttt{vpal} and \texttt{pvpal} plotted against the STFT shift $s$ used. The MSE is shown for the reconstructed amplitude $x^\textnormal{amp}$, phase $x^\textnormal{phase}$ (see \Cref{eq:phaseMSE}), and for the residual. \Cref{fig:PR_MSE} (a,b,c) shows the MSE for reconstructions based on the exponential window while \Cref{fig:PR_MSE} (d,e,f) uses data obtained with the Gaussian window. In both cases \texttt{vpal} and \texttt{pvpal} outperform the linear approach for the amplitude and residual MSE whenever the STFT shift $s>1$. Only for the Gaussian window does the linearized approach perform better in phase reconstruction.
The two methods produce
	comparable results; \texttt{pvpal} has slightly lower phase and residual
	errors but a higher variance in the amplitude.

\begin{figure}[htbp]
	\centering
	\begin{subfigure}[b]{0.24\textwidth}
		\includegraphics[width=\textwidth]{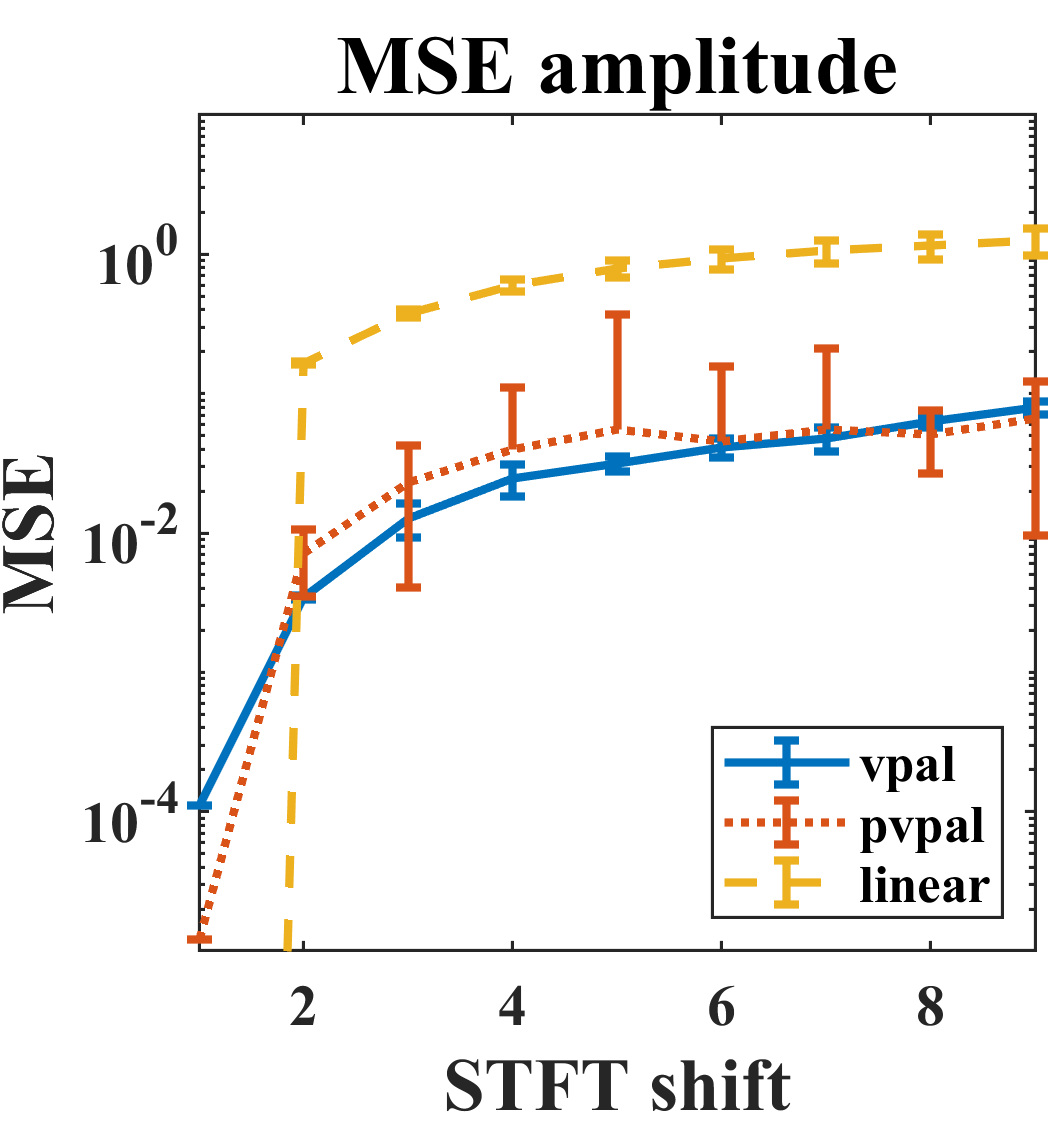}
		\caption{}
	\end{subfigure}
	\begin{subfigure}[b]{0.24\textwidth}
		\includegraphics[width=\textwidth]{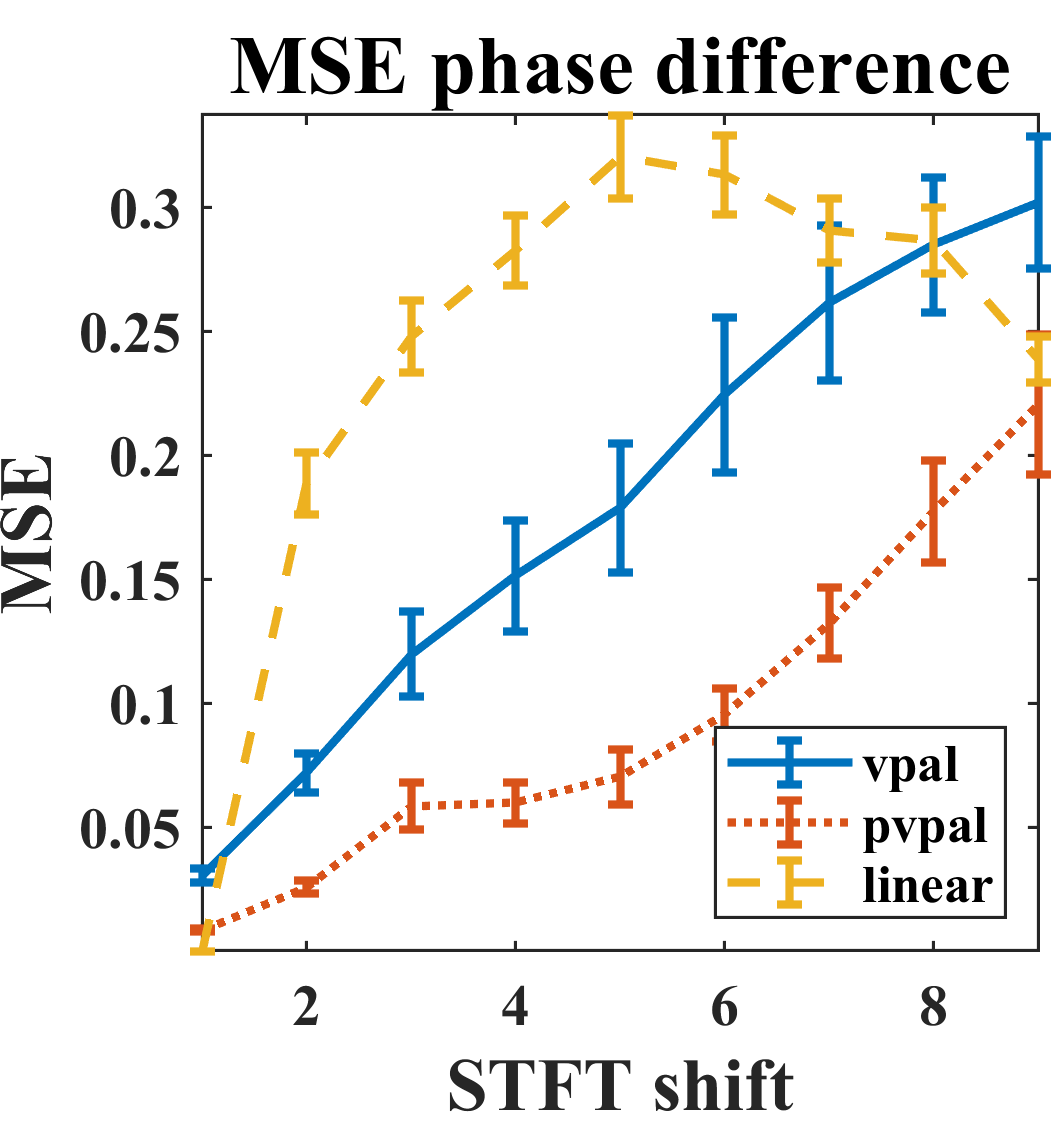}
		\caption{}
	\end{subfigure}
	\begin{subfigure}[b]{0.24\textwidth}
		\includegraphics[width=\textwidth]{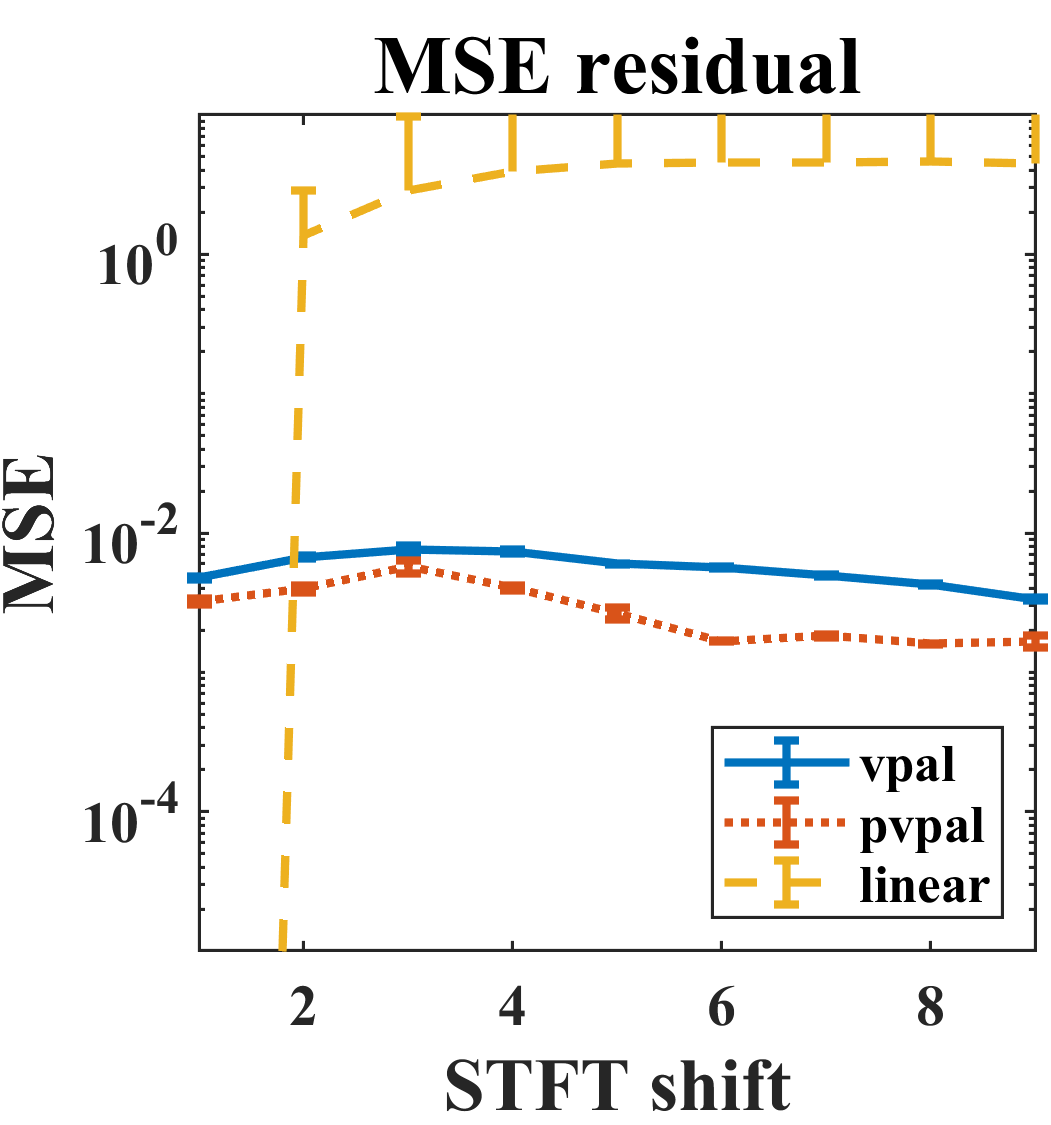}
		\caption{}
	\end{subfigure}\\
	\begin{subfigure}[b]{0.24\textwidth}
		\includegraphics[width=\textwidth]{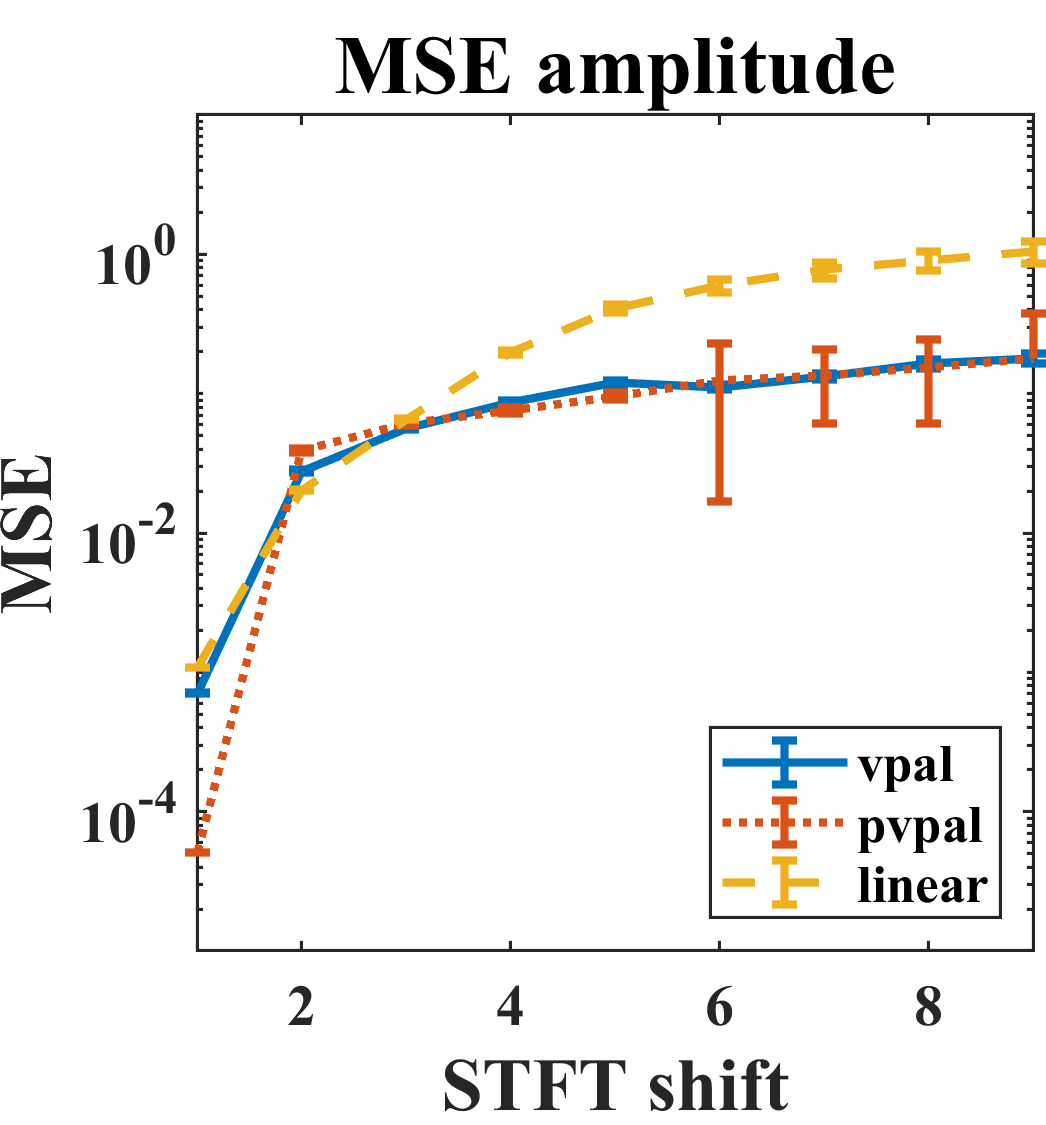}
		\caption{}
	\end{subfigure}
	\begin{subfigure}[b]{0.24\textwidth}
		\includegraphics[width=\textwidth]{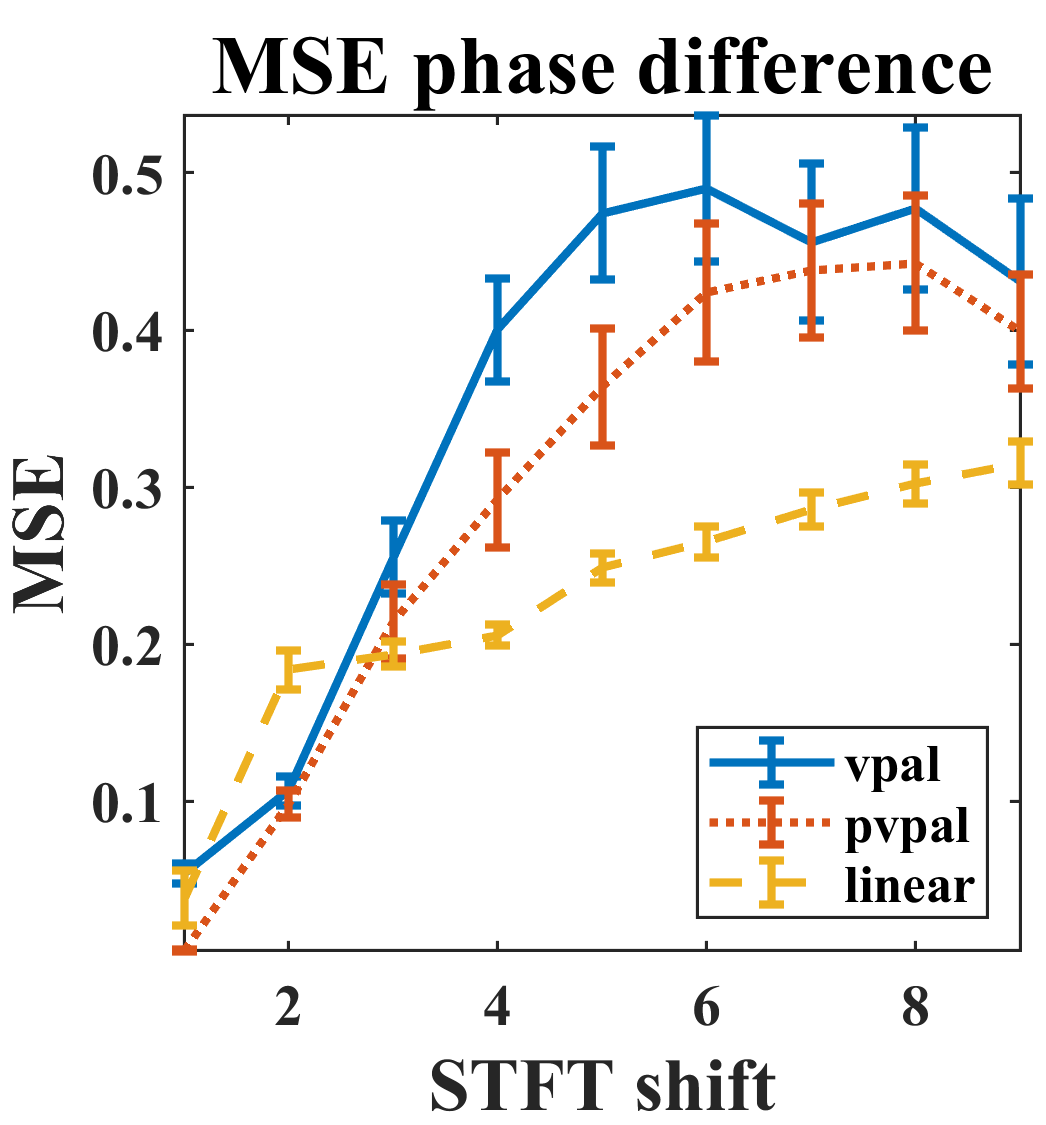}
		\caption{}
	\end{subfigure}
	\begin{subfigure}[b]{0.24\textwidth}
		\includegraphics[width=\textwidth]{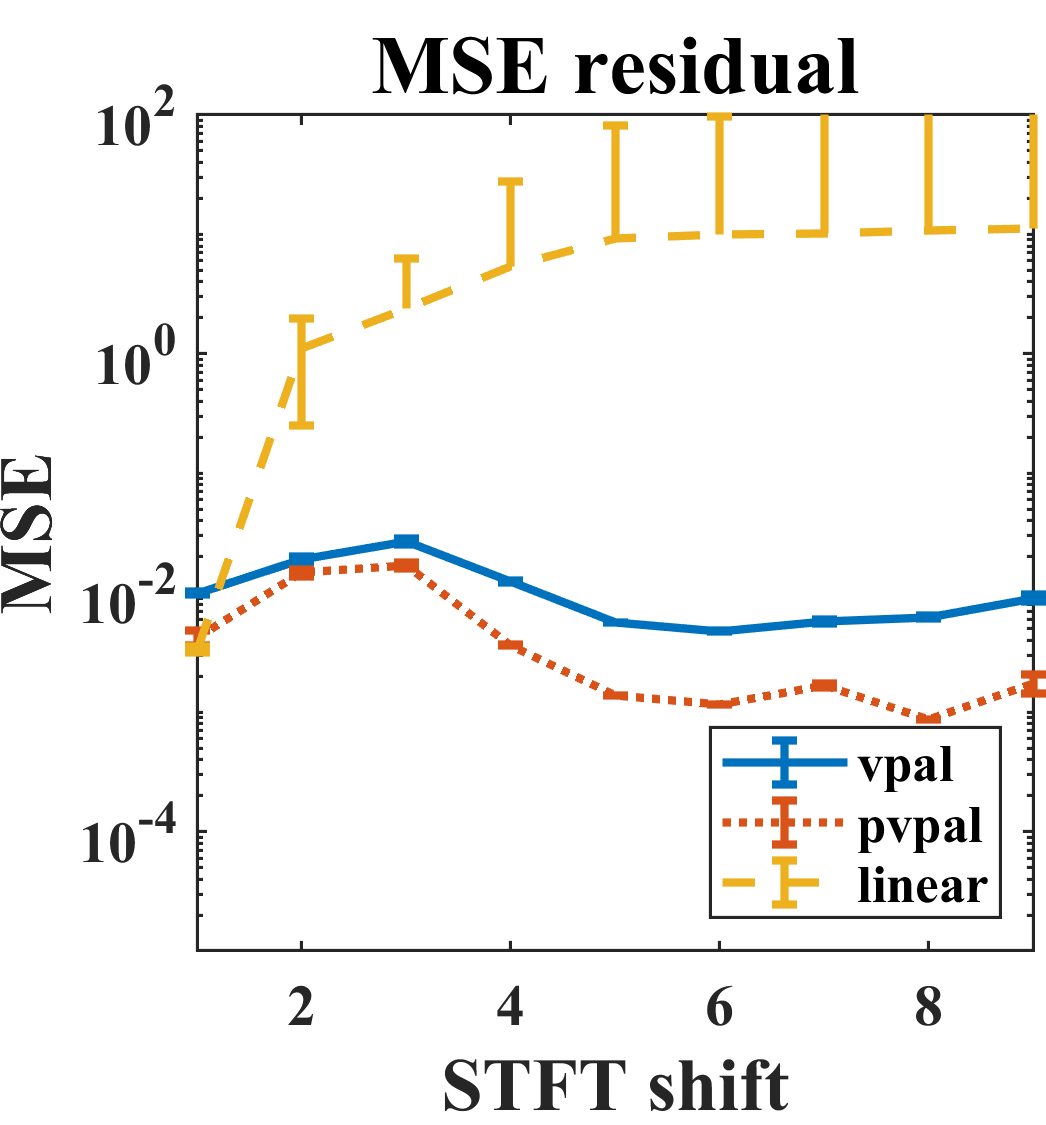}
		\caption{}
	\end{subfigure}
	\caption{Experiment 4: Mean Squared Error of amplitude (a,d) and phase (b,e) for the reconstructed solution as well as Mean Squared Error of the residual (c,f) plotted against STFT shift $s$, where (a,b,c) use the STFT with exponential window and (d,e,f) apply the Gaussian window.}
	\label{fig:PR_MSE}
\end{figure}

To have a more detailed comparison between \texttt{vpal} and \texttt{pvpal}, we present the achieved minimization error for each iteration in \Cref{fig:PR_iteration} for both the exponential and Gaussian window. We show the error for a STFT shift $s=1$ and $s=9$ exemplary, the other cases have a similar behavior. Although both strategies converge, \texttt{pvpal} converges much faster. \texttt{vpal} requires all 2000 iterations, whereas \texttt{pvpal} stabilizes after about 50 iterations as the zoomed in segments in \Cref{fig:PR_iteration} show. The oscillations that can be observed especially in the first few iterations stem from test samples with missteps. Since \texttt{pvpal} is more likely to recover from such missteps than \texttt{vpal} the oscillations are more present here.

\begin{figure}[htbp]
	\centering
	\begin{subfigure}[b]{0.24\textwidth}
		\includegraphics[width=\textwidth]{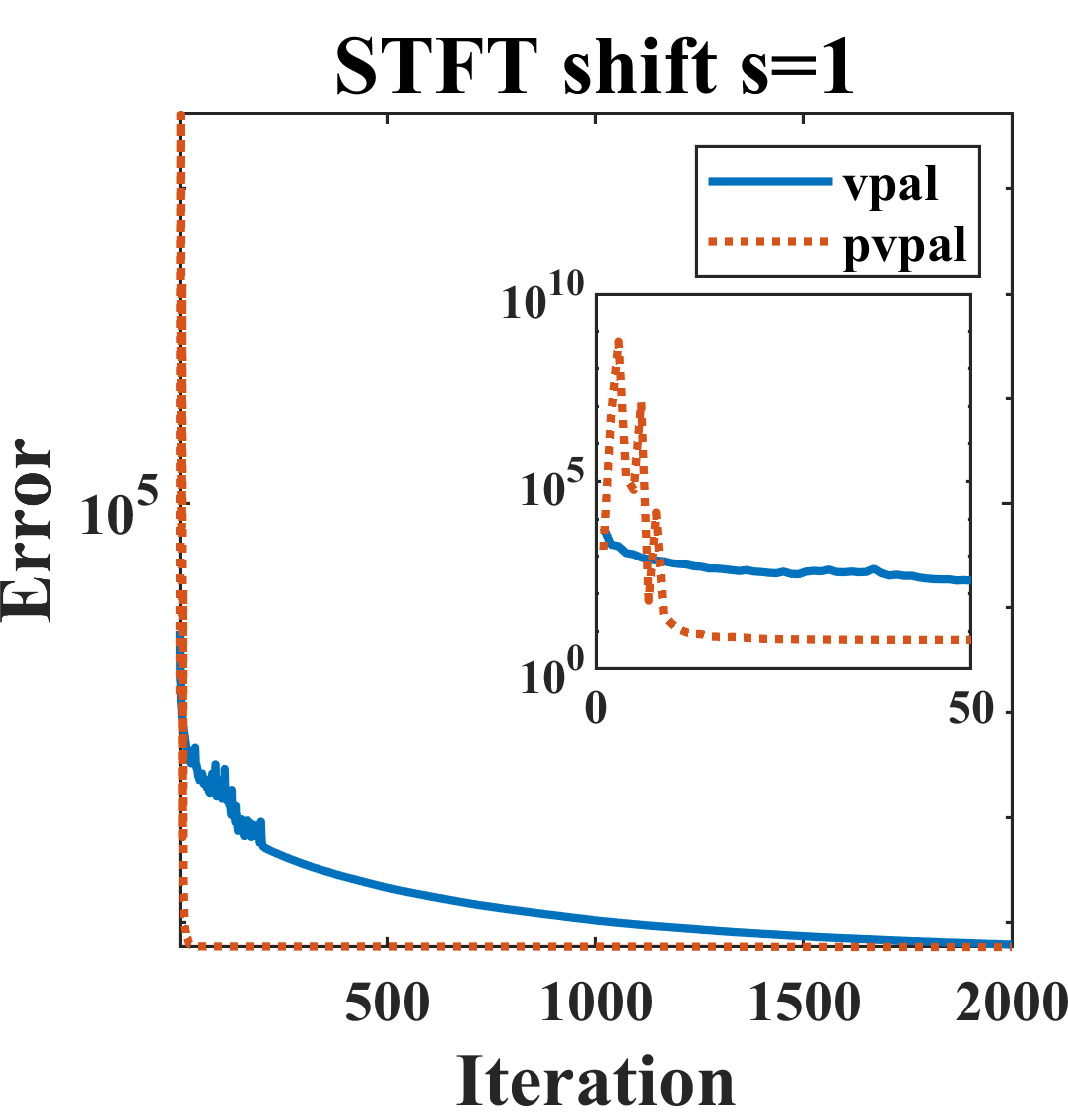}
		\caption{}
	\end{subfigure}
	\begin{subfigure}[b]{0.24\textwidth}
		\includegraphics[width=\textwidth]{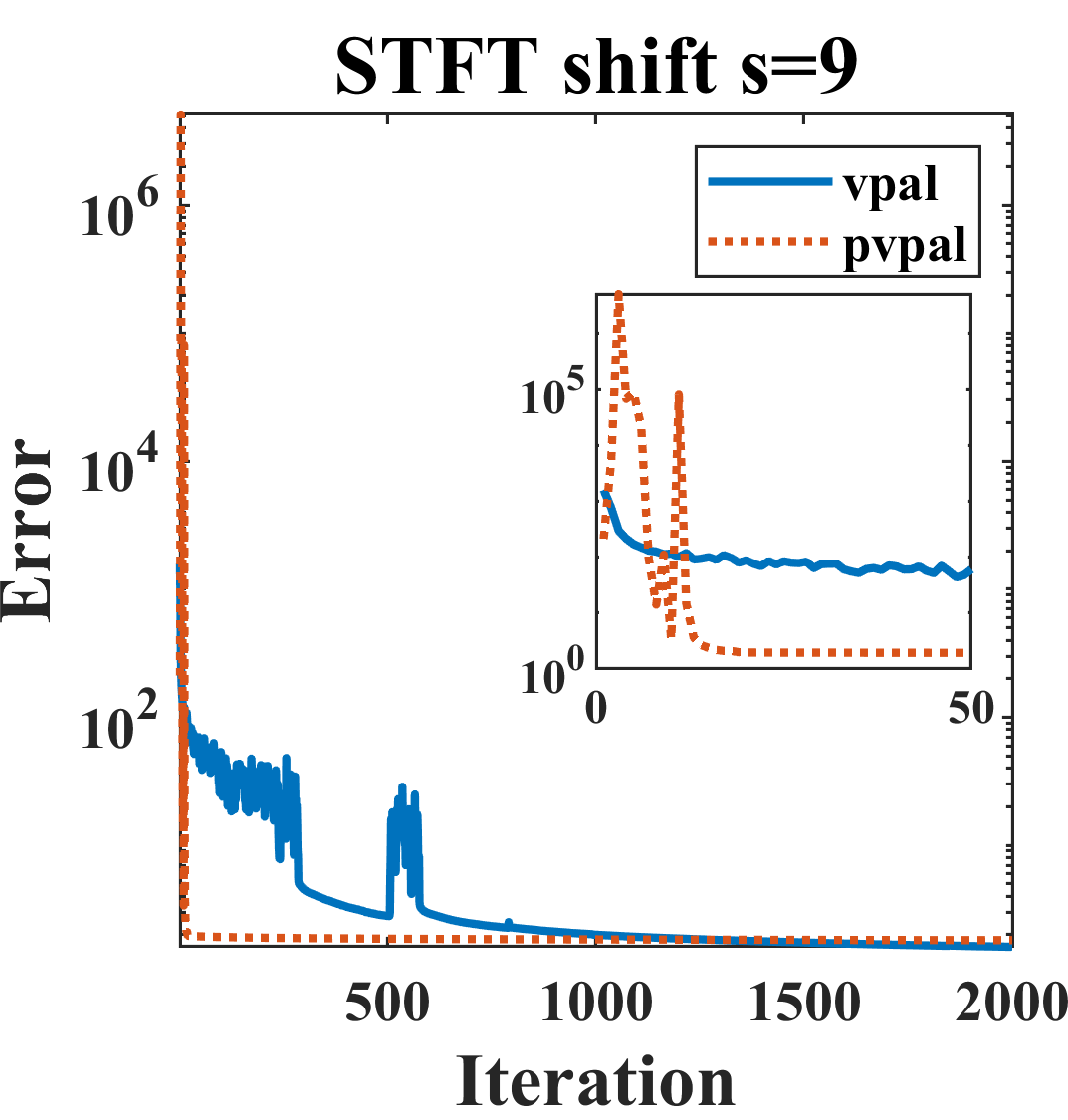}
		\caption{}
	\end{subfigure}
	\begin{subfigure}[b]{0.24\textwidth}
		\includegraphics[width=\textwidth]{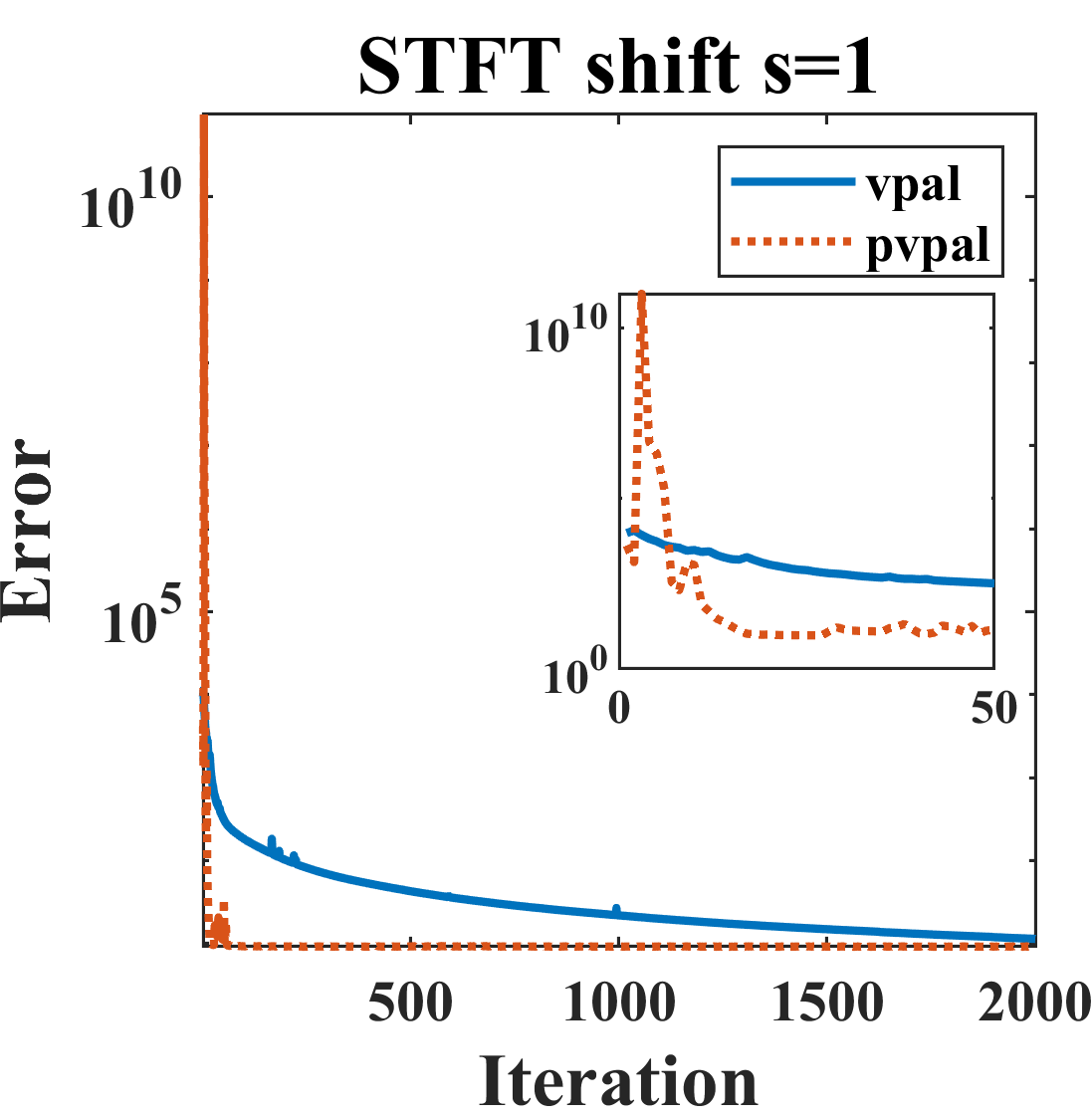}
		\caption{ }
	\end{subfigure}
	\begin{subfigure}[b]{0.24\textwidth}
		\includegraphics[width=\textwidth]{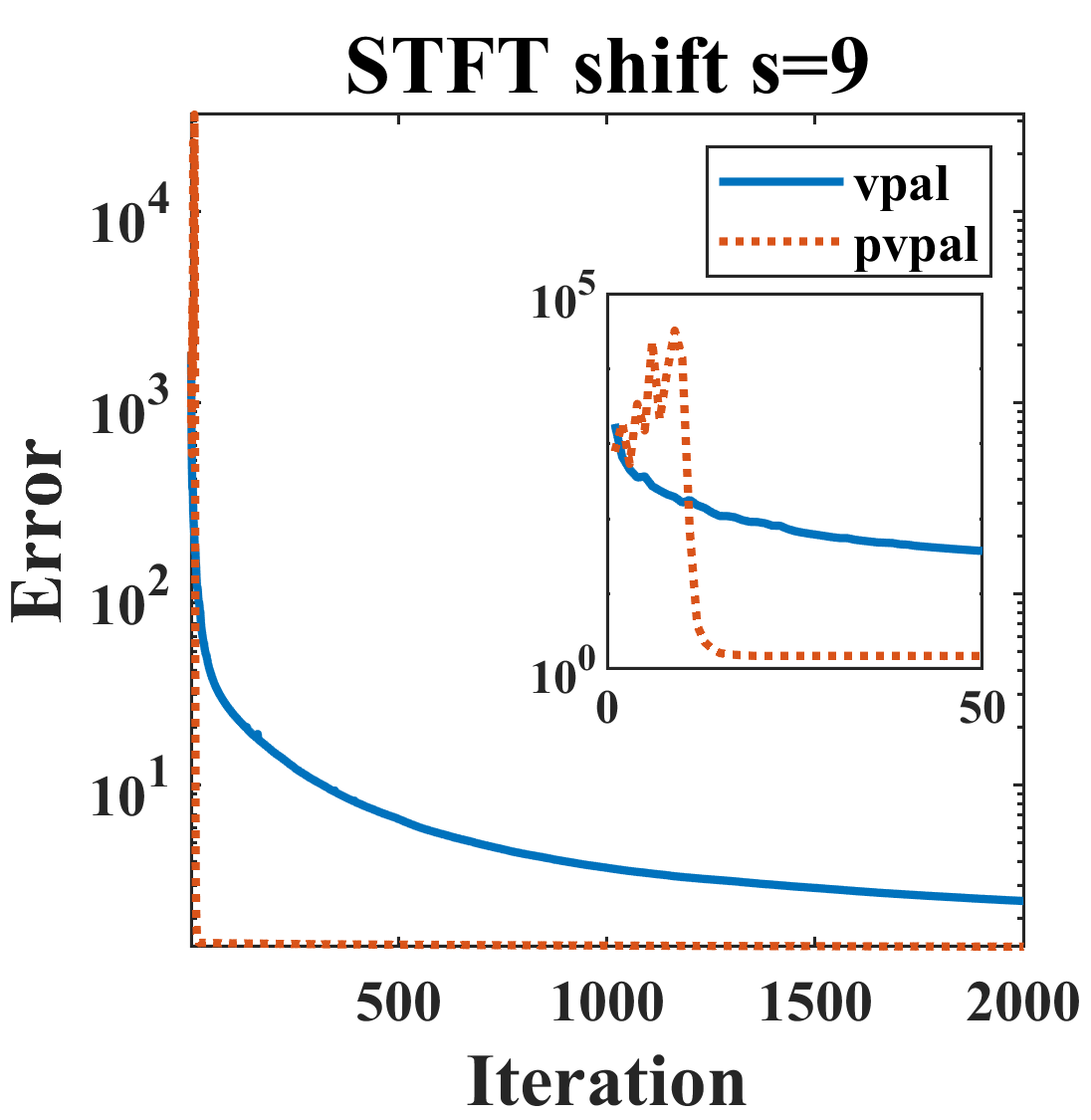}
		\caption{}
	\end{subfigure}
	\caption{Experiment 4: Minimization error during \texttt{vpal} and \texttt{pvpal} iterations. We use the STFT with exponential window (a,b) and Gaussian window (c,d) and a STFT shift $s=1$ (a,c) and $s=9$ (b,d). The small graph is zoomed in on the first $50$ iterations.}
	\label{fig:PR_iteration}
\end{figure}

We present one reconstruction example in \Cref{fig:PR_example} for a STFT shift $s=2$ using the exponential window (\Cref{fig:PR_example} (a,b)) and Gaussian window (\Cref{fig:PR_example} (c,d)). In \Cref{fig:PR_example} (e,f,g,h) we provide enlarged versions of the graphs for a better comparison. Especially for the exponential window The linearized reconstruction shows oscillating behavior on the amplitude. This effect was already observed in the original work \cite{sissouno2020direct}. Both \texttt{vpal} and \texttt{pvpal} return a much more accurate amplitude. \texttt{vpal} shows Gibbs phenomena when reconstructing the amplitude at jump discontinuities with exponential window while \texttt{pvpal} has similar effects under the Gaussian window. In both cases \texttt{pvpal} returns the most accurate phase.

\begin{figure}[htbp]
	\centering
	\begin{subfigure}[b]{0.24\textwidth}
		\includegraphics[width=\textwidth]{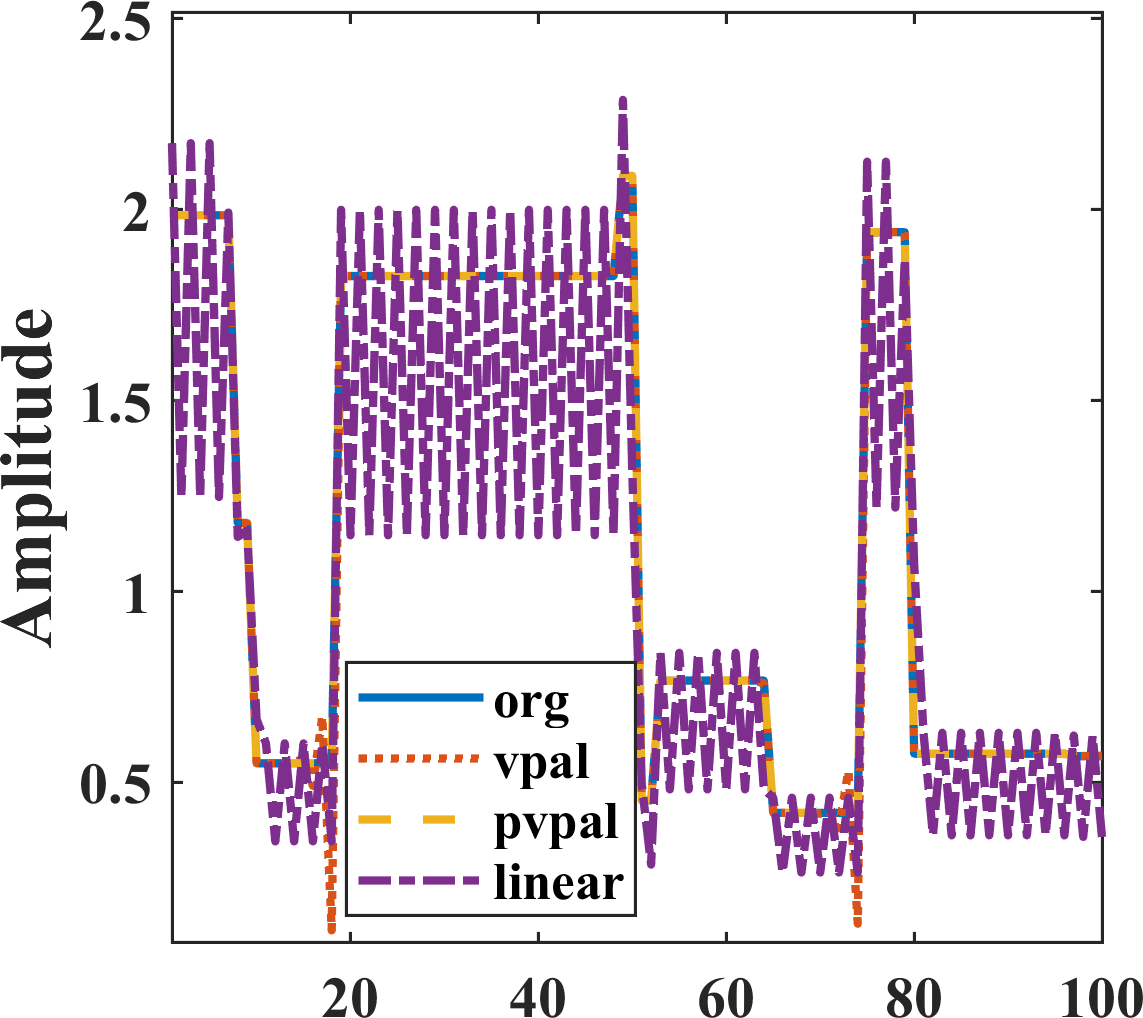}
		\caption{}
	\end{subfigure}
	\begin{subfigure}[b]{0.24\textwidth}
		\includegraphics[width=\textwidth]{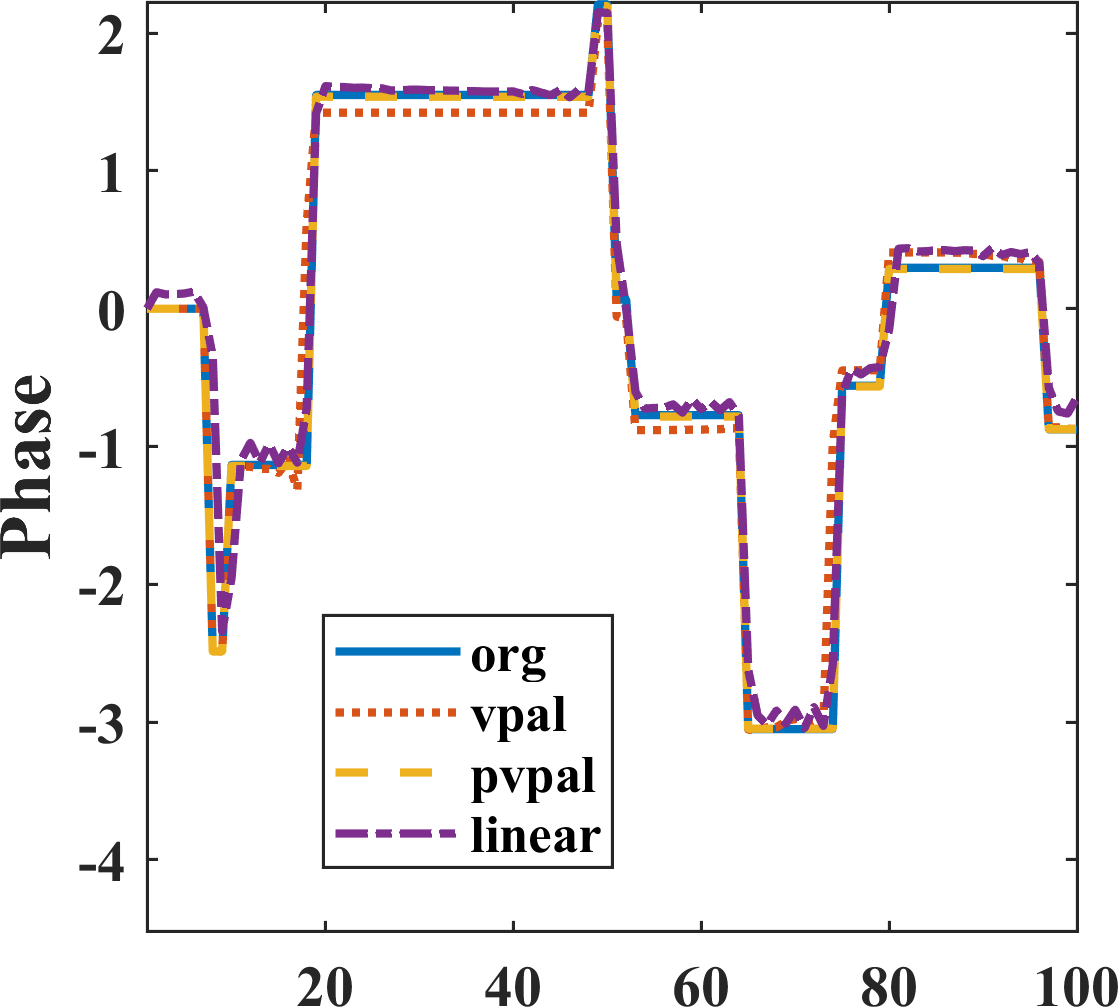}
		\caption{}
	\end{subfigure}
	\begin{subfigure}[b]{0.24\textwidth}
		\includegraphics[width=\textwidth]{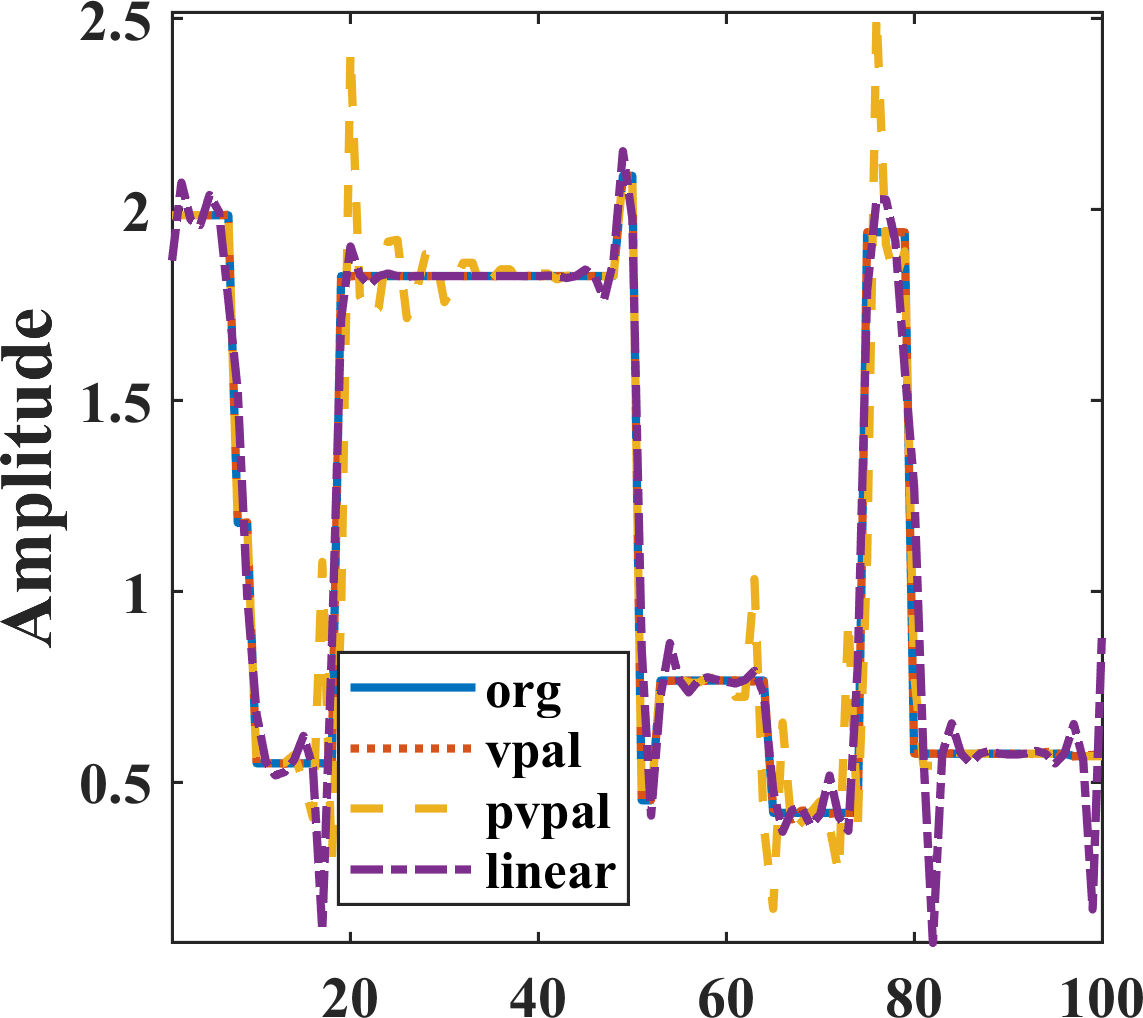}
		\caption{}
	\end{subfigure}
	\begin{subfigure}[b]{0.24\textwidth}
		\includegraphics[width=\textwidth]{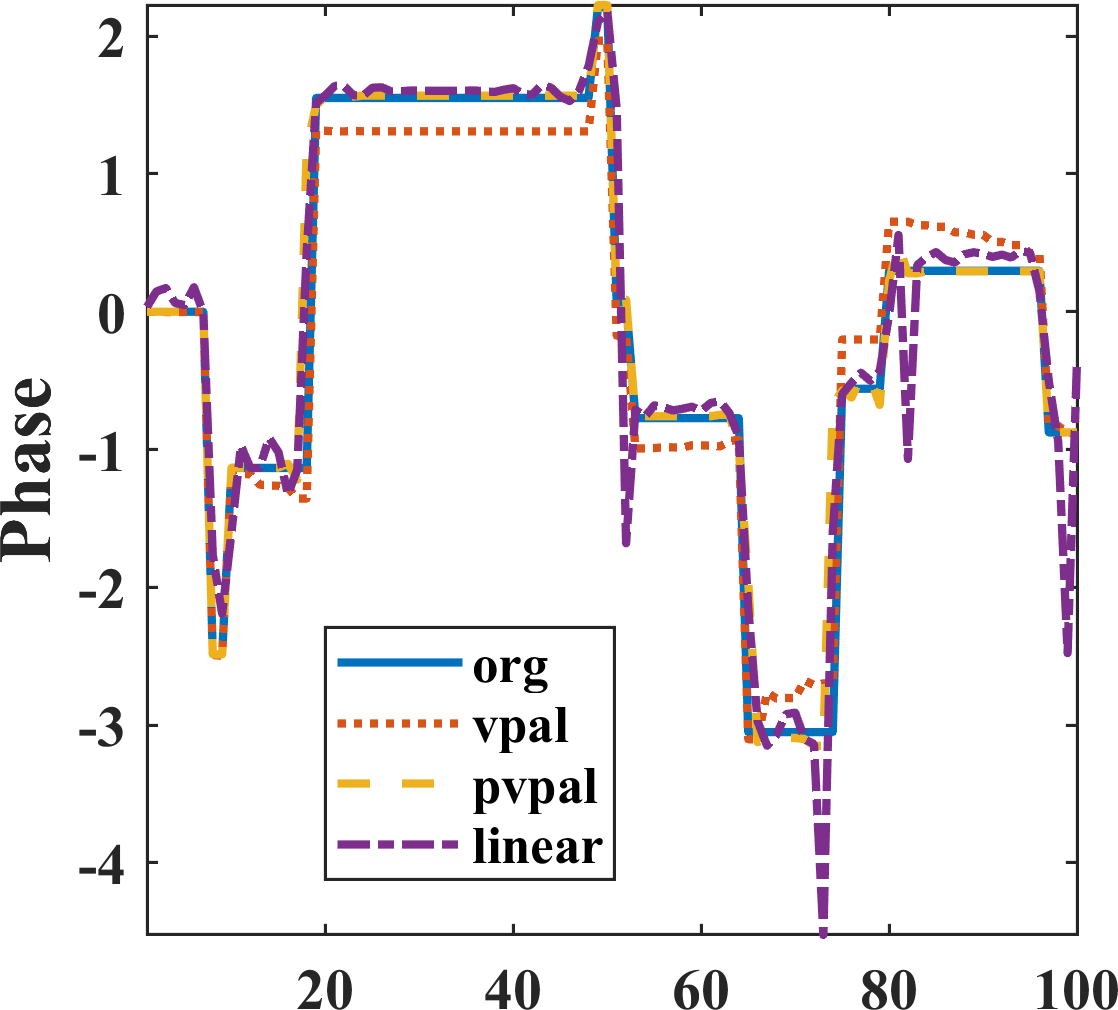}
		\caption{}
	\end{subfigure}
	\begin{subfigure}[b]{0.24\textwidth}
		\includegraphics[width=\textwidth]{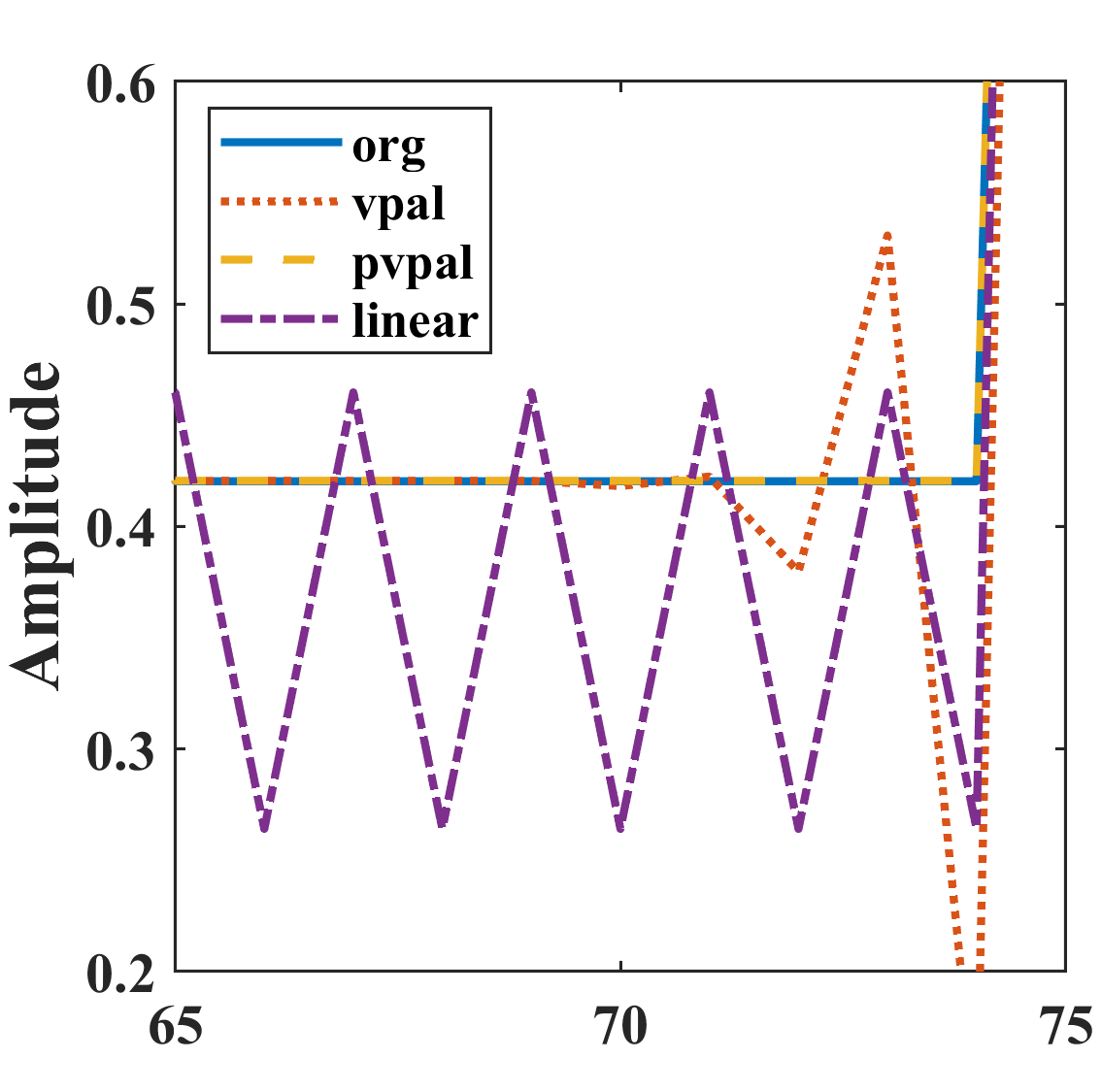}
		\caption{}
	\end{subfigure}
	\begin{subfigure}[b]{0.24\textwidth}
		\includegraphics[width=\textwidth]{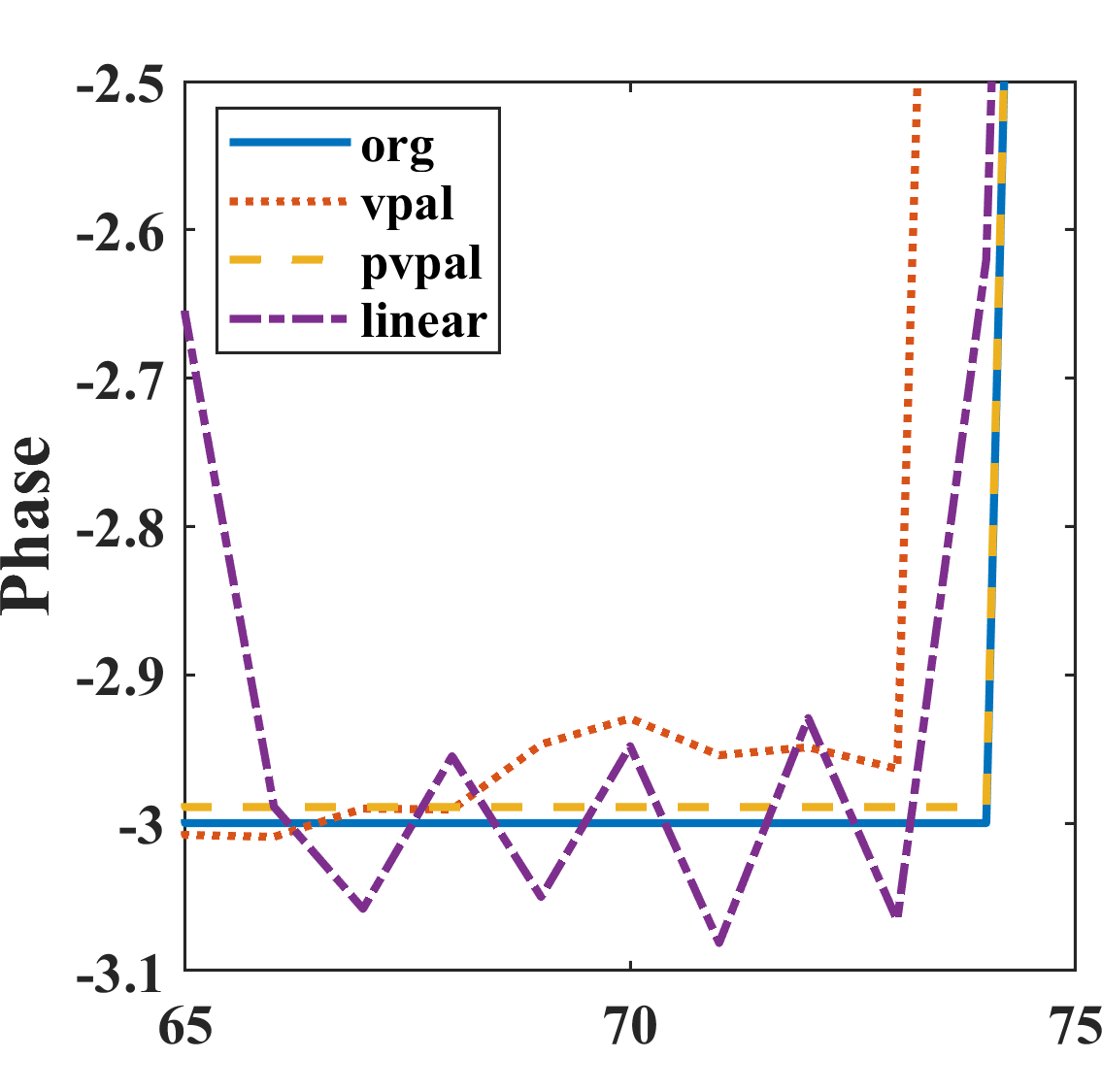}
		\caption{}
	\end{subfigure}
	\begin{subfigure}[b]{0.24\textwidth}
		\includegraphics[width=\textwidth]{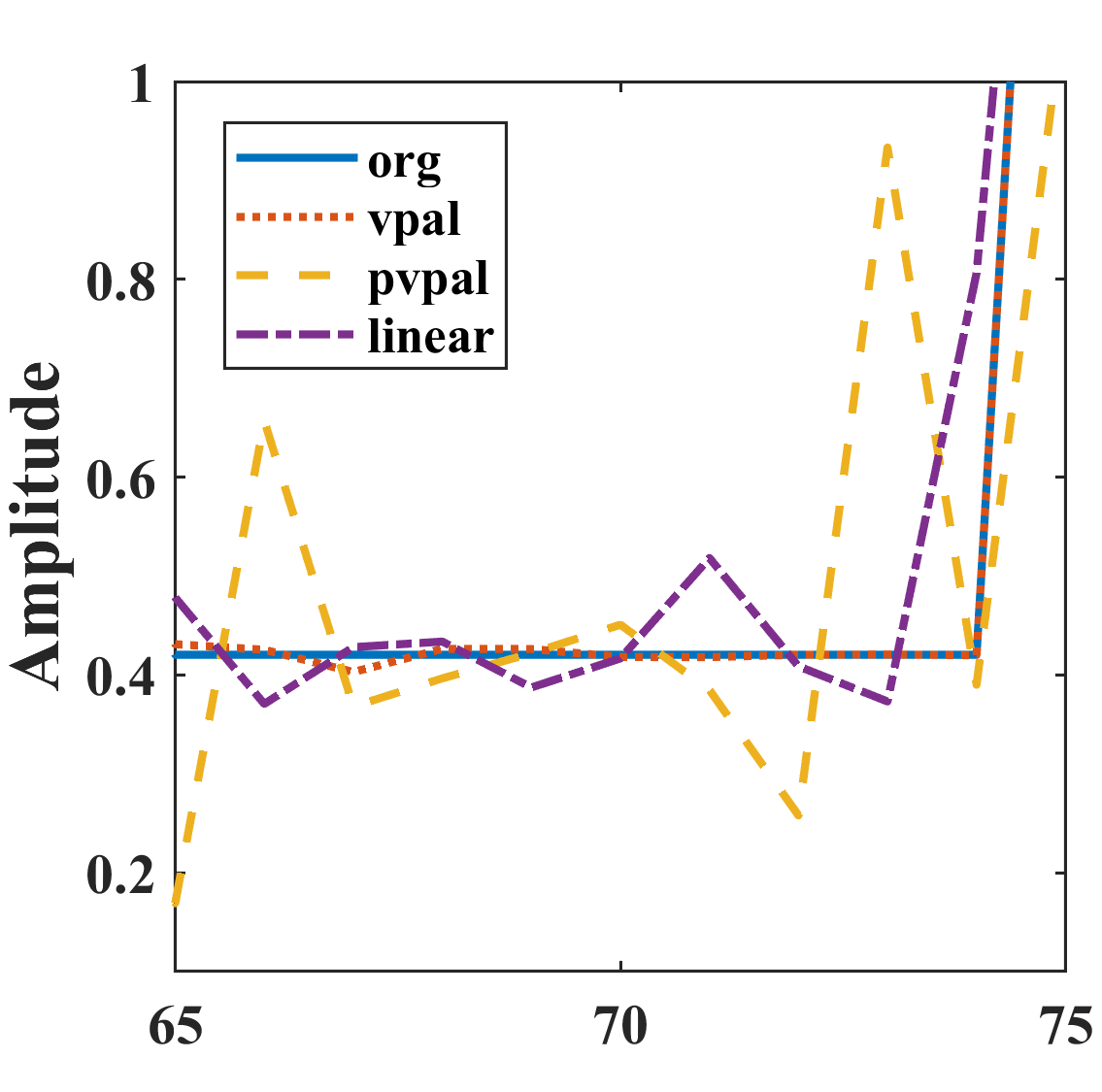}
		\caption{}
	\end{subfigure}
	\begin{subfigure}[b]{0.24\textwidth}
		\includegraphics[width=\textwidth]{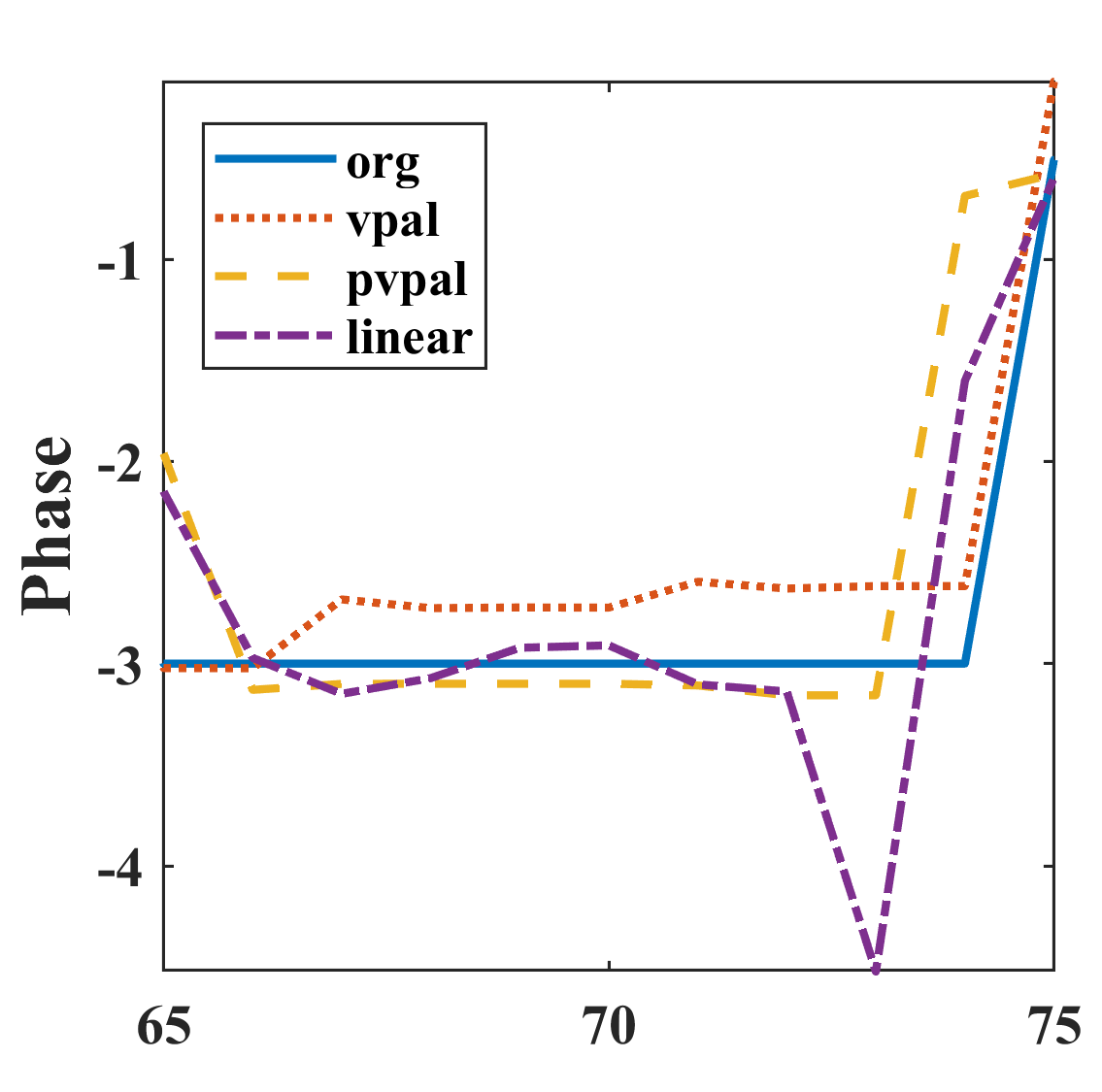}
		\caption{}
	\end{subfigure}
	\caption{Experiment 4: Exemplary reconstruction of amplitude (a,c) and phase (b,d) for STFT shift $s=2$ with exponential window (a,b) and Gaussian window (c,d). For better comparison (e,f,g,h) provide enlarged versions of the graph.}
	\label{fig:PR_example}
\end{figure}

Finally, we present in \Cref{fig:PR_stepsize} the distribution of chosen step sizes $\alpha$ from \Cref{alg:phase-step}. There was no significant difference between different STFT shifts $s$, hence we combined the results. When both axes are scaled logarithmically, the distribution follows a Gaussian-like curve centered around the $[0.001,0.01]$ bin. This motivates the choice of the default step size $\widetilde{\alpha}=0.001$ in \Cref{alg:phase-step}. Note that \texttt{pvpal} chooses much more step sizes within $[1,\infty)$ compared to \texttt{vpal}, which again indicates its much faster convergence rate.

\begin{figure}[htbp]
	\centering
	\includegraphics[width=0.8\textwidth]{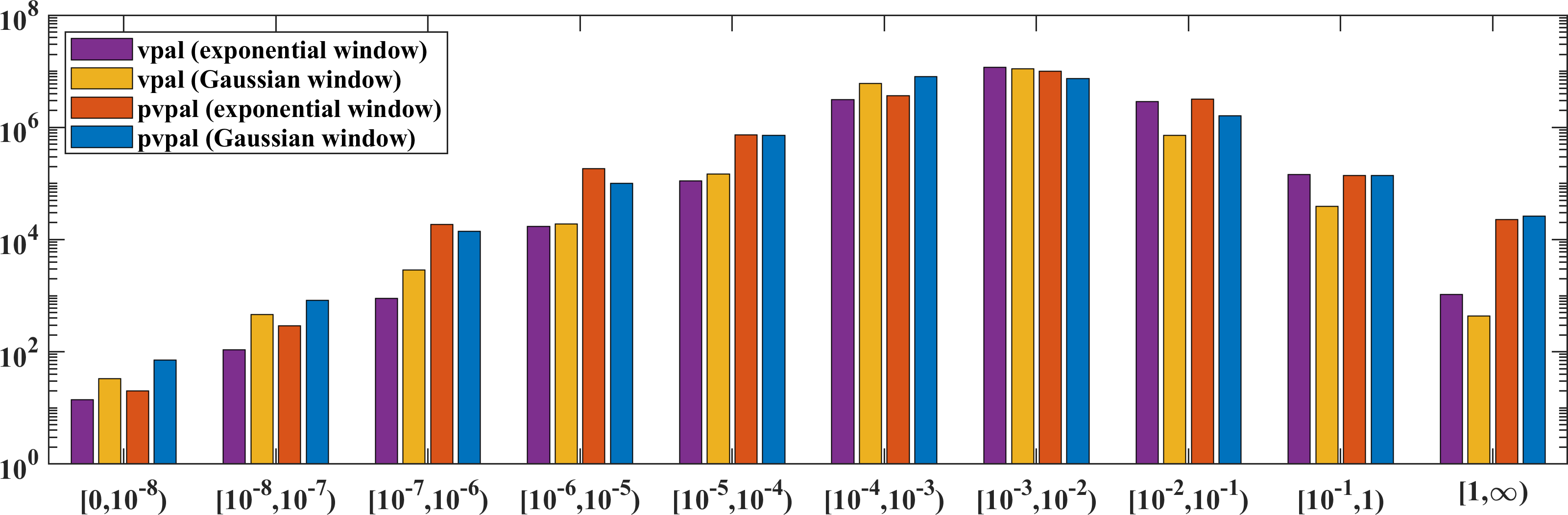}
	\caption{Experiment 4: Distribution of Step size $\alpha$ used by \texttt{vpal} and \texttt{pvpal} in different setups, both axes are scaled logarithmically.}
	\label{fig:PR_stepsize}
\end{figure}

\begingroup
\raggedbottom
\setlength{\textfloatsep}{12pt plus 2pt minus 2pt}
\setlength{\floatsep}{10pt plus 2pt minus 2pt}
\setlength{\intextsep}{10pt plus 2pt minus 2pt}
\captionsetup{font=small,skip=5pt}
\captionsetup[subfigure]{font=footnotesize,skip=3pt}

\subsection{Experiment 5: Learned Inverse Problem for Contrast Agent Reduction}\label{sec:LIPCAR}
In this experiment, we test \texttt{vpal} against a highly nonlinear forward operator in a real-world medical application. The test case is the Learned Inverse Problem for Contrast Agent Reduction (LIP-CAR), introduced in \cite{bianchi2024lip,evangelista2025lip}. In several diagnostic medical imaging procedures, doses of contrast agents (CAs) are administered to detect small or hard-to-see lesions. The Contrast Agent Reduction (CAR) problem refers to the task of lowering the CA dose while maintaining diagnostic accuracy (\Cref{fig:L2H}).

\begin{figure}[!htbp]
	\centering
	\includegraphics[width=0.72\linewidth]{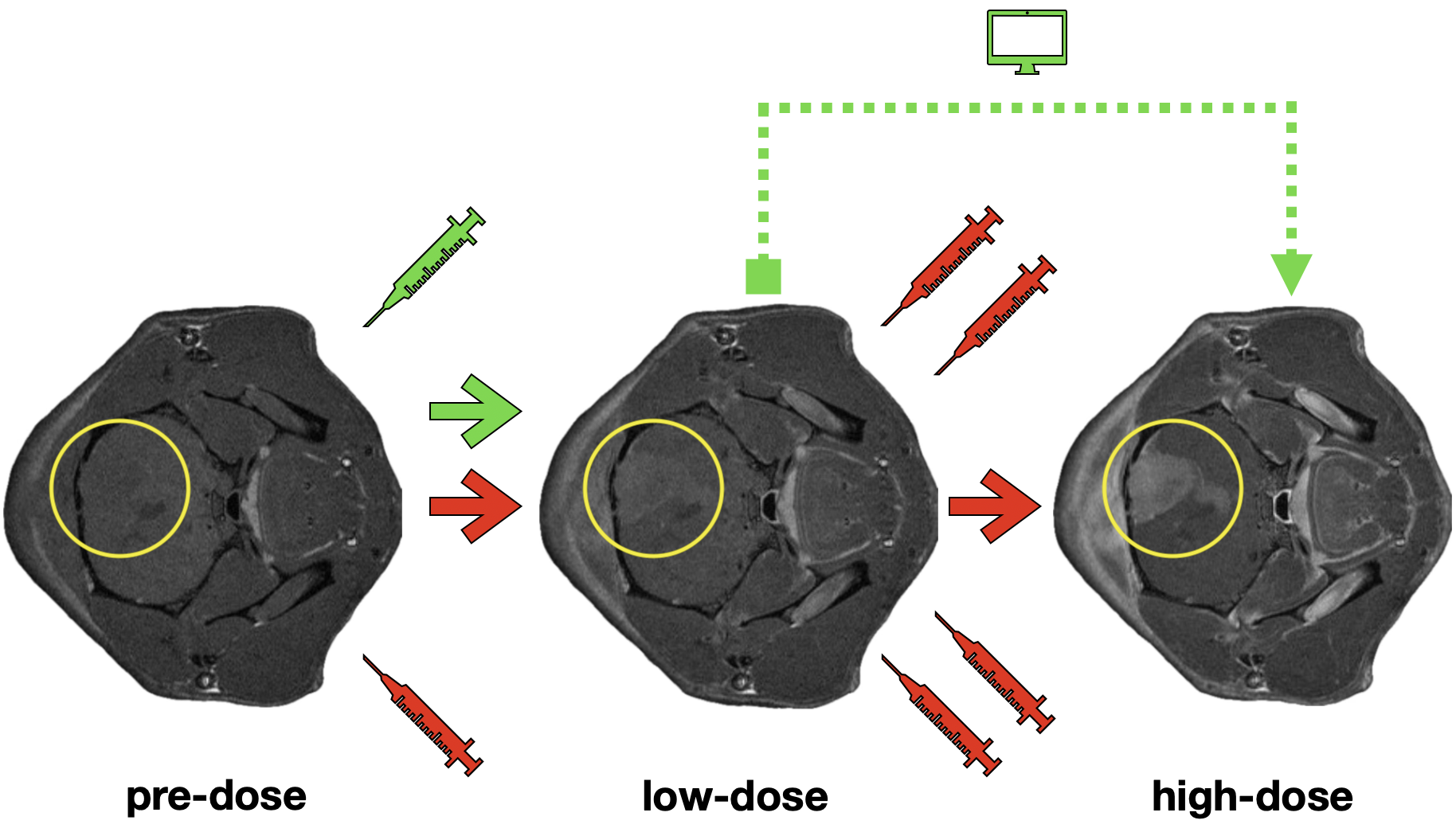}
	\caption{Experiment 5: Illustration of the CAR problem: MRI images acquired at 0\%, 20\%, and 100\% of the standard gadolinium-based CA dose (left to right). The high-dose image corresponds to the standard clinical protocol and clearly highlights the lesion (yellow circle). The conventional protocol (red path) proceeds from the pre-dose image to the high-dose image through the administration of the full CA dose. The CAR goal (green path) is to inject only the low dose (20\%) and to computationally simulate the high-dose (100\%) image.}\label{fig:L2H}
\end{figure}

Recently, CAR was recast as a learned inverse problem~\cite{bianchi2024lip,evangelista2025lip}; related inverse-problem formulations have also been disclosed in two international patent applications~\cite{bianchi2025highercontrast,bianchi2025longerdelays}. Given a noisy low-dose image~$x^\delta_{\textnormal{low}}$ and the corresponding noisy pre-dose image~$x^\delta_{\textnormal{pre}}$, i.e., the image acquired before any CA administration, we fix two trained neural networks, $\Phi_{\textnormal{L2H}}$ and $\Phi_{\textnormal{H2L}}$, which map low-dose images to high-dose images and high-dose images to low-dose images, respectively; both networks receive the pre-dose image as an additional input, as detailed below. The cited studies showed that solutions of the regularized problem
\begin{equation}\label{eq:LIP-CAR}
	\widehat{x}_{\textnormal{high}} \in \argmin_{x}
	\tfrac{1}{2}\,\bigl\lVert \Phi_{\textnormal{H2L}}(x^\delta_{\textnormal{pre}},\,x) \;-\; x^\delta_{\textnormal{low}}\bigr\rVert_{2}^{2}
	\;+\; \alpha\,\mathcal{R}\bigl(x,\,p\bigr), \ p = \Phi_{\textnormal{L2H}}(x^\delta_{\textnormal{pre}}, x^\delta_{\textnormal{low}}),
\end{equation}
are robust to noise perturbations and preserve contrast-enhanced regions better than predictions obtained using only the network $\Phi_{\textnormal{L2H}}$. Here, $\alpha>0$ is a regularization parameter, and $\mathcal{R}(x,p)$ is a regularizer that exploits the high-dose prediction $p$ of $\Phi_{\textnormal{L2H}}$ as a prior image; its explicit form is given in \Cref{eq:LIP-CAR_reg} below. The solution $\widehat{x}_{\textnormal{high}}$ is a synthetic high-dose image; throughout this section, the symbol $x_{\textnormal{high}}$ is reserved for the measured (ground-truth) high-dose image. The forward operator, here represented by the network $\Phi_{\textnormal{H2L}}$, is highly nonlinear.

Both networks share the same U-Net-like architecture~\cite{ronneberger2015u} with residual connections. The architecture is a multiscale convolutional neural network with an encoder--decoder structure and elementwise sums, rather than concatenations, for the skip connections. Its nonlinear operations comprise four two-dimensional max-pooling operations and 23 ReLU activation functions. The total number of trainable parameters is approximately 28 million.

To increase stability and accuracy, both networks take the pre-dose image as an additional input. Specifically, $\Phi_{\textnormal{L2H}}$ maps a pre-dose image and a low-dose image to a high-dose prediction, whereas $\Phi_{\textnormal{H2L}}$ maps a pre-dose image and a high-dose image to a low-dose prediction. In both cases, the two $256\times256$ grayscale input images are stacked along the channel dimension, and the output is a single $256\times256$ grayscale image.

The LIP-CAR experiments were implemented in Python using PyTorch~\cite{paszke2019pytorch}. We relied on the \texttt{torch.optim} implementations of Adam (adaptive moment estimation), SGD (stochastic gradient descent), and RMSprop (root mean square propagation)~\cite{pytorchOptimizers,ruder2016overview}, each with a learning rate of $10^{-3}$ and otherwise default optimizer settings. The \texttt{vpal} solver was implemented using PyTorch tensors on the same device. Reconstruction computations used single-precision (32-bit) tensors and the Metal Performance Shaders (MPS) backend on a MacBook Pro running macOS, equipped with an Apple M4 Max chip with a 16-core CPU, a 40-core GPU, and 48\,GB of unified memory.

The dataset comprises 61 cranial MRI examinations from a preclinical study involving laboratory rats with induced C6 gliomas. The procedures were conducted by the CRB/Test Facility of Bracco Imaging S.p.A. in accordance with national and international legislation on experimental animal research (Legislative Decree 26/2014; Directive 2010/63/EU) and under Italian ministerial authorization no.~215/2020-PR. Each examination includes three T1-weighted spin-echo sequences (TR/TE = 360/5.51 ms, FOV = $32\times32$ mm, matrix = $256\times256$, 24 slices of 0.75 mm thickness): pre-dose, low-dose (0.01 mmol Gd/kg), and high-dose (0.05 mmol Gd/kg) images.

The dataset was first preprocessed using a nonlocal-means filter~\cite{wiest2008rician,coupe2008optimized} to mitigate Rician noise and was then split into a training set comprising 42 examinations (1,008 images in total), a validation set comprising 15 examinations (360 images in total), and a disjoint test set comprising 4 examinations (96 images in total).

The networks were trained for 100 epochs with a batch size of 16 using the Adam optimizer and a fixed learning rate of $10^{-3}$. Training minimized the mean SSIM-based loss~\cite{wang2004image} over the training samples. On the test set, $\Phi_{\textnormal{L2H}}$ achieved a mean SSIM of 0.954 (standard deviation: 0.057), whereas $\Phi_{\textnormal{H2L}}$ achieved a mean SSIM of 0.988 (standard deviation: 0.006).

Once the networks have been trained, we freeze their weights. We then minimize the functional in \Cref{eq:LIP-CAR} over candidate high-dose images $x$, using the high-dose prediction $p=\Phi_{\textnormal{L2H}}(x^\delta_{\textnormal{pre}}, x^\delta_{\textnormal{low}})$ as the prior image in the regularization term $\mathcal{R}$. To reduce reliance on this potentially unstable prior, we divide the optimization into two phases. In Phase~1, the iterative solver is initialized with the noisy low-dose image $x^\delta_{\textnormal{low}}$ and uses the prior image in the regularizer until the stopping criteria are satisfied. In Phase~2, the solver is reinitialized with the output of Phase~1 and runs without the prior image in the regularization term until the stopping criteria are satisfied again. The complete procedure is summarized in \Cref{alg:LIP-CAR}.

\begin{algorithm}[!htbp]
	\caption{LIP-CAR}
	\label{alg:LIP-CAR}
	\small
	\begin{algorithmic}[1]
		\Require
		Noisy pre-dose and low-dose images $x^\delta_{\textnormal{pre}}$, $x^\delta_{\textnormal{low}}$; trained networks $\Phi_{\textnormal{L2H}}$, $\Phi_{\textnormal{H2L}}$; regularizer~$\mathcal{R}(\cdot)$;
		regularization parameters $\alpha_1, \alpha_2$; maximum number of iterations $k_{\textnormal{max}}$; tolerance $\text{tol}$;
		iterative solver~$\texttt{solver}(\cdot)$.
		
		\Statex
		\Procedure{ShouldStop}{$k, x_k, x_{k-1}, f_k, f_{k-1}$}
		\State $C_1 \gets |f_k - f_{k-1}| \le \text{tol}\,(1+f_k)$ \Comment{Objective check}
		\State $C_2 \gets \norm{x_k - x_{k-1}}_\infty \le \sqrt{\text{tol}}\,(1+\norm{x_k}_\infty)$ \Comment{Iterate check}
		\State $C_3 \gets k \ge k_{\textnormal{max}}$ \Comment{Maximum-iteration check}
		\State \textbf{return} $(C_1 \text{ \textbf{and} } C_2) \text{ \textbf{or} } C_3$
		\EndProcedure
		
		\Statex
		\Statex \texttt{// Phase 1: optimization with the prior image}
		\State Compute the prior image: $p = \Phi_{\textnormal{L2H}}(x^\delta_{\textnormal{pre}}, x^\delta_{\textnormal{low}})$.
		\State Initialize: $x_0 = x^\delta_{\textnormal{low}}$, $k = 0$.
		\State Let $f^{(1)}(x) = \left( \frac{1}{2}\norm{\Phi_{\textnormal{H2L}}(x^\delta_{\textnormal{pre}}, x) - x^\delta_{\textnormal{low}}}_2^2 + \alpha_1 \mathcal{R}(x, p) \right) / \norm{x^\delta_{\textnormal{low}}}_2^2$.
		\Repeat
		\State $x_{k+1} = \texttt{solver}(x_k, \alpha_1, p)$ \Comment{Apply one solver step to \Cref{eq:LIP-CAR}}
		\State $k = k+1$
		\Until{\Call{ShouldStop}{$k, x_k, x_{k-1}, f^{(1)}(x_k), f^{(1)}(x_{k-1})$}}
		\State Set the Phase 1 result: $x_{\textnormal{phase1}} = x_{k}$.
		
		\Statex 
		\Statex \texttt{// Phase 2: refinement without the prior image}
		\State Initialize: $x_0 = x_{\textnormal{phase1}}$, $k = 0$.
		\State Let $f^{(2)}(x) = \left( \frac{1}{2}\norm{\Phi_{\textnormal{H2L}}(x^\delta_{\textnormal{pre}}, x) - x^\delta_{\textnormal{low}}}_2^2 + \alpha_2 \mathcal{R}(x) \right) / \norm{x^\delta_{\textnormal{low}}}_2^2$.
		\Repeat
		\State $x_{k+1} = \texttt{solver}(x_k, \alpha_2)$ \Comment{Apply one solver step to \Cref{eq:LIP-CAR}}
		\State $k = k+1$
		\Until{\Call{ShouldStop}{$k, x_k, x_{k-1}, f^{(2)}(x_k), f^{(2)}(x_{k-1})$}}
		\State Set the final output: $\widehat{x}_{\textnormal{high}} = x_{k}$.
		
		\State \Return $\widehat{x}_{\textnormal{high}}$.
		
		\Ensure
		Reconstructed high-dose image $\widehat{x}_{\textnormal{high}}$.
	\end{algorithmic}
\end{algorithm}
We fix the preprocessed dataset $\{(x_{\textnormal{pre}}, x_{\textnormal{low}}, x_{\textnormal{high}})\}$ as a surrogate ground truth, which serves as our noise-free baseline. Both pre-dose and low-dose images from the test set are then perturbed with independent realizations $\eta_{\textnormal{pre}}$ and $\eta_{\textnormal{low}}$ of Rician-distributed noise, each scaled to the relative intensity $\delta>0$:
\begin{equation*}
	x_{\textnormal{pre}}^\delta = x_{\textnormal{pre}} + \delta\,\eta_{\textnormal{pre}}\frac{\norm{x_{\textnormal{pre}}}_2}{\norm{\eta_{\textnormal{pre}}}_2}, \qquad x_{\textnormal{low}}^\delta = x_{\textnormal{low}} + \delta\,\eta_{\textnormal{low}}\frac{\norm{x_{\textnormal{low}}}_2}{\norm{\eta_{\textnormal{low}}}_2},
\end{equation*}
so that $\norm{x_{\textnormal{pre}}^\delta - x_{\textnormal{pre}}}_2 = \delta\norm{x_{\textnormal{pre}}}_2$ and $\norm{x_{\textnormal{low}}^\delta - x_{\textnormal{low}}}_2 = \delta\norm{x_{\textnormal{low}}}_2$.
As the regularizer, we use an $\ell^1$-TV term. In the notation of \Cref{eq:genLasso},
\begin{equation}\label{eq:LIP-CAR_reg}
	\mathcal{R}(x,p)= \norm{D(x-p)}_1 \quad \text{with} \quad D = \begin{bmatrix} \nabla_h \\ \nabla_v \end{bmatrix},
\end{equation}
where $p$ is the prior image. In Phase~2 of \Cref{alg:LIP-CAR}, no prior image is used, and the regularizer reduces to its one-argument form $\mathcal{R}(x) \coloneqq \mathcal{R}(x,0) = \norm{Dx}_1$.

We evaluate \Cref{alg:LIP-CAR} using four different solvers: \texttt{vpal}, Adam, SGD, and RMSprop. The hyperparameters for these solvers were manually tuned based on the best overall observed performance and were then held fixed, with a maximum of $k_{\textnormal{max}} = 500$ iterations and a stopping tolerance of $\text{tol}=10^{-5}$ for all methods. For \texttt{vpal}, we use a linearized step size and fix the augmented-Lagrangian penalty parameter at $\lambda=1$. We do not use the preconditioned variant \texttt{pvpal} in this experiment. Since the forward operator $\Phi_{\textnormal{H2L}}$ is a deep neural network, the preconditioner of \Cref{sec:precond} would require products with the Jacobian of $\Phi_{\textnormal{H2L}}$ and its transpose at every inner iteration of the linear solver, that is, additional forward and backward passes through the network at each step. This cost would outweigh the gain in convergence speed, and plain \texttt{vpal} is already the fastest solver among those considered.

In \Cref{tab:ADAMvsVPAL}, we compare solver performance for \Cref{alg:LIP-CAR} across various noise intensity levels $\delta$, using a fixed reference triplet of images $(x^\delta_{\textnormal{pre}}, x^\delta_{\textnormal{low}}, x_{\textnormal{high}})$. All solvers considered outperform the neural network $\Phi_{\textnormal{L2H}}$. In particular, \texttt{vpal} is consistently the fastest solver and achieves results comparable to those of the best-performing alternative. See \Cref{fig:LIP-CAR_comparison,fig:LIP-CAR_error_maps} for visual comparisons.

\begin{table}[!htbp]
	\centering \footnotesize
	\setlength{\tabcolsep}{3pt}
	\begin{tabular}{@{} l *{3}{rrr} @{}}
		\toprule
		& \multicolumn{3}{c}{$\delta=0.1$}
		& \multicolumn{3}{c}{$\delta=0.2$}
		& \multicolumn{3}{c}{$\delta=0.3$} \\
		& \multicolumn{3}{c}{$(\alpha_1,\alpha_2)=(0.03,0.01)$}
		& \multicolumn{3}{c}{$(\alpha_1,\alpha_2)=(0.05,0.02)$}
		& \multicolumn{3}{c}{$(\alpha_1,\alpha_2)=(0.07,0.03)$} \\
		\cmidrule(lr){2-4}\cmidrule(lr){5-7}\cmidrule(l){8-10}
		\textbf{Method} & \textbf{Time (s)} & \textbf{RRE} & \textbf{SSIM}
		& \textbf{Time (s)} & \textbf{RRE} & \textbf{SSIM}
		& \textbf{Time (s)} & \textbf{RRE} & \textbf{SSIM} \\
		\midrule
		$\Phi_{\textnormal{L2H}}$ & -- & 0.429 & 0.858 & -- & 0.446 & 0.766 & -- & 0.444 & 0.699 \\
		LIP-CAR: \\
		\ \ Adam & 12.47 & 0.208 & 0.911 & 15.60 & 0.197 & 0.884 & 18.01 & 0.207 & 0.858 \\
		\ \ SGD & 10.68 & 0.235 & 0.880 & 28.30 & 0.229 & 0.852 & 29.68 & 0.238 & 0.805 \\
		\ \ RMSprop & 9.63 & 0.205 & 0.912 & 12.18 & 0.198 & 0.884 & 12.05 & 0.212 & 0.858 \\
		\ \ \texttt{vpal} & \textbf{6.98} & 0.207 & 0.906 & \textbf{9.62} & 0.196 & 0.878 & \textbf{11.38} & 0.199 & 0.855 \\
		\bottomrule
	\end{tabular}
	\caption{Experiment 5: Metrics comparison. We compare $\Phi_{\textnormal{L2H}}$ and LIP-CAR with different solvers (Adam, SGD, RMSprop, and \texttt{vpal} across different noise levels $\delta$. The metrics compare each reconstructed high-dose image $\widehat{x}_{\textnormal{high}}$ with the ground-truth high-dose image $x_{\textnormal{high}}$. Note the speed-up gain achieved with \texttt{vpal}.} 
	\label{tab:ADAMvsVPAL} 
\end{table}

\begin{figure}[!htbp]
	\centering
	\begin{subfigure}[b]{0.3\linewidth}
		\includegraphics[width=\linewidth]{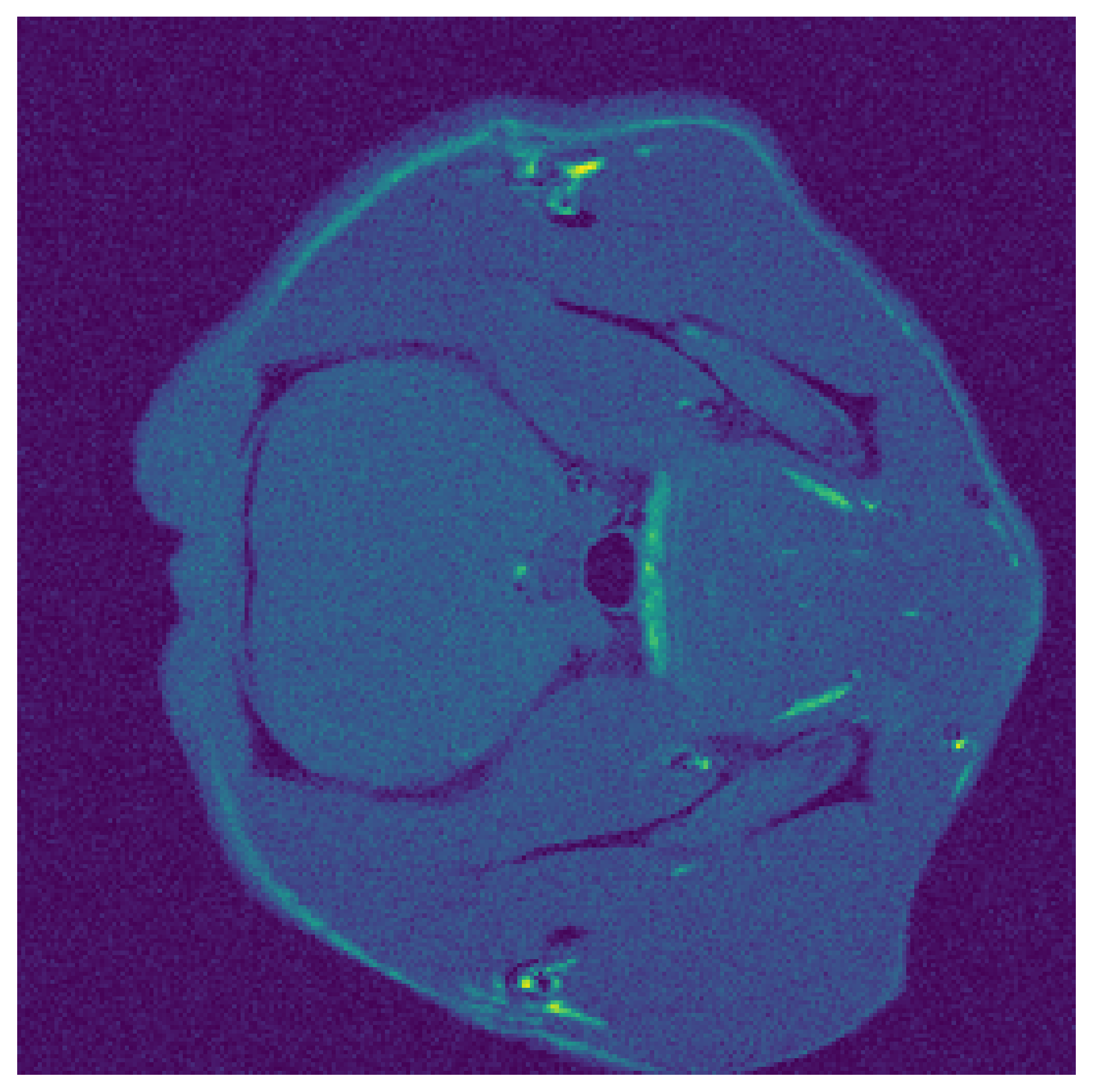}
		\caption{Noisy pre-dose image}
	\end{subfigure}
	\hspace{0.015\linewidth}
	\begin{subfigure}[b]{0.3\linewidth}
		\includegraphics[width=\linewidth]{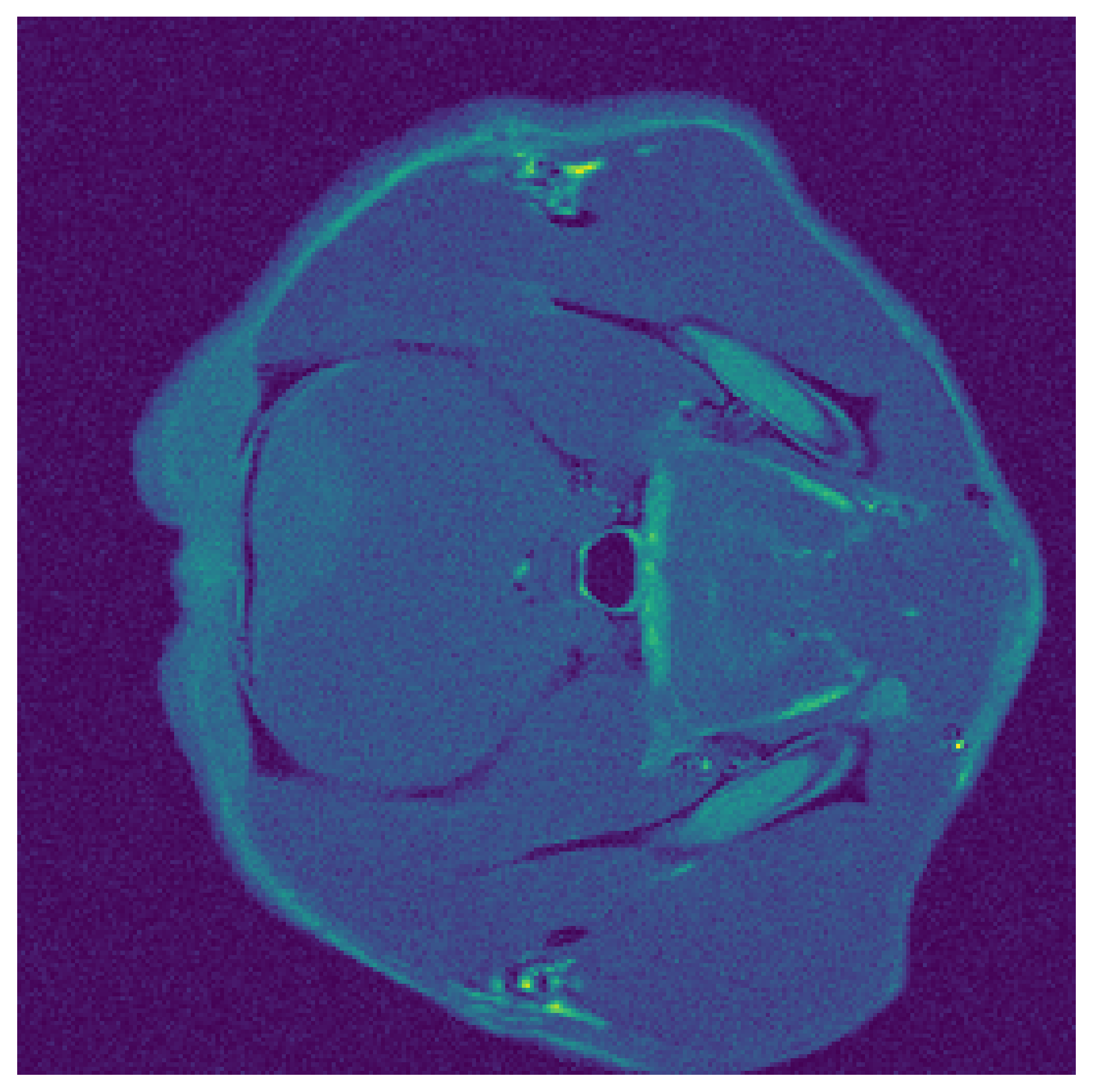}
		\caption{Noisy low-dose image}
	\end{subfigure}
	\hspace{0.015\linewidth}
	\begin{subfigure}[b]{0.3\linewidth}
		\includegraphics[width=\linewidth]{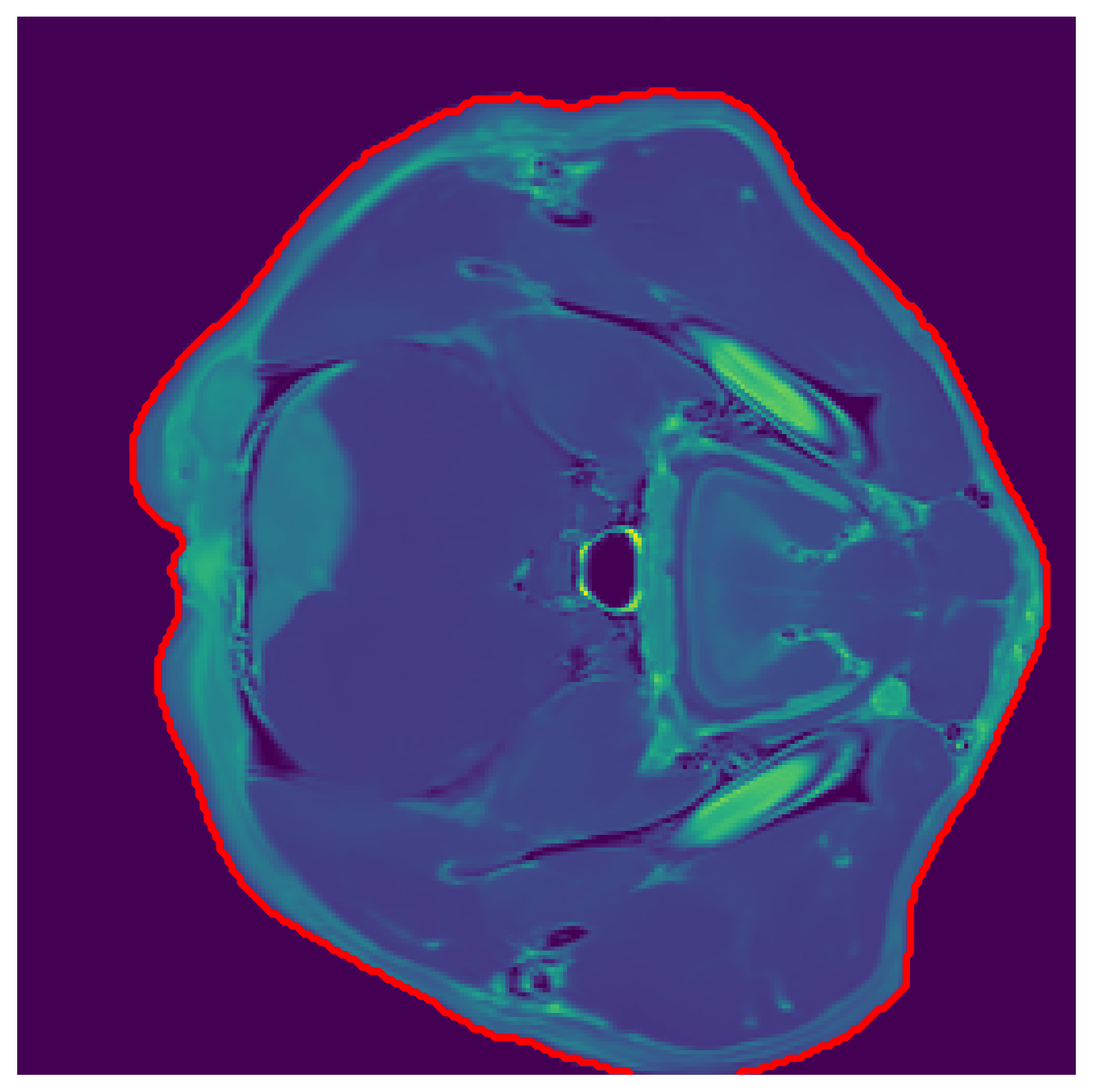}
		\caption{High-dose ground truth}
	\end{subfigure}\\
	
	\begin{subfigure}[b]{0.3\linewidth}
		\includegraphics[width=\linewidth]{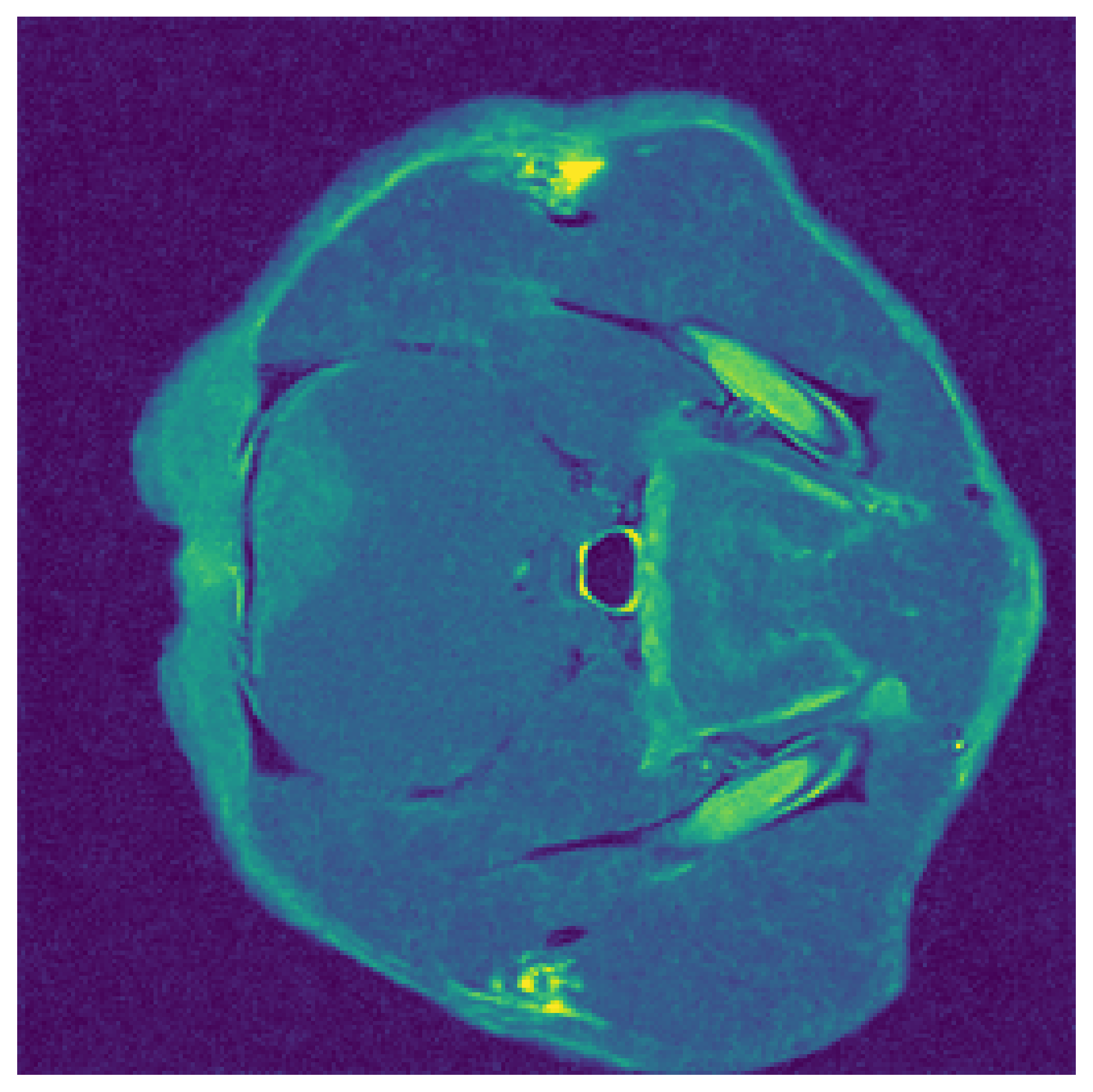}
		\caption{$\Phi_{\textnormal{L2H}}$}
	\end{subfigure}
	\hspace{0.015\linewidth}
	\begin{subfigure}[b]{0.3\linewidth}
		\includegraphics[width=\linewidth]{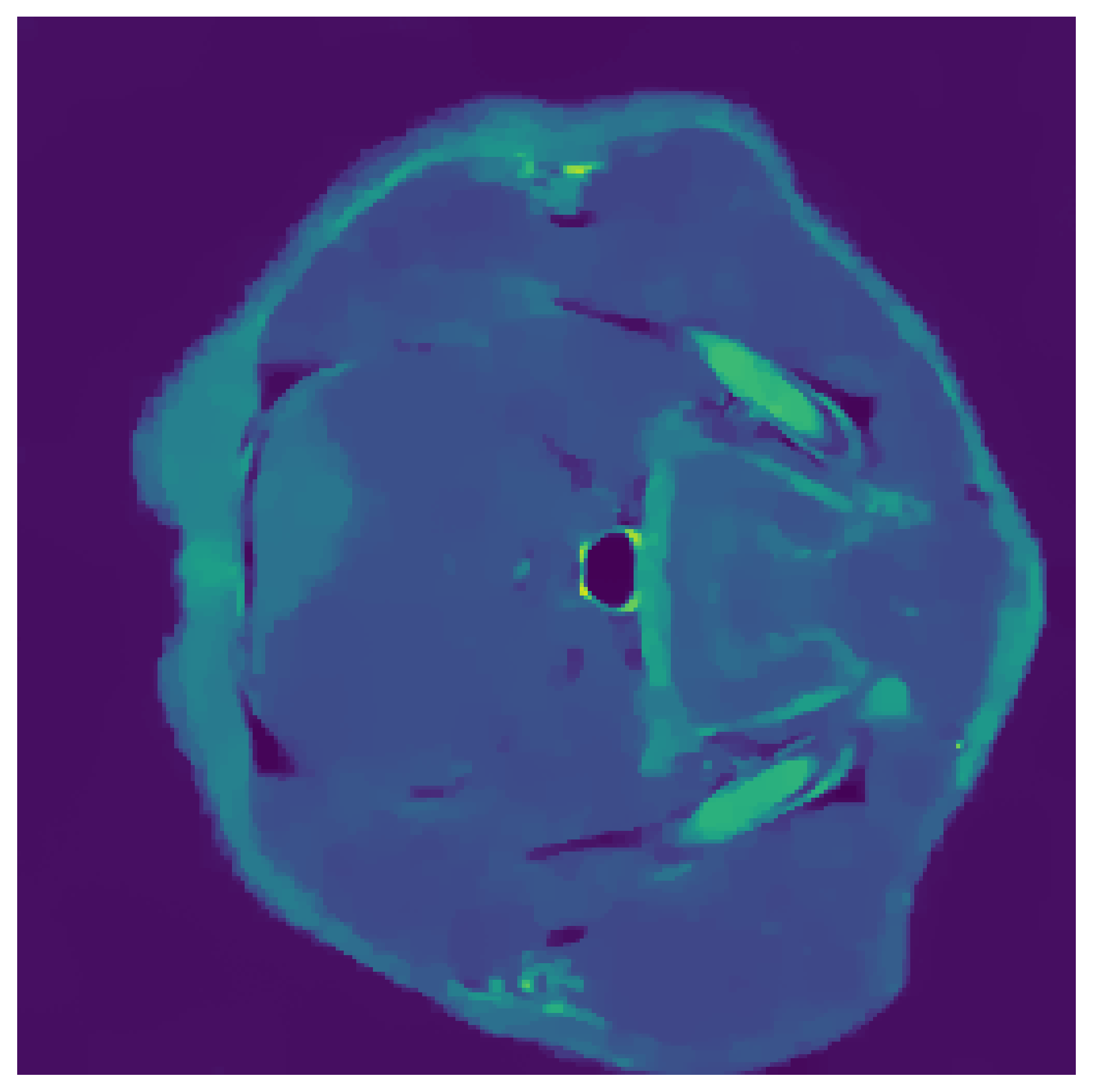}
		\caption{LIP-CAR (Adam)}
	\end{subfigure}
	\hspace{0.015\linewidth}
	\begin{subfigure}[b]{0.3\linewidth}
		\includegraphics[width=\linewidth]{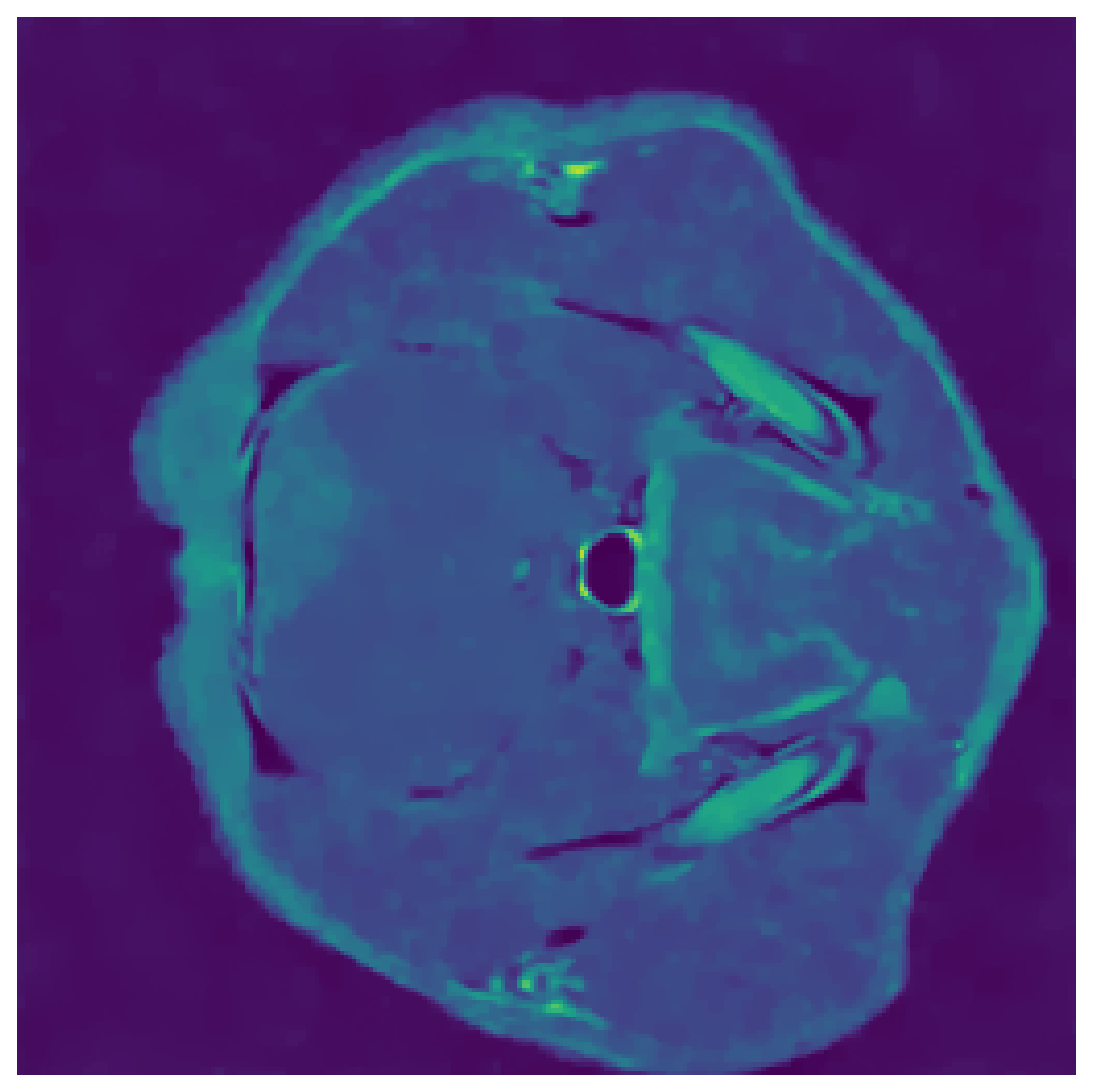}
		\caption{LIP-CAR (\texttt{vpal})}
	\end{subfigure}
	\caption{Experiment 5: Reconstructions for $\delta=0.2$ in \Cref{tab:ADAMvsVPAL}. The red contour in (c) marks the region of interest used for all comparison metrics. Panel (d) shows the network prediction $\Phi_{\textnormal{L2H}}(x^\delta_{\textnormal{pre}},x^\delta_{\textnormal{low}})$; panels (e,f) show reconstructed high-dose images $\widehat{x}_{\textnormal{high}}$ obtained by \Cref{alg:LIP-CAR} with Adam and \texttt{vpal}, respectively.}
	\label{fig:LIP-CAR_comparison}
\end{figure}

\begin{figure}
	\centering
	\begin{subfigure}[b]{0.3\linewidth}
		\includegraphics[width=\linewidth]{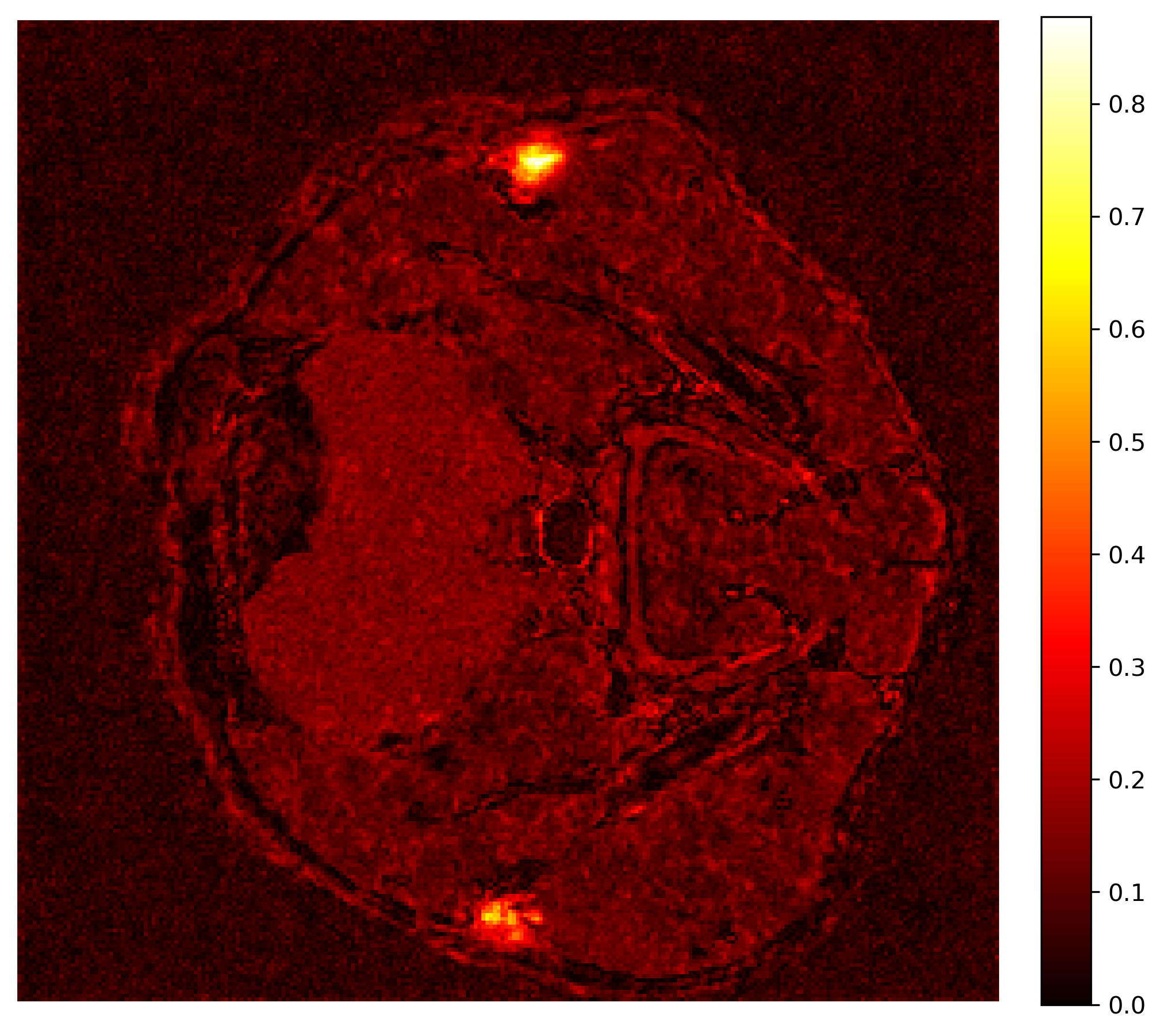}
		\caption{$\Phi_{\textnormal{L2H}}$ (max.\ 0.877)}
	\end{subfigure}
	\hspace{0.015\linewidth}
	\begin{subfigure}[b]{0.3\linewidth}
		\includegraphics[width=\linewidth]{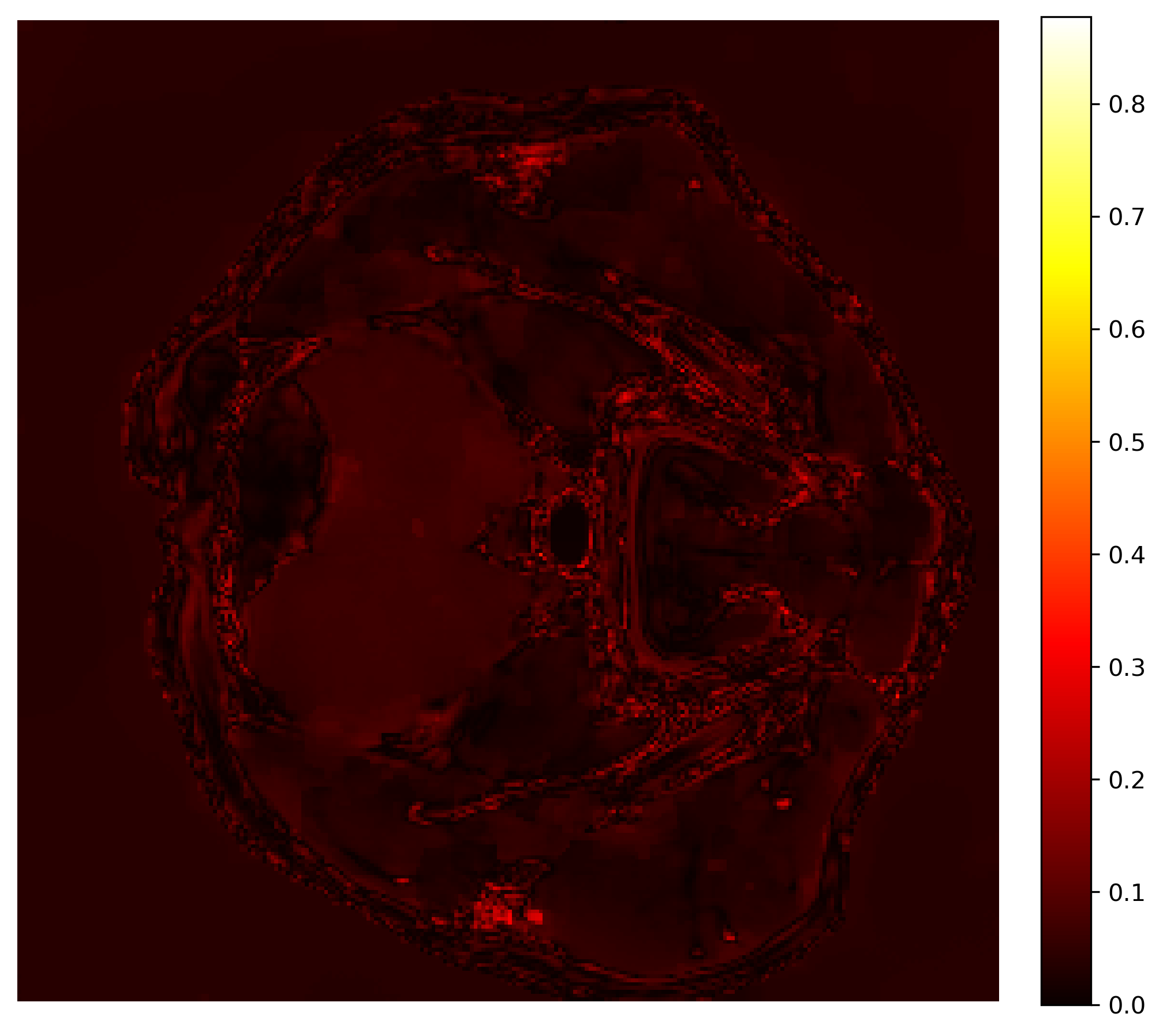}
		\caption{Adam (max.\ 0.358)}
	\end{subfigure}
	\hspace{0.015\linewidth}
	\begin{subfigure}[b]{0.3\linewidth}
		\includegraphics[width=\linewidth]{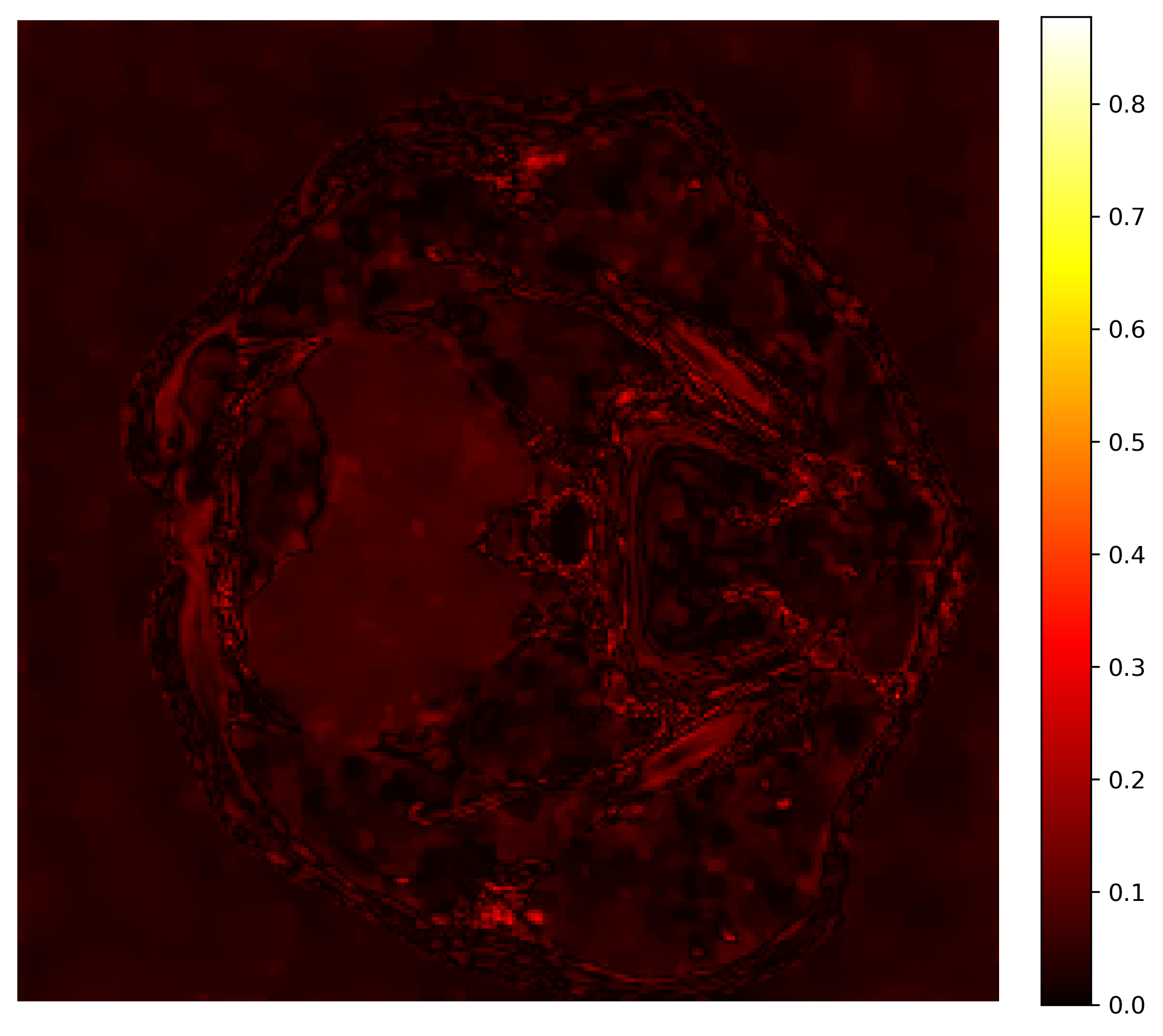}
		\caption{\texttt{vpal} (max.\ 0.350)}
	\end{subfigure}
	
	\par\medskip
	
	\begin{subfigure}[b]{\linewidth}
		\centering
		\includegraphics[width=0.6\linewidth]{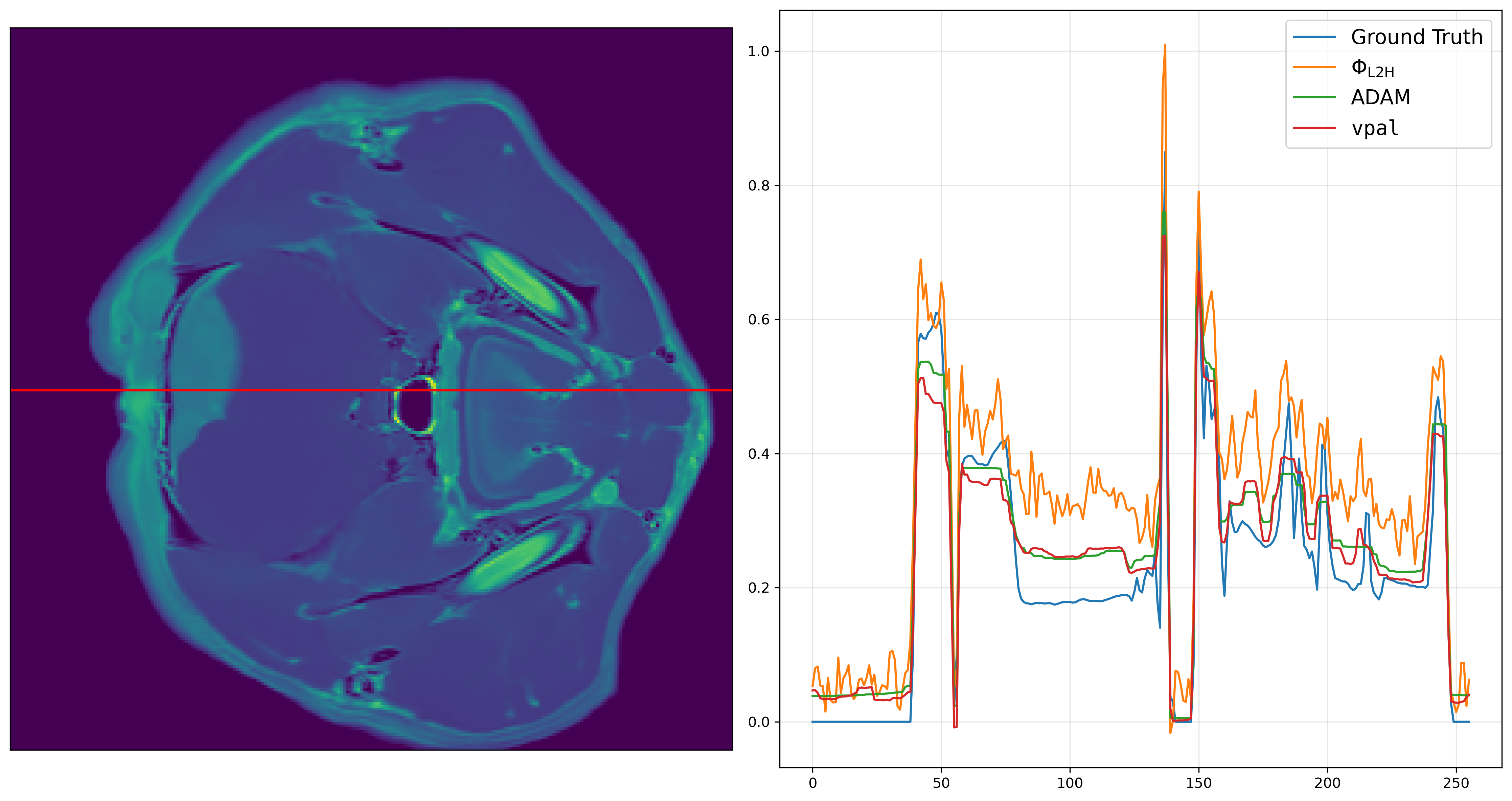}
		\caption{Signal profiles along the central horizontal line.}
	\end{subfigure}
	
	\caption{Experiment 5: Errors for the same $\delta=0.2$ case. Top row: error maps for the network prediction and the Adam and \texttt{vpal} reconstructions; maxima are reported below each map. Bottom row: intensity profiles of the high-dose ground truth and reconstructions along the central horizontal line (red line in the left panel).}
	\label{fig:LIP-CAR_error_maps}
\end{figure}

\Cref{tab:vpal_speed_summary} and \Cref{fig:all_optimizers_comparison} present the results obtained by applying \Cref{alg:LIP-CAR} to all 96 test image cases at a fixed noise intensity level $\delta=0.2$ with each of the four solvers. Again, all solvers outperform the neural-network reconstructions and achieve comparable performance, while \texttt{vpal} is consistently the fastest.

\noindent\begin{minipage}{\linewidth}
	\captionsetup{type=figure}
	\centering
	\begin{minipage}{\linewidth}
		\captionsetup{type=table}
		\small
		\centering
		\begin{tabular}{l c c c}
			\toprule
			\textbf{Solver} & \textbf{Solver faster} & \textbf{\texttt{vpal} faster} & \textbf{Avg.\ \texttt{vpal} speedup} \\
			\midrule
			Adam     & 2 (2.1\%)   & 94 (97.9\%) & 1.43$\times$ \\
			SGD      & 1 (1.1\%)   & 95 (98.9\%) & 2.63$\times$ \\
			RMSprop  & 28 (29.2\%) & 68 (70.8\%) & 1.10$\times$ \\
			\bottomrule
		\end{tabular}
		\caption{Experiment 5: Summary of speed comparisons of \Cref{alg:LIP-CAR} 
			using \texttt{vpal} and the other solvers (96 cases per solver). For each case, we added independent Rician noise of intensity $\delta=0.2$ to the inputs $x_{\textnormal{pre}}$ and $x_{\textnormal{low}}$ and fixed the regularization parameters to $\alpha_1 = 0.05$ and $\alpha_2 = 0.02$. Each row reports the number (and percentage) of cases in which the given solver was faster than \texttt{vpal}, the number of cases in which \texttt{vpal} was faster, and the average speedup of \texttt{vpal} relative to the given solver.}
		\label{tab:vpal_speed_summary}
	\end{minipage}
	\par\medskip
	\nextfloat
	\begin{subfigure}[b]{0.43\linewidth}
		\includegraphics[width=\linewidth]{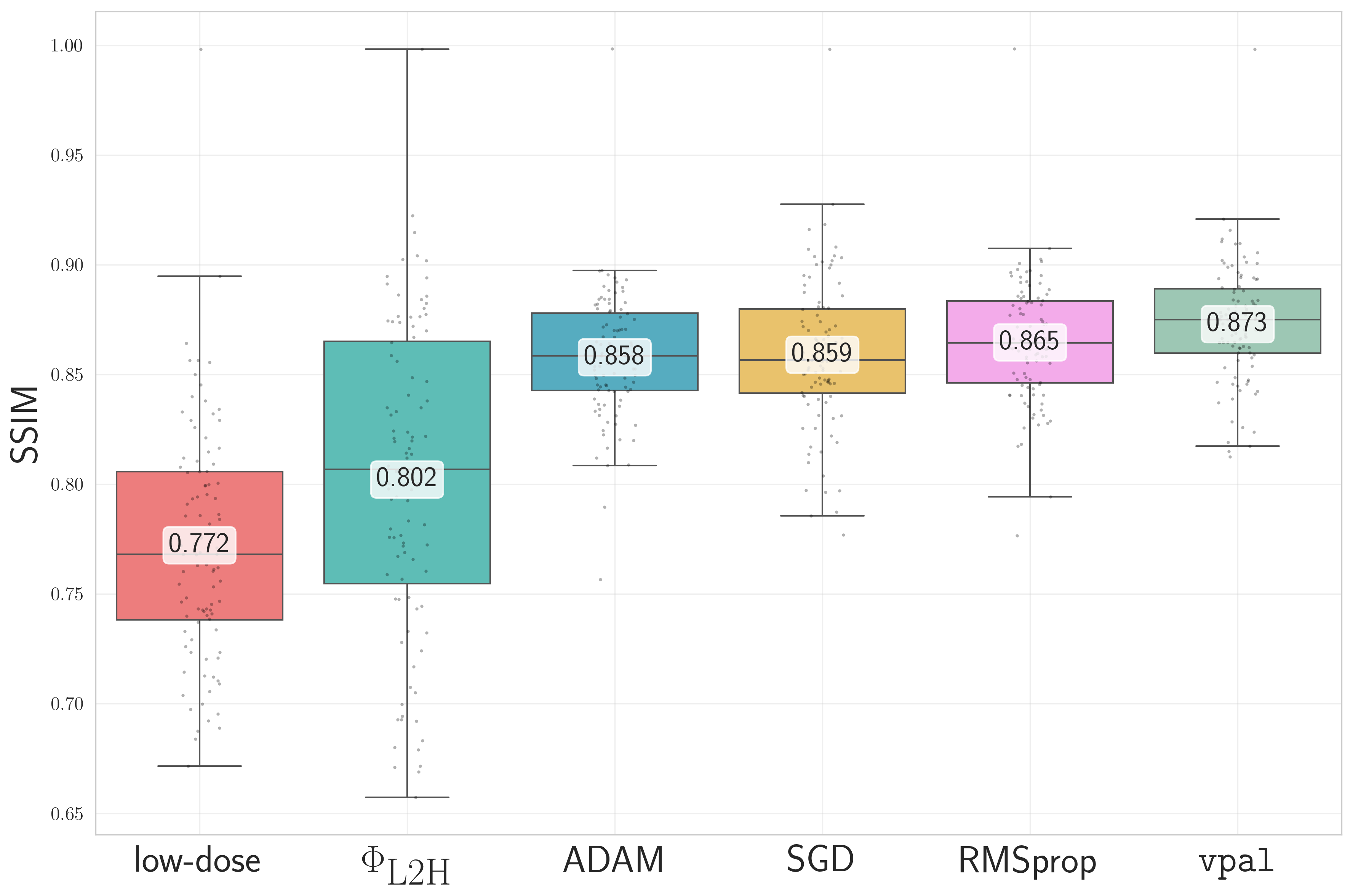}
		\caption{SSIM distributions}
	\end{subfigure}
	\begin{subfigure}[b]{0.43\linewidth}
		\includegraphics[width=\linewidth]{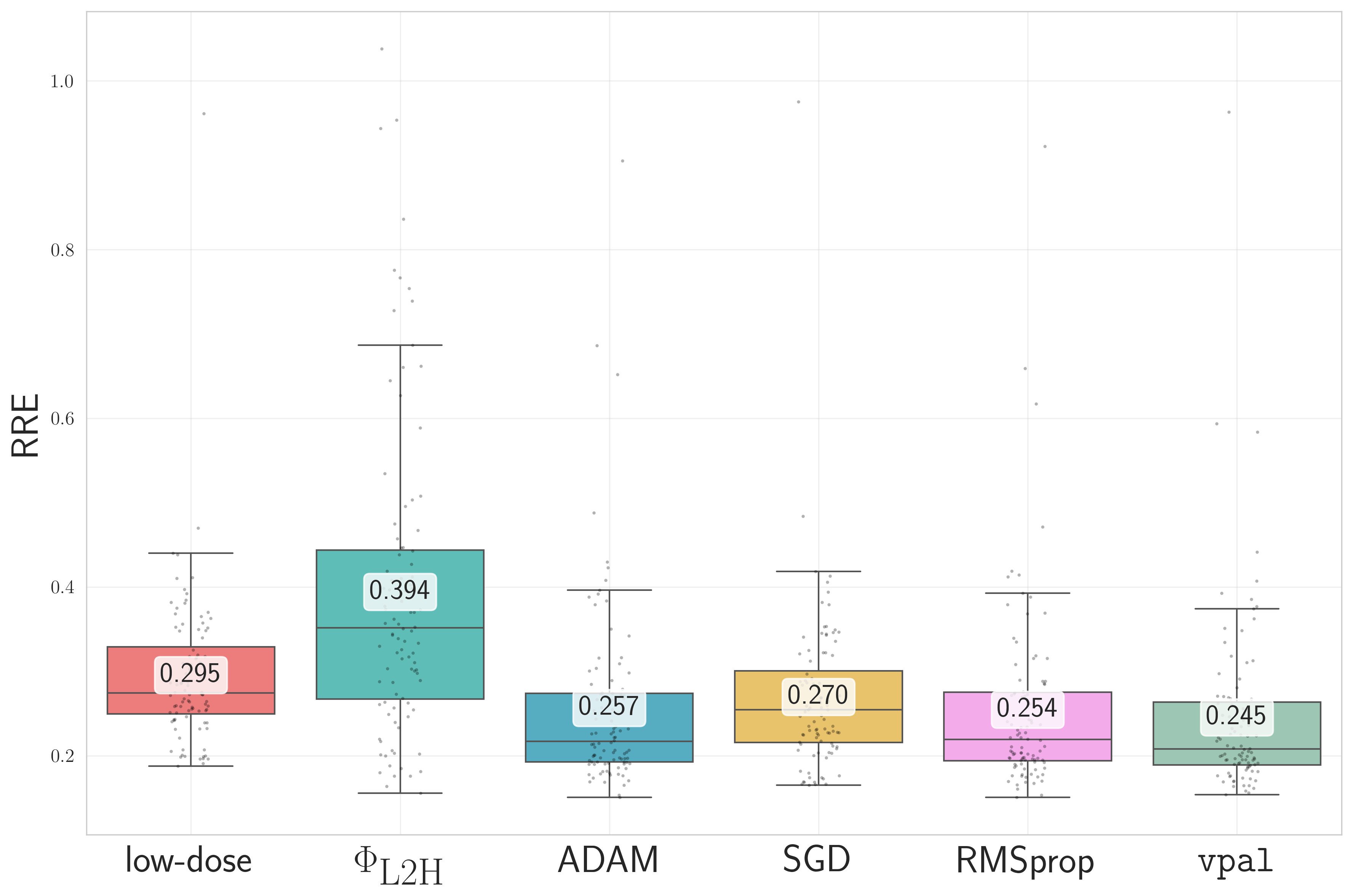}
		\caption{RRE distributions}
	\end{subfigure}
	\caption{Experiment 5: Boxplots of SSIM (a) and RRE (b) obtained using \Cref{alg:LIP-CAR} with \texttt{vpal} and the other solvers over the same 96 cases per solver and with the same noise and regularization parameters as in \Cref{tab:vpal_speed_summary}.}
	\label{fig:all_optimizers_comparison}
\end{minipage}\par

\endgroup

\section{Conclusion}\label{sec:conclusion}

We have introduced a generalized variable projection framework for solving inverse problems with nonlinear and nonsmooth structure, extending the classical variable projection method to encompass broader settings, including $\ell_1$-penalized formulations. Our approach relies on alternating minimization and projected gradient descent in the reduced space, with convergence guarantees under mild assumptions.

To enhance convergence in practice, we proposed a preconditioned variant, \texttt{pvpal}, motivated by inexact Newton-type updates. This variant leverages curvature information from the reduced objective to accelerate descent without requiring full second-order computations.

Extensive numerical experiments across both synthetic and real-world nonlinear problems, including ptychographic phase retrieval and learned inverse operators, demonstrate the versatility and superior performance of our approach. Our methods yield improved convergence behavior and reconstructions even in high-noise or ill-posed regimes.

This work opens avenues for further investigation, including the design of adaptive preconditioners, theoretical guarantees in the nonconvex setting, randomized extensions to {\tt vpal} and {\tt pvpal},  automatic hyperparameter tuning, and applications to large-scale scientific machine learning problems involving latent-variable structure and physics-informed models.

\section*{Acknowledgements}
D. Bianchi is supported by the Startup Fund of Sun Yat-sen University. The authors would like to thank Bracco Imaging S.p.A. for providing access to the dataset used in the LIP-CAR experiment (\Cref{sec:LIPCAR}) and for granting permission to publish the results.

\appendix

\section{Computed tomography experiment}

We consider a two-dimensional computed tomography (CT)
reconstruction problem generated using the IR Tools MATLAB
package~\cite{Gazzola}. The ground-truth image $x_{\mathrm{true}}$ is the
Shepp--Logan phantom discretized on a $128\times128$ pixel
grid. The acquisition consists of $60$ projection angles
in the interval $[0,180]$, with $p=121$ rays per projection.
Thus, the resulting sinogram has dimensions $121\times60$.

We corrupt the noise-free sinogram with $5\%$ relative
additive Gaussian noise. Specifically, the noisy
observations are generated as
\begin{equation}
	b = Ax_{\mathrm{true}}
	+ \eta \frac{\|Ax_{\mathrm{true}}\|_2}
	{\|\varepsilon\|_2}\,\varepsilon,
	\qquad \eta=0.05,
\end{equation}
where $\varepsilon$ is a realization of white Gaussian
noise. \Cref{fig:CT setting} shows the ground-truth phantom
and the corresponding noisy sinogram.

The regularization operator $D$ is chosen as the
first-order forward finite-difference operator along
both spatial directions. We employ non-periodic
boundary conditions, including only differences
between adjacent pixels within the computational domain.

To compare the performance of the {\tt{vpal}} method
and its preconditioned variant {\tt{pvpal}}, we set
the regularization parameter to $\mu=10$ and the
augmented Lagrangian penalty parameter to $\lambda=5$
for both methods. Each method is run for $400$ iterations.

\Cref{fig:CT RRE and time} reports the Relative Reconstruction Error (RRE) obtained by each method at every iteration, as well as its progression over time (in seconds), for both the linearized and optimal strategies. Again, the use of preconditioner reduce the number of iterations and the overall time to compute the same final reconstruction achieved by the standard approach.

The overall performance of the two methods is summarized in \Cref{tab:CT comparison} for both step-size strategies. Although the preconditioned method has a higher computational cost per iteration, it requires substantially fewer iterations to achieve a reconstruction error comparable to that obtained by the standard method after 400 iterations. For the optimal step-size strategy, {\tt pvpal} achieves an RRE of 0.1971 after 12 iterations, compared with an RRE of 0.1973 obtained by {\tt vpal} after 400 iterations, reducing the computational time from 1.4644 s to 0.3888 s. Similarly, for the linearized strategy, {\tt pvpal} achieves an RRE of 0.1965 after 13 iterations, compared with 0.1973 for {\tt vpal} after 400 iterations, reducing the computational time from 0.8447 s to 0.3306 s.
These results demonstrate that preconditioning can substantially reduce the computational effort required to reach a prescribed reconstruction accuracy, despite its higher per-iteration cost. The iteration counts corresponding to comparable reconstruction errors are identified retrospectively using the known ground-truth phantom and are intended for benchmarking rather than as practical stopping criteria.

Lastly, \Cref{fig:CT step size} illustrates the step-size
sequences obtained with the linearized and optimal
strategies. For the standard \texttt{vpal} method,
the step sizes remain nearly constant, exhibiting only
small fluctuations of order $10^{-4}$.
In contrast, the preconditioned \texttt{pvpal} method
exhibits pronounced oscillations when using the optimal
step-size strategy, unlike the more stable behavior
observed in the other linear experiments.

\begin{figure}[htbp]
	\centering
	\begin{subfigure}[b]{0.4\textwidth}
		\centering
		\includegraphics[width=5cm, height=5cm]{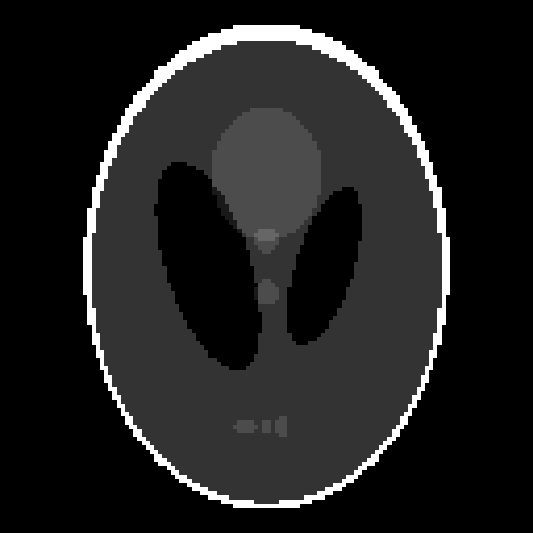}
		\caption{True image}
	\end{subfigure}
	\hfill
	\begin{subfigure}[b]{0.4\textwidth}
		\centering
		\includegraphics[width=5cm, height=5cm]{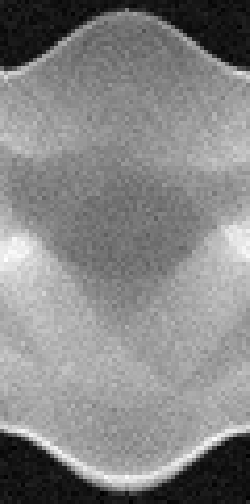}
		\caption{Observed sinogram}
	\end{subfigure}
	\caption{(a) Original image and (b) resulting noisy sinogram.}
	\label{fig:CT setting}
\end{figure}

\begin{figure}[htbp]
	\centering
	\begin{subfigure}[b]{0.42\textwidth}
		\includegraphics[width=\textwidth]{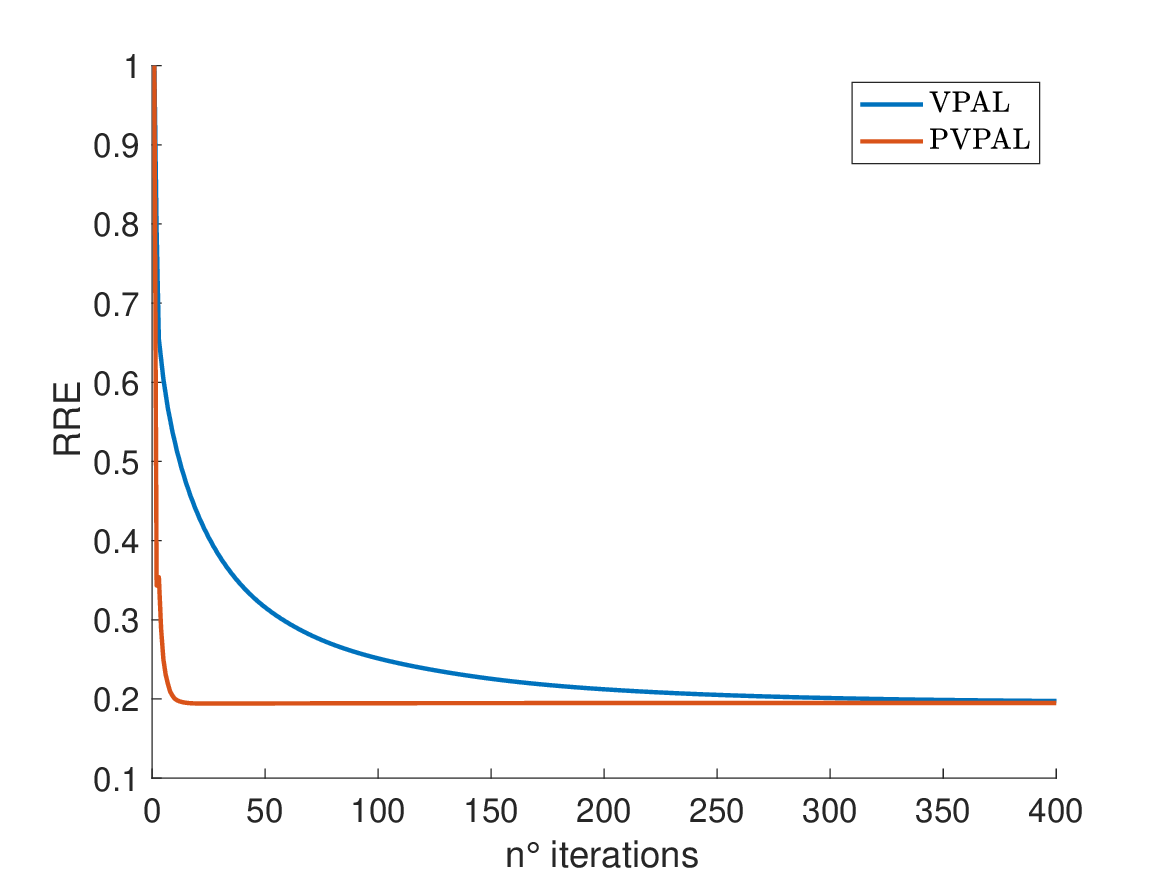}
		\caption{lin: RRE vs Iterations}
	\end{subfigure}
	\begin{subfigure}[b]{0.42\textwidth}
		\includegraphics[width=\textwidth]{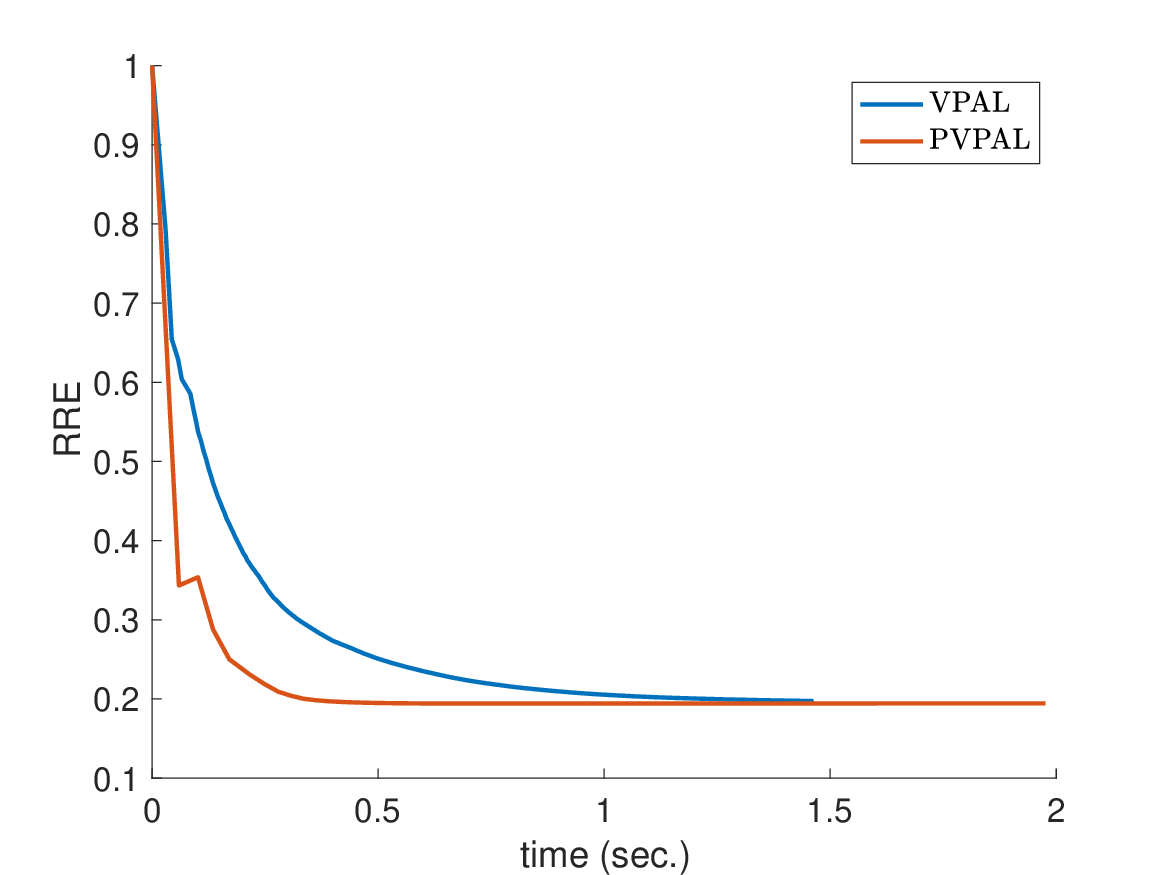}
		\caption{lin: RRE vs Time}
		
	\end{subfigure}
	\\
	\begin{subfigure}[b]{0.42\textwidth}
		\includegraphics[width=\textwidth]{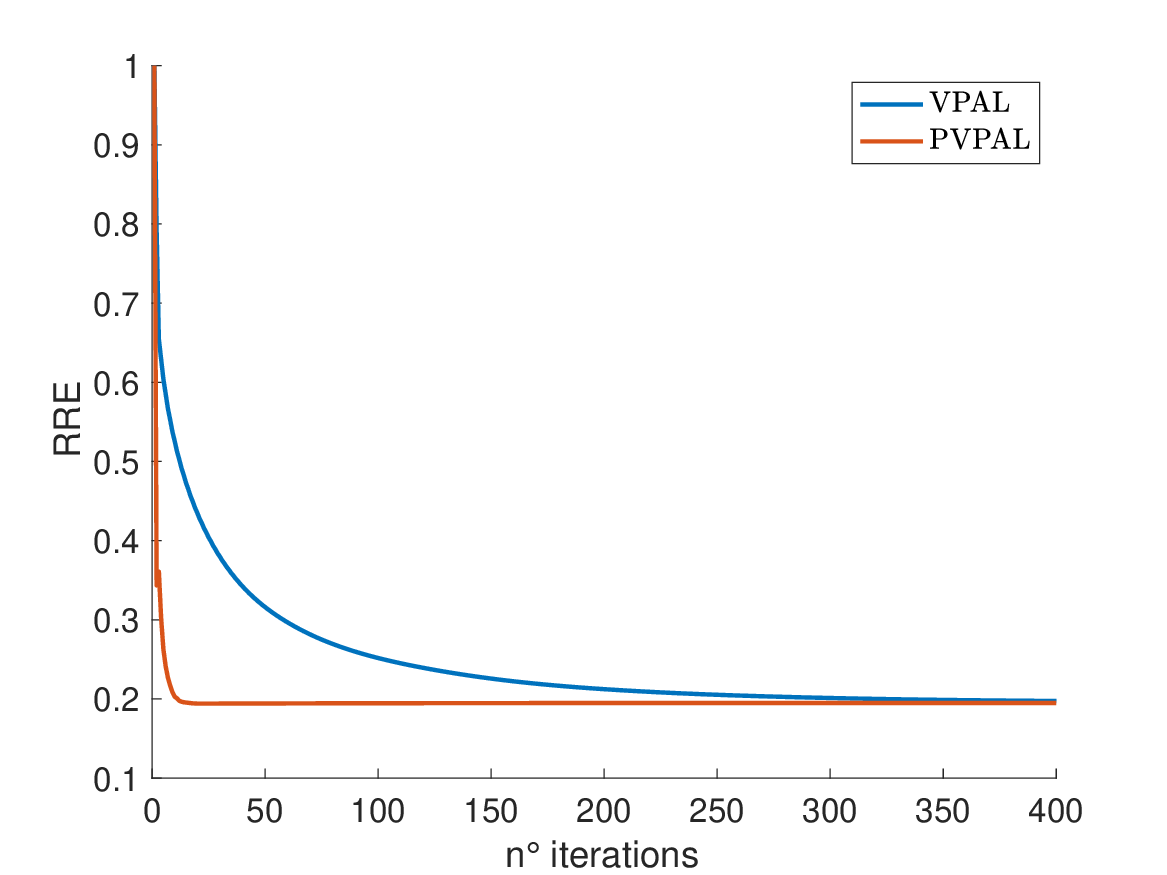}
		\caption{opt: RRE vs Iterations}
	\end{subfigure}
	\begin{subfigure}[b]{0.42\textwidth}
		\includegraphics[width=\textwidth]{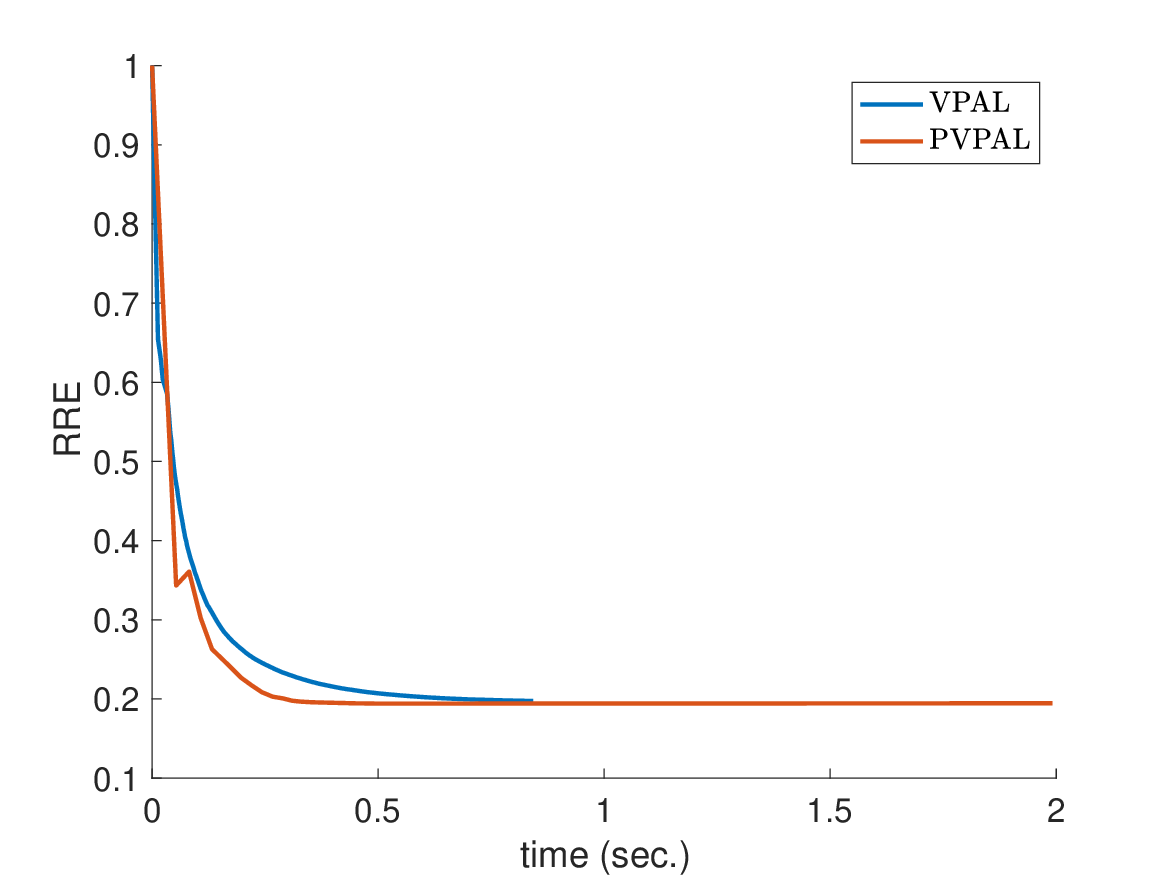}
		\caption{opt: RRE vs Time}
		
	\end{subfigure}
	\caption{Comparison between {\tt vpal} and {\tt pvpal} methods. The first row shows the behavior of the RRE with respect to the number of iterations and the corresponding computation time (in seconds) using the optimal step size strategy. In contrast, the second row focuses on the linearized strategy.}
	\label{fig:CT RRE and time}
\end{figure}

\begin{table}[tbhp]
	\centering \small
	\begin{tabular}{c|c|c|c|c|c|}
		\cline{2-6}
		& & {\tt vpal} (400 iter) & {\tt pvpal} (400 iter) & {\tt pvpal} (49 iter) & {\tt pvpal} (12 iter) \\ 
		\hline 
		\multirow{2}{*}{\rotatebox[origin=c]{90}{\textbf{opt}}} 
		& \textbf{RRE} & \textcolor{matlab1}{0.1973} & 0.1950 & 0.1943 & \textcolor{matlab1}{0.1971} \\ 
		\cline{2-6} & \textbf{Time (sec.)}
		& \textcolor{matlab2}{1.4644} & 11.257& \textcolor{matlab2}{1.4503} & 0.3888 \\
		\hline
		& & {\tt vpal} (400 iter) & {\tt pvpal} (400 iter) & {\tt pvpal} (35 iter) & {\tt pvpal} (13 iter) \\ \hline
		\multirow{2}{*}{\rotatebox[origin=c]{90}{\textbf{lin}}} 
		& \textbf{RRE} & \textcolor{matlab1}{0.1973} & 0.1950 & 0.1942 & \textcolor{matlab1}{0.1965} \\ 
		\cline{2-6} & \textbf{Time (sec.)}
		& \textcolor{matlab2}{0.8447} & 9.5043 & \textcolor{matlab2}{0.8423} & 0.3306 \\
		\hline
	\end{tabular}
	\caption{Comparison between the standard {\tt vpal} method and its preconditioned version, {\tt pvpal}. The top half of the table reports results obtained using the optimal strategy to compute the step size at each iteration, while the bottom half refers to the linearized approach. Values highlighted in \textcolor{matlab1}{blue} indicate the final RRE achieved by {\tt vpal} after 400 iterations. We also report the number of {\tt pvpal} iterations needed to reach the same RRE and the corresponding computation time. Values in \textcolor{matlab2}{red} show the time required by {\tt vpal} to complete 400 iterations, and the performance of {\tt pvpal} within approximately the same amount of time.}\label{tab:CT comparison}
\end{table}

\begin{figure}[b]
	
	\begin{center}
		\includegraphics[width=0.8\textwidth]{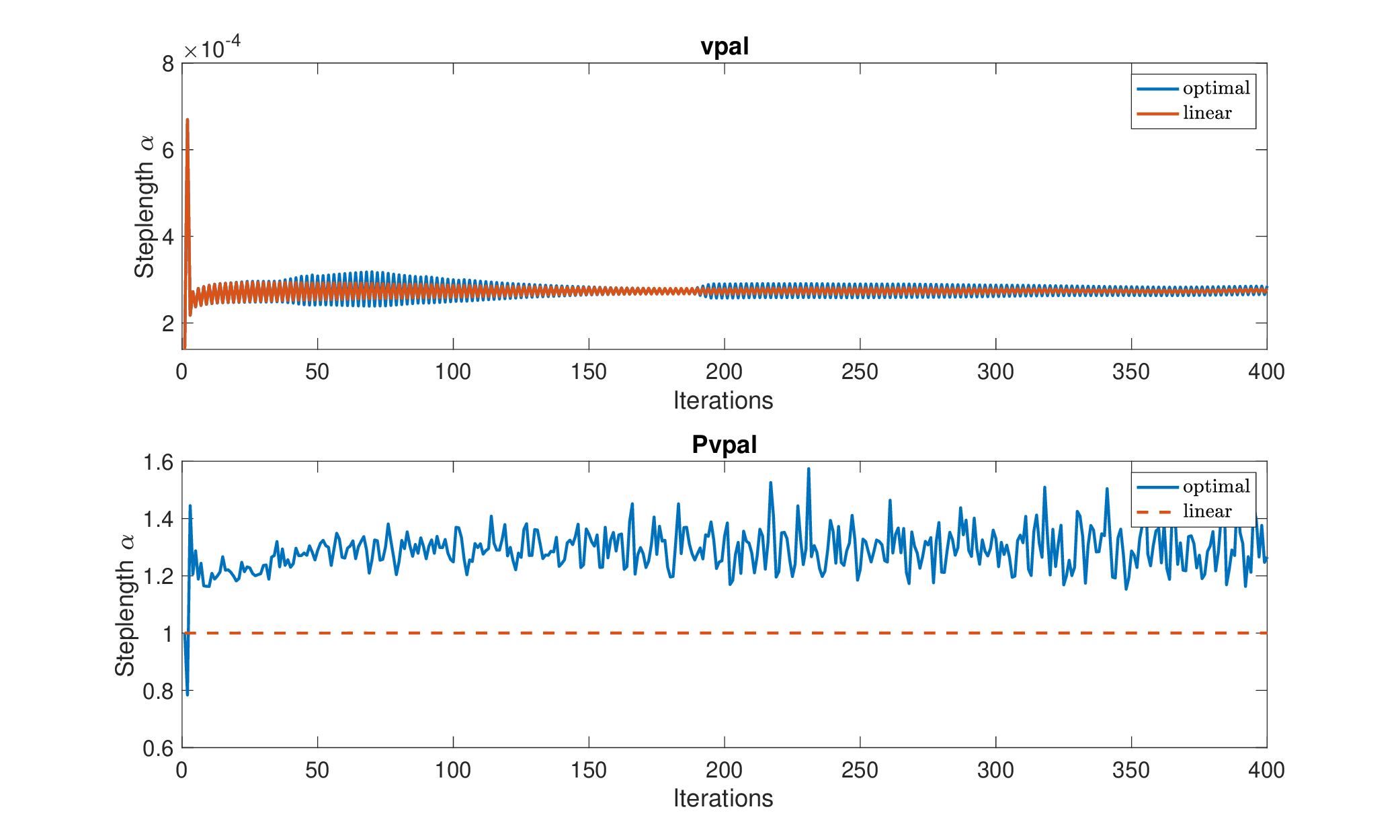}
	\end{center}
	\caption{Step size behavior of {\tt vpal} and {\tt pvpal} for both the linearized and optimal strategies.}
	\label{fig:CT step size}
\end{figure}

\bibliographystyle{unsrt}
\bibliography{main.bib}   

\end{document}